\documentclass[a4paper]{amsart}
\pdfoutput=1
\usepackage{xcolor}
\usepackage{booktabs}
\usepackage{hyperref}
\hypersetup{colorlinks=true,
	linkcolor=blue!50!black,
	citecolor=blue!50!black,
	urlcolor=blue!50!black,
	filecolor=blue!50!black,
	pdfborder={0 0 0},
	pdftitle={Real quadratic fields and finite quantum dilogarithms I},
	pdfauthor={Danylo Radchenko and Campbell Wheeler}
}

\usepackage[a4paper,margin=0.9in]{geometry}
\usepackage{amssymb,amsfonts,amsmath,amsthm}
\usepackage{latexsym,mathtools,stmaryrd}
\usepackage{tikz}

\newcommand*\smat[4]{\begin{smallmatrix}#1&#2\\#3&#4\end{smallmatrix}}
\newcommand*\pmat[4]{\begin{pmatrix}#1&#2\\#3&#4\end{pmatrix}}

\newcommand{\cycle}[1]{\left(\mkern-4mu\left(#1\right)\mkern-4mu\right)}

\newcommand{\Mod}[1]{\ (\text{mod}\ #1)}
\newcommand{\ol}[1]{\overline{#1}}

\newcommand*\ang[1]{\langle#1\rangle}

\newcommand{\sm}{\smallsetminus}
\newcommand{\ds}{\displaystyle}

\newcommand{\CC}{\mathbb{C}}
\newcommand{\HH}{\mathbb{H}}

\newcommand{\QQ}{\mathbb{Q}}
\newcommand{\RR}{\mathbb{R}}

\newcommand{\ZZ}{\mathbb{Z}}

\newcommand{\calC}{\mathcal{C}}
\newcommand{\calD}{\mathcal{D}}

\newcommand{\Li}{\mathrm{Li}}
\newcommand{\Gal}{\operatorname{Gal}}

\newcommand{\tr}{\operatorname{tr}}
\newcommand{\re}{\operatorname{Re}}

\newcommand{\eps}{\varepsilon}
\renewcommand{\phi}{\varphi}
\newcommand{\sgn}{\mathrm{sgn}}

\newcommand{\fC}{\mathfrak{C}}

\newcommand{\fa}{\mathfrak{a}}
\newcommand{\fb}{\mathfrak{b}}
\newcommand{\fc}{\mathfrak{c}}
\newcommand{\ff}{\mathfrak{f}}
\newcommand{\fp}{\mathfrak{p}}

\newcommand{\fr}{\mathfrak{r}}

\newcommand{\lC}{\mathcal{C}}

\newcommand{\lO}{\mathcal{O}}
\newcommand{\lS}{\mathcal{S}}

\newcommand{\SL}{\operatorname{SL}}

\newcommand{\Cl}{\operatorname{Cl}}
\newcommand{\Hom}{\operatorname{Hom}}

\newcommand{\bk}{\mathbf{k}}
\newcommand{\bK}{\mathbf{K}}
\newcommand{\bp}{\mathbf{p}}
\newcommand{\bq}{\mathbf{q}}

\newcommand{\bF}{\mathbf{F}}

\newcommand{\ee}{\mathbf{e}}

\newcommand{\sfb}{\mathsf{b}}
\newcommand{\tq}{\widetilde{q}}
\def\Res#1{\underset{#1}{\mathrm{Res}}}

\title{Real quadratic fields and finite quantum dilogarithms I}
\author[Radchenko, Wheeler]{Danylo Radchenko and Campbell Wheeler}
\date{\today}
\subjclass[2010]{}
\keywords{}
\address{Institut des Hautes \'Etudes Scientifiques, CNRS, Laboratoire Alexander Grothendieck,
	35 route de Chartres, Bures-sur-Yvette 91440, France}
\email{danradchenko@gmail.com}
\address{Coburg 3058, Australia}
\email{wheeler@ihes.fr}

\newtheorem{theorem}{Theorem}
\newtheorem{proposition}{Proposition}
\newtheorem{lemma}{Lemma}

\theoremstyle{definition}
\newtheorem{definition}{Definition}
\newtheorem{remark}{Remark}
\newtheorem{corollary}{Corollary}

\usepackage{enumitem}
\newlist{thmenum}{enumerate}{1}
\setlist[thmenum]{
	label=\upshape(\roman*),       
	ref=\thetheorem.(\roman*),     
	leftmargin=*,                 
	itemsep=1.5pt,
	topsep=2pt,
	partopsep=0pt,
	parsep=0pt
}

\begin{document}
    \begin{abstract}
        We prove that Stark--Shintani ray class invariants (Stark units) associated to real quadratic fields are algebraic numbers. These invariants are given by special values of Faddeev's modular quantum dilogarithm, introduced by Garoufalidis--Kashaev--Zagier.
        Our main discovery is that special values of the modular quantum dilogarithm satisfy an explicit overdetermined system of polynomial equations, matching a variation on the defining equations of Andersen--Kashaev's notion of a quantum dilogarithm on a product of two cyclic groups.
        We give two and a half proofs that this system of equations defines a zero-dimensional variety.
        The simplest follow from an uncertainty principle for finite Fourier transform and $2$-adic valuation bounds. The last proof is more involved and shows finite quanatum dilogarithms can be used to categorify fusion rings introduced by Izumi, and the algebraicity of the special values then follows by Ocneanu's rigidity theorem. As a byproduct, we obtain an explicit infinite family of irrational near-group fusion categories. As a further application, we prove a family of quadratic relations for Stark units recently conjectured by Appleby, Flammia, and Kopp motivated by Zauner's conjecture about SIC-POVMs (complex equiangular lines).
    \end{abstract}
	\maketitle
    \begin{center}
    \emph{To Don Zagier, on the occasion of his 75th birthday.}
    \end{center}
    \vspace{0.2cm}
    \tableofcontents

\section{Introduction and main results}
The Kronecker--Weber theorem and theory of complex multiplication are some of the most beautiful discoveries of 19th century mathematics, forming a major part of the as yet unrealised dream of explicit class field theory: to find a description of all abelian extensions of number fields via special values of transcendental functions.
At the most basic level, a consequence of this idea is that certain special transcendental functions should take algebraic values at countably many algebraic arguments.
Indeed, Euler's formula implies that the function $\ee(x)=\exp(2\pi i x)$ always takes algebraic values for $x\in\QQ$. For example, one finds that
    \[
    \ee\Big(\frac{1}{5}\Big)
    =\frac{\sqrt{5}-1}{4}+i\sqrt{\frac{5+\sqrt{5}}{8}}\,.
    \]
Even more beautifully, the modular $j$-function, defined on the upper half-plane~$\HH=\{\tau\in\CC\colon \Im\tau>0\}$ by 
    \[j(\tau) = \frac{(1+240\sum_{m,n\ge1}n^3q^{mn})^3}{q\prod_{n=1}^{\infty}(1-q^n)^{24}} = q^{-1} + 744 + 196884q + \cdots\,,\]
where $q=\ee(\tau)$, takes algebraic values for all $\tau\in\HH$ satisfying $a\tau^2+b\tau+c=0$ for some $a,b,c\in\QQ$. An example, due to Weber, states that
    \[
    j(\sqrt{-14}) = \bigg(646 + 456\sqrt{2} + (462 + 322\sqrt{2})\sqrt{2\sqrt{2} - 1}\bigg)^3\,.
    \]

Besides Shimura--Taniyama's theory of complex multiplication for higher-dimensional abelian varieties~\cite{ST61}, it has long been unclear whether such a theory could exist for arbitrary number fields, and it is the subject of Hilbert's 12th problem~\cite{Hil00}. The next simplest case is that of a real quadratic field, and it already has significant number-theoretic differences with the rational and imaginary quadratic cases. Indeed, Dirichlet's unit theorem states that the latter two are the only cases with a torsion unit group.

A proposal for an answer to Hilbert's 12th problem, for totally real fields, was introduced by Stark~\cite{Sta76,Sta80} and Shintani~\cite{Shi77,Shi78}, in terms of leading terms at $s=0$ of Artin $L$-functions. Rationality of these leading terms, up to a regulator, are Beilinson's conjectures for the special case of Artin motives. In the so-called abelian rank one case, leading terms of related partial $\zeta$-functions are predicted to be logarithms of units, and Stark's conjectures describe the Galois action on them explicitly (see~\cite[Ch.~I,~IV]{Tate84}). A $p$-adic analogue of these predictions, the Gross–Stark conjecture, was proved by Dasgupta--Kakde~\cite{DK24} (see also Darmon--Vonk~\cite{DV21}); it gives explicit generators of the maximal abelian extension of a totally real field, but via $p$-adic rather than complex analytic means.

In the case of real quadratic fields, unravelling Shintani's work (see for example~\cite{Yam10b}, ~\cite{Yal26},~\cite{Kop24}) shows that Stark units are explicitly related to special values of Faddeev's quantum dilogarithm (or the double sine function~\cite{KK03}, a close relative of Barnes's double gamma function~\cite{Bar04}). Shintani was able to prove parts of Stark's conjectures in some very special examples of real quadratic fields.
In this paper, we will give a complete treatment of the algebraicity part of Stark's conjectures for all real quadratic fields. For example, our work leads to provable identities such as
    \[
    \exp\Big(\int_{\RR+i0}
        \frac{e^{-2izw}}{4\sinh(\sfb w)\sinh(\sfb^{-1} w)}\frac{dw}{w}\Big) = \frac{(\sqrt{3}+i)\sqrt{1+\sqrt{21}-\sqrt{6+2\sqrt{21}}}}{4}\,, 
    \]
where $\sfb = \sqrt{\frac{5+\sqrt{21}}{2}}$ and $z=i\sqrt{\frac{1}{12}}$.

\bigskip

We will now describe our results in detail. Consider the meromorphic functions $\Phi_{\gamma,m,n}(z;\tau)$ (introduced in~\cite[Eq. 11]{GKZ}) defined for $\gamma\in\SL_2(\ZZ)$, $m,n\in\ZZ$, and $(z,\tau)\in\CC\times(\CC\sm\RR)$ by
    \begin{equation}\label{eq:phigam.def}
    \Phi_{\gamma,m,n}(z;\tau)=\frac{(q^m\ee(z);\ee(\tau))_{\infty}}{(\tilde q^n\ee(z/(c\tau+d));\ee(\gamma\tau))_{\infty}}\,,\qquad \tilde q = \ee\Big(\frac{a\tau+b}{c\tau+d}\Big)\,,
    \end{equation}
where $(x;q)_{\infty}$ is the $q$-Pochhammer symbol (see~\S\ref{sec:qpochhammer}). We call this function \emph{Faddeev's modular quantum dilogarithm}. It follows from~\cite{Fad95}, that this function has a meromorphic continuation to all $z$ and $\tau$ such that $c\tau+d\notin\RR_{\leq0}$.
The product expansion also makes the location of the zeros and poles on explicit subsets of $\ZZ+\tau\ZZ$ clear (one can see that generic $c\tau+d\in\RR_{\leq 0}$ would lead to an accumulation of poles and zeros as a function in~$z$). We are interested in certain special values of this function, which we describe using the following notation: fix a hyperbolic $\gamma=(\smat abcd)\in\SL_2(\ZZ)$ with $c>0$ and trace $N+2$ for $N>0$, and let~$\tau$ be a fixed point of~$\gamma$ (this implies that $\eps=c\tau+d$ satisfies $\eps^2-(N+2)\eps+1=0$); for $u\in\ZZ^2$ if $u\gamma \equiv u \Mod{N}$ then denote $F_{\gamma}^{\pm}(u) := \eps^{\pm1/2}$ and otherwise let
\begin{equation}\label{eq:fgam.def}
    F_{\gamma}^{\pm}(u) := 
	\Phi_{\gamma,u_1,0}\Big(\frac{u_1\tau+u_2}{\eps^{-1}-1};\tau\Big)\,\mu_{\gamma}\,,
\end{equation}
where $\mu_\gamma$ is the multiplier system of the Dedekind eta function~$\eta$, namely $\mu_\gamma=\frac{\eta(\gamma z)}{\eta(z)\sqrt{cz+d}}$ for any $z\in\HH$.

We remark that if the formula in Equation~\eqref{eq:phigam.def} made sense when $\tau\in\RR$, then one would expect that $F_{\gamma}^{\pm}(u)\mu_{\gamma}^{-1}=1$, since the associated arguments satisfy $q^m\ee(z)=\tq^n\ee(z/\eps)$ and $q=\tq$. However, these values are not unity and are in fact related to Stark units.
Indeed, for a ray class $\fc\in \Cl_{\ff\infty_2}(K)$ of a real quadratic field $K$, the Stark unit\footnote{We also sometimes call it the Stark--Shintani ray class invariant. In this paper by a Stark unit we mean the real number~$\epsilon_{\fc}$, and not the conjectural unit in the corresponding ray class field.} can be defined as
    \[\epsilon_{\fc} = \exp(\zeta'(0,\fc)-\zeta'(0,\fr\fc))\,,\]
where $\fr$ is the class of $(t)$ for any $t\equiv 1\Mod{\ff}$ with $\sigma_2(t)<0$. Then there is a formula relating this number to the above special values: 
    \[\epsilon_{\fc} = |\Phi_{\gamma^{r},m,0}(z_0;\tau_0)|^{-t} 
    = |F^{+}_{\gamma^r}(u)|^{-t}\,,\]
where $t$ is either $1$ or $2$ (for details on what all the quantities above are, see~\S\ref{sec:starkshintani}). The algebraicity part of Stark's conjecture is then subsumed in the following:
\begin{theorem} \label{thm:faddeevqbar}
The numbers $F_{\gamma}^{\pm}(u)$ are algebraic, i.e., for all $u\in\ZZ^2$,
    \[
    F_{\gamma}^{\pm}(u)
    \in\overline{\QQ}\,.
    \]
\end{theorem}

The proof of Theorem~\ref{thm:faddeevqbar} relies on our main result, which is much more surprising: we will exhibit an explicit (and immensely overdetermined) system of polynomial equations that these numbers satisfy. This system is defined over $\ZZ[\zeta_{2N},1/\sqrt{N}]$, where $\zeta_{2N}$ is a primitive $2N$-th root of unity. To give these explicit equations, we need to introduce more notation. For $u,v\in\ZZ^2$ define
	\begin{equation} \label{eq:gaussianpairing}
    \ang{u}_{\gamma} := (-1)^{u_1+u_2+u_1u_2}\ee_{2N}(\omega(u\gamma,u))\,,\qquad\text{and}\qquad \ang{u;v}_{\gamma} = \ee_N(\omega(u\gamma-u,v))\,,
    \end{equation}
where $\omega(v,w)=v_1w_2-v_2w_1$ and $\ee_N(x)=e^{2\pi i x/N}$.
If we let $\Lambda=\Lambda_{\gamma}=\{u\in\ZZ^2\colon u\gamma\equiv u\Mod{N}\}$ (that is also equal to $\ZZ^2(\gamma-I)$), then $\ang{u}_{\gamma}$ is $\Lambda$-periodic. Using the quasi-periodicity of $\Phi_{\gamma,m,n}$ in the $(\ZZ+\tau\ZZ)$-lattice, one can show that $u\mapsto F_{\gamma}^{\pm}(u)$ is also $\Lambda$-periodic (see Lemma~\ref{lem:lam.inv}). Letting $G=\ZZ^2/\Lambda$, we see that $\ang{u}_{\gamma}$ and $F_{\gamma}^{\pm}(u)$ both give well-defined functions on~$G$. This quotient has size $|G|=N$ and hence we have $N$ values of $F_\gamma^\pm$.\footnote{It is also worth reiterating that $F_{\gamma}^{+}(u)=F_{\gamma}^{-}(u)$ unless $u=0$ in $G$.}

\begin{theorem}\label{thm:fg.equs}
The numbers $F_\gamma^\pm(u)$ satisfy the following system of algebraic equations:
\begin{align}
    F_{\gamma}^{\pm}(0) &= \sqrt{\frac{N+2\pm\sqrt{N^2+4N}}{2}} \,,\label{eq:Fgpm.zero}\\
	F_{\gamma}^{+}(u)F_{\gamma}^{-}(-u) &= \ang{u}_{\gamma}^{-1} \,,\label{eq:Fgpm.reflection}\\[5pt]
	\frac{1}{\sqrt{N}}\sum_{x\in G}\ang{x;u}_{\gamma}F_{\gamma}^{+}(x) & =\mu_{\gamma}\frac{1}{F_{\gamma}^{-}(u)}\,,\label{eq:Fgpm.fourier}\\
	\frac{1}{\sqrt{N}}\sum_{x\in G}\ang{x;u}_{\gamma}\frac{F_{\gamma}^{+}(x)}{F_{\gamma}^{-}(x+v)} &=
    \frac{F_{\gamma}^{+}(u+v)}{F_{\gamma}^{-}(u)F_{\gamma}^{-}(v)}\,,\qquad (u,v)\ne (0,0)\,.\label{eq:Fgpm.5term}
\end{align}
\end{theorem}

These equations are parallel to the functional equations of $\Phi_{\gamma,m,n}(z;\tau)$, but with integrals replaced by sums. One of the corollaries of equations~\eqref{eq:Fgpm.zero}--\eqref{eq:Fgpm.5term} is that the function $E(u)=F_{\gamma}^{+}(u)$ satisfies the cubic equations
    \begin{equation} \label{eq:fivetermintro}
    E(x)E(y)\ang{x;y}_{\gamma} = C+\frac{1}{\sqrt{N}}\sum_{t\in G}E(x-t)E(t)E(y-t)\ang{t}_{\gamma}\,,\qquad x,y\in G\,,
    \end{equation}
for a non-zero constant~$C$. This structure is very close to the one defined by Andersen and Kashaev in~\cite[Def.~9]{AK14b}, where an abstract notion of quantum dilogarithm on locally compact abelian groups is given. The difference is that in~\eqref{eq:fivetermintro} we require that $C\ne 0$, whereas Andersen and Kashaev have~$C=0$. Although we do not discuss this here, for~$C=0$ one can classify all solutions of Equation~\eqref{eq:fivetermintro}, for finite~$G$, and no interesting functions appear: all solutions essentially reduce to indicators of isotropic subgroups. For $C\ne 0$ the situation is opposite, since for all groups of the form $G=\ZZ/n\ZZ\times \ZZ/m\ZZ$ we get non-trivial solutions whose values are constructed from Stark units. In Theorem~\ref{thm:fqdilogbasicproperties}, we also show that for $C\ne 0$ Equation~\eqref{eq:fivetermintro} already implies all of the other identities from Theorem~\ref{thm:fg.equs}. We call any solution to Equation~\eqref{eq:fivetermintro} (with $C\ne0$) a~\emph{finite quantum dilogarithm} on~$G$.

We have several approaches that can be used to prove Theorem~\ref{thm:faddeevqbar}.
When $N=|G|$ is a prime, we give a short elementary proof based on Tao's uncertainty principle. Then we give an elementary $2$-adic proof that also has the benefit of proving these numbers are $2$-adic units. We then prove a more general statement (Theorem~\ref{thm:algebraicity}) that any solution of Equation~\eqref{eq:fivetermintro} (for $|G|>1$ and $C\ne 0$) has values in~$\ol{\QQ}$. 
When $N=|G|$ is odd, the $2$-adic methods also supply an equally elementary proof again with the benefit of proving they are $2$-adic units.
In general, the proof relies on the fact that exactly the same system of equations has been studied by Izumi~\cite{Izu93,Izu00,Izu01} in the context of near-group fusion categories. Specifically, Izumi studied fusion categories whose Grothendieck ring has basis elements $[\alpha_g]$ ($g\in G$) and $[\rho]$ satisfying the fusion rules
	\[[g][h] = [g+h]\,,\qquad [g][\rho] = [\rho][g]=[\rho]\,,
	\qquad [\rho]^2 = N[\rho] + \sum_{g\in G}[g]\,.\]
Izumi's original construction assumes an additional property of $E$ with regards to complex conjugation, but we show that it is not necessary, closely following the algebraic approach of Evans--Gannon~\cite{EG17}. As a result of this construction, we show that fusion categories of this type exist for all $G$ of the form $G=\ZZ/n\ZZ\times \ZZ/m\ZZ$. Similar constructions of a different class of fusion categories, namely of Haagerup--Izumi categories for odd cyclic groups, appeared recently in independent works~\cite{Huang26,Gannon26}.

\medskip
The paper is organised as follows. In~\S\ref{sec:preliminaries}, we collect basic properties of Faddeev's quantum dilogarithm and, following the paper of Garoufalidis--Kashaev--Zagier~\cite{GKZ}, its modular variant. In~\S\ref{sec:mainidentities}, we prove our main result, that is Theorem~\ref{thm:fg.equs}. In~\S\ref{sec:finitedilog}, we define an abstract notion of a finite quantum dilogarithm, give its basic properties, and show that Theorem~\ref{thm:fg.equs} provides us with examples for all groups of the form $\ZZ/n\ZZ\times \ZZ/m\ZZ$. We also prove a quadratic convolution identity for Stark units, conjectured in~\cite{AFK25}, and briefly discuss the connection to Zauner's conjecture. In~\S\ref{sec:algebraicity} we prove algebraicity of values of finite quantum dilogarithms, first using an elementary argument for prime cyclic groups, then using $2$-adic valuations, and then in full generality using Izumi fusion categories. Finally, in~\S\ref{sec:conclusions} we discuss directions and open questions that we plan to investigate in future work.

\subsection*{Acknowledgments}
The authors wish to thank Dustin Clausen, Stavros Garoufalidis, Maxim Kontsevich, Daniil Rudenko, Bora Yalkinoglu, and Don Zagier for enlightening conversations. Of fundamental importance to this paper was the introduction in~\cite{GKZ} of Faddeev's modular quantum dilogarithm and its functional equations. We want to thank the authors for sharing their manuscript with us many years ago, which contained all the necessary definitions and proofs. Any potential errors in exposition of their ideas are due entirely to the authors of the current paper. C.W. was supported by the Institut des Hautes \'{E}tudes Scientifiques, France. D.R. acknowledges funding by the European Union (ERC, FourIntExP, 101078782).

\smallskip
\noindent \textbf{AI use disclosure}
We used LLMs exclusively for the purpose of checking the mathematical calculations in Appendix~\ref{sec:appendixB}, which helped identify several mistakes in an earlier draft. All final calculations were independently verified by the authors, who take full responsibility for the results.

\section{Preliminaries}
\label{sec:preliminaries}
In this section, we will recall some basic properties of Faddeev's quantum dilogarithm~\cite{Fad95}. The main emphasis is that this function captures the failure of modularity of the infinite $q$-Pochhammer symbol, which is ``half of the Jacobi $\theta$-function'' and, while not being a Jacobi form, is a quantum Jacobi form~\cite{GZ23,GZ24}. Then we will discuss its full modular generalisation as the associated multiplicative cocycle~\cite{GKZ}. In~\S\ref{sec:starkshintani}, we will describe the relation with the Stark--Shintani invariants.

Throughout this paper, we will use the following notation: 
    \[\ee(x)=\exp(2\pi i x)\,,\qquad \ee_N(x)=\ee(x/N),\]
$(x;q)_n=\prod_{j=0}^{n-1}(1-q^jx)$ if $n\geq0$ and $(x;q)_n=\prod_{j=n}^{-1}(1-q^jx)^{-1}$ if $n<0$; $S=(\smat 0{-1}{1}0)$ and $T=(\smat 1{1}{0}1)$.

\subsection{Faddeev's quantum dilogarithm}
We will recall here some basic properties of Faddeev's quantum dilogarithm.
(See~\cite[Appendix A]{AK14a} for an excellent summary of its basic properties.)
We emphasise here its role as a \emph{quantum dilogarithm} but note that this and related functions were studied, at least as early as Barnes~\cite{Bar04} and is often also referred to as the double sine function.
It is most basically defined as the analytic continuation of the following integral:
    \[
    \phi_\sfb(z)
    =
    \exp\Big(\int_{\RR+i0}
    \frac{e^{-2izw}}{4\sinh(\sfb w)\sinh(\sfb^{-1} w)}\frac{dw}{w}\Big)\,,
    \]
which converges when $|\Im(z)|<|\Im(\frac{i}{2}(\sfb+\sfb^{-1}))|$.
When $\Im(\sfb^2)\neq 0$, the residue theorem and Equation~\eqref{eq:qpoch-log} imply that this function has the following product expansion:
\[
\begin{aligned}
    \phi_\sfb(z)
    &=
    \frac{(-q^{\frac{1}{2}}e^{2\pi\sfb z};q)_\infty}
    {(-\tq^{\frac{1}{2}}e^{2\pi\sfb^{-1}z};\tq)_\infty}\,,
    &&\quad\text{when}\quad\Im(\sfb^2)>0\\
    \phi_\sfb(z)
    &=
    \frac{(-\tq^{-\frac{1}{2}}e^{2\pi\sfb^{-1}z};\tq^{-1})_\infty}{(-q^{-\frac{1}{2}}e^{2\pi\sfb z};q^{-1})_\infty}\,,&&
    \quad\text{when}\quad\Im(\sfb^2)<0\\
\end{aligned}
\]
where $q=\ee(\sfb^2)$ and $\tq=\ee(-\sfb^{-2})$.
This formula makes it clear that $\phi_\sfb(z)$ is a solution to a $(q,\tq)$-holonomic bimodule.
This can in fact be used to explicitly construct the analytic continuation of $\phi_\sfb(z)$.
This function can then be shown to have a meromorphic continuation to $\Re(\sfb)>0$ and $z\in\CC$. 

In this paper, we are particularly interested in the modular properties of Faddeev's quantum dilogarithm.
It was driven home in~\cite{GZ23,GZ24} that $\phi_\sfb(x)$ appears as the multiplicative cocycle associated to the matrix $S=(\smat 0{-1}{1}0)\in\SL_2(\ZZ)$ of the quantum Jacobi form associated to the infinite Pochhammer symbol $(x;q)_\infty$.
Therefore, we will use a slight change of variable and consider the function for each $m,n\in\ZZ$ given by
\[
    \Phi_{S,m,n}(z;\tau)
    =
    \phi_\sfb(i\sfb^{-1}z+i(m-\tfrac{1}{2})\sfb+i(n+\tfrac{1}{2})\sfb^{-1})
    =
    \Phi_{S,0,0}(z+m\tau-n;\tau)\,,
\]
which has product expansion when $\Im(\tau)>0$ given by
\begin{equation}\label{eq:prod.fadeev}
    \Phi_{S,m,n}(z;\tau)
    =
    \frac{(q^m\ee(z);q)_\infty}
    {(\tq^n\ee(z/\tau);\tq)_\infty}\,.
\end{equation}
This function has poles when $z\in\ZZ_{\geq n}+\tau\ZZ_{>-m}$ and zeros when $z\in\ZZ_{<n}+\tau\ZZ_{\leq -m}$.
The difference equations are simply described by
\[
    \Phi_{S,m,n}(z;\tau)
    =
    (1-q^m\ee(z))\Phi_{S,m+1,n}(z;\tau)
    =
    (1-\tq^{n}\ee(z/\tau))^{-1}\Phi_{S,m,n+1}(z;\tau)\,.
\]
We also have reflection formulas $\overline{\Phi_{S,0,0}(z;\tau)}=\Phi_{S,1,1}(-\overline{z};\overline{\tau})^{-1}$ and $\Phi_{S,0,0}(z/\tau;1/\tau)=\Phi_{S,1,1}(z;\tau)^{-1}$.

We will also need to understand the basic asymptotic behaviour of this function. For $\tau\in\RR_{>0}$, we have $\Phi_{S,m,n}(z;\tau)=\mathrm{O}(1)$ when $z\to i\infty$ and $\Phi_{S,m,n}(z;\tau)=\mathrm{O}(\ee(-(z+(m-1/2)\tau+(n+1/2))^2/2\tau)$ when $z\to -i\infty$.
When $\tau\to 0$ with $z\notin\ZZ+\tau\ZZ$ fixed, we have $\Phi_{S,m,n}(z;\tau)=\mathrm{O}(\ee(\Li_2(\ee(z))/\tau))$, which partially explains why this function is called a quantum dilogarithm.

The main reason this function is called a quantum dilogarithm is that it satisfies analogues of the classical dilogarithmic identities (for an excellent introduction to the topic, we refer to Zagier~\cite{Zag07}). The three most interesting examples for us correspond to:
\begin{align}
    \Li_2(x)+\Li_2(x^{-1})
    &=-\frac{\pi^2}{6}-\frac{1}{2}\log(-x)^2\,,\\
    \Li_2(x)+\Li_2(1-x)
    &=\frac{\pi^2}{6}-\log(x)\log(1-x)\,,\\
    \Li_2(x)+\Li_2(y)-\Li_2(xy)+\Li_2\Big(\frac{1-x}{1-xy}\Big)-\Li_2\Big(y\frac{1-x}{1-xy}\Big)
    &=\frac{\pi^2}{6}-\log(x)\log\Big(\frac{1-x}{1-xy}\Big)\,.
\end{align}
(Note that $y=0$ in the third equality implies the second.)
Faddeev and Kashaev used the fact that these equations come from critical values of combinations of dilogarithms (explicitly,
\[
    \Li_2(z)-\Li_2(zy)+\log(z)\log(x)
    =
    \frac{\pi^2}{6}-\Li_2(x)-\Li_2(y)+\Li_2(xy)\,,
\]
where $-\log(1-z)+\log(1-zy)+\log(x)=0$) to describe non-commutative analogues of those functional equations. In~\cite{FK94}, they prove that for $\tau\in\RR_{>0}$, $w\in(0,\tau+1)$, and $y\in(-1,\tau-w)$,
we have the following three identities:
    \begin{align}
    \Phi_{S,1,1}(z;\tau)\Phi_{S,0,0}(-z;\tau)
    &=iq^{-\frac{1}{12}}\tq^{\frac{1}{12}}\ee\Big(-\frac{z^2}{2\tau}-\frac{z}{2}+\frac{z}{2\tau}\Big)\label{eq:PhiS.reflection}\\
    \int_{i\RR-0}\Phi_{S,1,0}(z;\tau)\;\ee\Big(\frac{zw}{\tau}\Big)\;dz
    &=\ee(1/8)\sqrt{\tau}q^{-\frac{1}{24}}\tq^{\frac{1}{24}}\;
    \Phi_{S,0,1}(w;\tau)^{-1}\label{eq:PhiS.fourier}\\
    \int_{i\RR-0}\frac{\Phi_{S,1,0}(z;\tau)}{\Phi_{S,0,0}(z+y;\tau)}
    \ee\Big(\frac{zw}{\tau}\Big)dz
    &=
    \ee(1/8)\sqrt{\tau}q^{-\frac{1}{24}}\tq^{\frac{1}{24}}
    \frac{\Phi_{S,0,0}(w+y;\tau)}{\Phi_{S,0,0}(y;\tau)\Phi_{S,0,1}(w;\tau)}\label{eq:PhiS.5term}\,.
    \end{align}
Applying a stationary phase approximation to these integrals as $\tau\to0$ (i.e., a semiclassical limit) recovers the three classical dilogarithmic identities above.

These identities are analogues of the Jacobi triple product formula and the $q$-binomial theorem, i.e., Equation~\eqref{eq:jac.trip}, Equation~\eqref{eq:euler-second} and Equation~\eqref{eq:q-binomial-series} respectively.  
They can be analytically continued outside of the region we have described here by collecting residues and applying the $(q,\tq)$-difference equations.
For example, when $y\in\RR_{<\tau-w}$ but $y\notin\ZZ_{<0}+\tau\ZZ_{\leq 0}$, the analytic continuation of the left hand side is given by exactly the same contour with the addition of the residues around all zeros of $\Phi_{S,0,0}(z+y;\tau)$ for $\Re(z)\geq0$. The proofs of these identities follow in an elementary manner from the $(q,\tq)$-difference equations. They can also be proved when $\Im(\tau)>0$ by a simple application of the residue theorem.

\subsection{A modular extension}
In this subsection, we will give some basic properties of the modular quantum dilogarithm. This function was defined and studied in~\cite{GKZ} and most of this section can be found there. The modular quantum dilogarithm arises from the cocycle associated to the quantum Jacobi form $(x;q)_\infty$ when we consider arbitrary group elements of $\SL_2(\ZZ)$---whereas Faddeev's quantum dilogarithm corresponds to the element $S=(\smat 0{-1}{1}0)$.

For $\gamma\in\SL_2(\ZZ)$, recall the function $\Phi_{\gamma,m,n}(z;\tau)$ from Equation~\eqref{eq:phigam.def}.
Here are some simple properties of this function:
\begin{align*}
	&\Phi_{\gamma,m,n}(z+a\tau+b;\tau) = \Phi_{\gamma,m+a,n+1}(z;\tau), 
	&&\Phi_{\gamma,m,n}(z+\tau;\tau) = \Phi_{\gamma,m+1,n+d}(z;\tau),\\
	&\Phi_{\gamma,m,n}(z+c\tau+d;\tau) = \Phi_{\gamma,m+c,n}(z;\tau), 
	&&\Phi_{\gamma,m,n}(z+1;\tau) = \Phi_{\gamma,m,n-c}(z;\tau)\,.
\end{align*}
Its quasi-periodicity in the index $(m,n)\in\ZZ^2$ is described by the following:
\begin{equation*}
\begin{aligned}
	&\Phi_{\gamma,m+1,n}(z;\tau) = (1-q^m\ee(z))^{-1}\Phi_{\gamma,m,n}(z;\tau),\\
	&\Phi_{\gamma,m,n+1}(z;\tau) = \Big(1-\tilde q^n\ee\big(\tfrac{z}{c\tau+d}\big)\Big)\Phi_{\gamma,m,n}(z;\tau).
\end{aligned}
\end{equation*}
As an immediate corollary, if $\gamma\tau^{*} = \tau^{*}$ and $q^m\ee(z^{*}) = \tilde q^n\ee(\frac{z^{*}}{c\tau^{*}+d})$, then
    \begin{equation} \label{eq:faddeevperiod}
    \Phi_{\gamma,m+k,n+k}(z^{*};\tau^{*})=\Phi_{\gamma,m,n}(z^{*};\tau^{*})\,,\qquad k\in\ZZ\,.
    \end{equation}
We also have an additional reflection formula: 
	\[\Phi_{\delta\gamma\delta,1-m,1-n}(z;-\tau) = \Phi_{\gamma,m,n}(z;\tau)^{-1},\]
where $\delta=(\smat 100{-1})$.
Of fundamental importance is the following cocycle property
\begin{equation}\label{eq:g2r}
    \Phi_{\gamma\gamma',m,n}(z;\tau)
    =
    \Phi_{\gamma,k,n}(\gamma'(z;\tau))
    \Phi_{\gamma',m,k}(z;\tau)\,,
\end{equation}
where $\gamma(z;\tau) = (\frac{z}{c\tau+d};\frac{a\tau+b}{c\tau+d})$.
An immediate consequence of this, which we will use later is that
\begin{equation}\label{eq:gamr.cocyc}
    \Phi_{\gamma^r,m,n}(z;\tau)
    =
    \prod_{j=0}^{r-1}\Phi_{\gamma,m_j,n_j}(\gamma^j(z;\tau))\,,
\end{equation}
where $m_{0}=m$, $n_{r-1}=n$ and $n_k=m_{k+1}$.

We also have the following slightly more involved formula (the generalisation of Equation~\eqref{eq:PhiS.reflection}) that can be derived using the modularity of the Jacobi $\theta$-function:
\begin{proposition}\label{prop:reflection}
	Faddeev's modular quantum dilogarithm satisfies
	\[\Phi_{\gamma,m+1,n+1}(z;\tau)\Phi_{\gamma,-m,-n}(-z;\tau) = \mu_{\gamma}^{-2}(-1)^{m+n}\ee(Q_{\gamma,m,n}(z,\tau))\,,\]
	where
	\[Q_{\gamma,m,n}(z,\tau)=-\frac{1}{2}\Big[\frac{c z^2}{c\tau+d} + \Big(2m+1 - \frac{2n+1}{c\tau+d}\Big) z + B_2(m+1)\tau - B_2(n+1)\frac{a\tau+b}{c\tau+d}\Big]\]
	and $B_2(x)=x^2-x+1/6$ is the 2nd Bernoulli polynomial.
\end{proposition}

More interestingly, $\Phi_{\gamma,m,n}$ is a meromorphic function of $(z,\tau)\in \CC\times (\CC\sm c^{-1}\RR_{\le -d})$ (i.e., except for when $c\tau+d<0$). Denoting for $z\in\ZZ\tau+\ZZ$, $z=k\tau+\ell=k'(a\tau+b)+\ell'(c\tau+d)$, the poles are in the set
	\[P_{\gamma,\tau} = \left\{z\in\ZZ\tau+\ZZ\,\colon k'+n \le 0, \quad k+m > 0 \right\},\]
and the zeros are in the set
	\[N_{\gamma,\tau} = \left\{z\in\ZZ\tau+\ZZ\,\colon k'+n > 0, \quad k+m \le 0 \right\}.\]
This all follows from the next proposition, which gives product representations of Faddeev's modular quantum dilogarithm reducing the general function to the case $\gamma=S=(\smat{0}{-1}{1}{0})$, $m=n=0$ (i.e., to Faddeev's quantum dilogarithm).
(Note also the relation to~\cite[Eq. (122)]{AK14b}.)
\begin{proposition}\label{prop:prod.id.mod,fad}
\emph{(i)}\cite[Thm. 2.1]{GKZ} For $m,n\in\ZZ$ and $\gamma=(\smat abcd)\in\SL_2(\ZZ)$ with $c\ne 0$, we have
    \begin{equation} \label{eq:modulartofaddeevC}
    \Phi_{\gamma,m,n}(z;\tau) = \prod_{j=0}^{c-1}\Phi_{S,0,0}(z+(m+j)\tau+k_j;c\tau+d)\,,\qquad k_j=\Big\lfloor \frac{d(m+j)-n}{c}\Big\rfloor\,.
    \end{equation}
\emph{(ii)} For $m,n\in\ZZ$ and $\gamma=\prod_{j=1}^{r}(\smat{b_j}{-1}{1}{0})$, where $b_j\ge2$, define $m_r = m$, $m_0 = n$, and choose $m_1,\dots,m_{r-1}$ to be arbitrary integers. Let $\tau_r = \tau$, $z_r = z$ and define $\tau_{i},z_i$ for $i=1,\dots,r-1$ via $\tau_{j-1} = b_j - \frac{1}{\tau_j}$ and $z_{j-1} = \frac{z_j}{\tau_j}$. Then
    \begin{equation} \label{eq:modulartofaddeevCF}
    \Phi_{\gamma,m,n}(z;\tau) = \prod_{j=1}^{r}\Phi_{S,0,0}\big(z_j+m_j\tau_j-m_{j-1};\,\tau_j\big)\,.
    \end{equation}
\end{proposition}

To finish our discussion on the basic properties of Faddeev's modular quantum dilogarithm, we will describe the generalisations of Equation~\eqref{eq:PhiS.fourier}~and~\eqref{eq:PhiS.5term}.
This again follows from the $q$-binomial theorem.
We include a proof in \S\ref{app:mod.fad} for completeness.
Before giving the statement, we state the following:
\begin{lemma}[\cite{GKZ}*{Thm. 2.2}]\label{lem:asymp}
For $\tau\in\RR$ and $z\to\infty$ with fixed argument, we find that
\begin{align}
    \Phi_{\gamma,m,n}(z;\tau)
    =
    \Big\{\begin{array}{cl}
    \mathrm{O}(1) & \text{if }0<\arg(z)<\pi\,,\\
    \mathrm{O}(\ee(Q_{\gamma,-m,-n}(-z,\tau))) & \text{if }-\pi<\arg(z)<0\,.
    \end{array}
\end{align}
\end{lemma}
\begin{proof}
We can use Proposition~\ref{prop:prod.id.mod,fad} to deduce the asymptotic behaviour from the known asymptotics of Faddeev's quantum dilogarithm.
Indeed, for $0<\arg(z)<\pi$ we see that $\Phi_{S,0,0}(z+(m+j)\tau+k_j;c\tau+d)=O(1)$ as $z\to\infty$.
For the other case, we use Proposition~\ref{prop:reflection} to reduce to the first case.
\end{proof}
\begin{theorem}[\cite{Dim15}*{Equ. 1.12},\cite{AK14b}*{\S8}, \cite{GKZ}*{Thm. 2.4}]\label{thm:5term.mod.fad}
For $c\neq 0$, $\tau\in\RR_{>0}$, $0<\frac{cw}{c\tau+d}+\ell$, $\frac{c(w+y)}{c\tau+d}+\ell-1<0$, $\calC_{m,n}=\frac{n}{c}-m\frac{c\tau+d}{c}+i\RR-0$, and all the zeros of $\Phi_{\gamma,m+p,n}(z+y;\tau)$ are to the left of the contours (if not, one can take a deformation of the contour inside a compact set depending on $y$ so that this is true), then Faddeev's modular quantum dilogarithm satisfies the equations:
\begin{align}
&\mu_\gamma^{-1}\sqrt{c\tau+d}\,\Phi_{\gamma,\ell,1-h}(w;\tau)^{-1}\\
&=\sum_{m\in\ZZ/c\ZZ}
\int_{\calC_{m+1,n}}
\Phi_{\gamma,m+1,n}(z;\tau)\ee\Big(\frac{(cz-n+(c\tau+d)m)w}{c\tau+d}+\ell (z+m\tau)+h\Big(\frac{z}{c\tau+d}+n\frac{a\tau+b}{c\tau+d}\Big)\Big)dz\,,\nonumber\\
&\mu_\gamma^{-1}\sqrt{c\tau+d}\,\frac{\Phi_{\gamma,p+\ell,-h}(w+y;\tau)}{\Phi_{\gamma,p,0}(y;\tau)\Phi_{\gamma,\ell,1-h}(w;\tau)}\label{eq:5term.int}\\
&=\sum_{m\in\ZZ/c\ZZ}
\int_{\calC_{m+1,n}}
\frac{\Phi_{\gamma,m+1,n}(z;\tau)\ee\big(\frac{(cz-n+(c\tau+d)m)w}{c\tau+d}+\ell (z+m\tau)+h(\frac{z}{c\tau+d}+n\frac{a\tau+b}{c\tau+d})\big)}{\Phi_{\gamma,m+p,n}(z+y;\tau)}dz\,.\nonumber
\end{align}
\end{theorem}
\begin{remark}
The poles of $\Phi_{\gamma,m,n}(z;\tau)$ are contained in a cone between the lines $-m\tau+\RR$ and $-n(a\tau+b)+(c\tau+d)\RR$.
These intersect at the point $(n-m(c\tau+d))/c$ and this is the origin of this choice of contour.
\end{remark}

\subsection{Relation to Stark--Shintani ray class invariants}
\label{sec:starkshintani}
In this subsection, we describe an explicit relation between special values of Faddeev's modular quantum dilogarithm and Stark–Shintani ray class invariants. The results in this section are not new. The main identity (Proposition~\ref{prop:starkfaddeev}) is essentially a reformulation of the results of Shintani~\cite{Shi77}, \cite{Shi78}, with the added observation, originally due to Yamamoto~\cite{Yam10b}, that products of special values of the double sine appearing in Shintani's formula can be rewritten as a limit of a ratio of two $q$-Pochhammer symbols. Such a formula can be found in a recent paper by Yalkinoglu~\cite{Yal26}, and the same idea has also been recently developed in great detail in a preprint of Kopp~\cite{Kop24}. We do not explicitly use their results here, and to make the paper more self-contained we give a concise description of the main formula following Yamamoto~\cite{Yam08}, \cite{Yam10a} (see also Tangedal~\cite{Tan07}).

Let $K$ be a real quadratic field of discriminant~$\Delta_{K}>0$ with two real embeddings $\sigma_i$, $i=1,2$ and corresponding infinite places $\infty_1,\infty_2$. It will be convenient to identify $K$ with its image $\sigma_1(K)$ in $\RR$. We denote by $x'$ the Galois conjugate of $x\in K$, so that with the above identification we have $\sigma_2(x)=x'$. For any subset $X\subset K$ let $X_{+}$ be the set of all totally positive elements of~$X$, i.e., $X_{+}:=\{x\in X \colon x,x'>0\}$. Let $\eps>1$ be the totally positive fundamental unit in $\lO_{K}$ (so $\eps$ is of norm $1$).

Given an integral ideal~$(0)\ne \ff\subsetneq \lO_K$, we denote the narrow ray class group of modulus~$\mathfrak{f}$ by $\Cl_{K}^{+}(\ff)$. By definition, the group $\Cl_{K}^{+}(\ff)$ is the quotient of the multiplicative group $\mathcal{I}(\ff)$ of all fractional ideals of~$K$ coprime to~$\ff$ by the subgroup of principal ideals $(\alpha)$ with $\alpha\in K_{+}$ and $\alpha\equiv 1\Mod{\ff}$.\footnote{The congruence $\alpha\equiv 1\Mod{\ff}$ means $v_{\fp}(\alpha-1)\ge v_{\fp}(\ff)$ for all $\fp|\ff$.}
We also denote by $\eps_{\ff}$ the generator of the group $(\lO_{K}^{\times}\cap (1+\ff))_{+}$ that satisfies $\eps_{\ff}>1$. Given a narrow ray class $\fC\in\Cl_{K}^{+}(\ff)$ we denote by $\zeta(s,\fC)$ the corresponding partial zeta function\footnote{In this sum $\fa$ runs over non-zero \emph{integral} ideals in~$\fC$.}:
	\[\zeta(s,\fC) := \sum_{\fa\in\fC}N(\fa)^{-s}\,.\]
We also define the sign classes $\fC_i=[(\mu_i)]$, where $\mu_i\equiv 1\Mod{\ff}$ satisfy
	\[\mu_1<0<\mu_1'\,,\qquad \mu_2'<0<\mu_2\,.\]

Let us recall an encoding of ideals belonging to a fixed ray class. Let $\fa$ be an integral ideal in the class~$\fC$ and fix a fractional ideal $\fb=\langle 1,\omega \rangle$, with $0<\omega'<1<\omega$, such that $\fa\fb=(z)\ff$ for some $z\in (K^{\times})_{+}$. Then~$\omega$ can be expanded into a purely periodic Hirzebruch--Jung continued fraction
	\[\omega = \cycle{b_0,\dots,b_{\ell-1}} := b_0-\frac{1}{b_1-\cdots\ds\frac{1}{b_{\ell-1}-\ds\frac{1}{b_0-\cdots}}}\,,\qquad b_i\ge2\,,\]
of (least) period $\ell$. We extend the sequence $b_n$ by $\ell$-periodicity $b_{n+\ell}=b_n$ to all $n\in\ZZ$ and denote $\omega_n:=\cycle{b_{n},\dots,b_{n+\ell-1}}$. Define $A_{n}$ by 
	\[A_0=1\,,\qquad \qquad A_{n+1}=\frac{A_n}{\omega_{n+1}}\,,\qquad n\in \ZZ\,.\]
With these definitions we have $A_{k+1}=b_kA_k-A_{k-1}$ and 
	\[\ang{A_{k+1},A_k}=\dots=\ang{A_0,A_{-1}}=\ang{1,\omega}=\fb\,,\]
and there is a unique pair $(x_k,y_k)\in\QQ^2\cap (0,1]\times [0,1)$ such that 
	\[x_kA_{k-1}+y_kA_k\in z+\fb\,.\]
The choice of $x_0,y_0$ is uniquely determined by $z$ and for other indices they are extended by
	\[\begin{pmatrix} x_{n+1} \\ y_{n+1} \end{pmatrix} \equiv 
	\begin{pmatrix} b_{n} &  1 \\ -1 & 0 \end{pmatrix} 
	\begin{pmatrix} x_{n} \\ y_{n} \end{pmatrix} \Mod{\ZZ^2}\,.\]
The sequence $\{(x_n,y_n)\}_n$ is also periodic, with period $r\ell$, where $r$ is defined by $\eps_{\ff}=\eps^r$. The following result is~\cite[Proposition 2.1.4]{Yam08}.
\begin{proposition}
	For each $n\in\ZZ$, let $Q_n$ be the quadratic form
	\[Q_n(x,y)=\frac{(x\omega_n+y)(x\omega_n'+y)}{\omega_n-\omega_n'}\,.\]
	Then 
	\[\zeta(s,\fC)=(\Delta_K^{1/2}N(\ff))^{-s}\sum_{n\Mod{r\ell}}Z_{Q_n}(s,x_n,y_n)\,,\]
	where $r=\log\eps_{\ff}/\log\eps$ and
	\[Z_{Q}(s,x,y)=\sum_{p,q\ge0}Q(x+p,y+q)^{-s}\,.\]
\end{proposition} 
Let $\zeta_2(s,\omega,z)$ be the Barnes double zeta function
	\[\zeta_2(s,\omega,z) = \sum_{p,q\ge0}(z+p\omega+q)^{-s}\,,\qquad \re s>2\,,\]
that has a meromorphic continuation to $s\in \CC$, let $G(z;\omega)=\zeta_2'(0,\omega,z)$, and denote by
	\[\lS(z;\omega) = \exp(G(1+\omega-z;\omega)-G(z;\omega))\]
the Shintani double sine function. Using Shintani's computation of $Z_Q'(0,x,y)$~\cite[Prop.~3]{Shi77}, one then gets the following result (see~\cite[Theorem 5.1.1]{Yam08}).
\begin{proposition}
	With the above notation we have
	\[\exp(-\zeta'(0,\fC)+\zeta'(0,\fC\fC_1\fC_2)) = \prod_{n\Mod {r\ell}}\lS(x_n\omega_n+y_n;\omega_n)\lS(x_n\omega_n'+y_n;\omega_n')\,.\]
\end{proposition}
Yamamoto then shows~\cite[Theorem 5.2.3, Remark 5.2.4]{Yam08} that the refined invariants
	\[X_1(\fC) = \prod_{n\Mod {r\ell}}\lS(x_n\omega_n+y_n;\omega_n)\,,
	\qquad 
	X_2(\fC) = \prod_{n\Mod {r\ell}}\lS(x_n\omega_n'+y_n;\omega_n')\]
satisfy the symmetries
	\begin{align*}
	X_1(\fC) &= X_1(\fC\fC_1)\,,&& X_1(\fC) = X_1(\fC\fC_2)^{-1}\,,\\
	X_2(\fC) &= X_2(\fC\fC_1)^{-1}\,,&& X_2(\fC) = X_2(\fC\fC_2)\,.
	\end{align*}
In particular, this implies that
	\[\log X_1(\fC) = \frac{1}{2}\Big(-\zeta'(0,\fC)-\zeta'(0,\fC\fC_1)+\zeta'(0,\fC\fC_2)+\zeta'(0,\fC\fC_1\fC_2)\Big)\,.\]
Comparing this to Stark~\cite[p.~65]{Sta76}, under assumption $\fC_1\ne \fC_2$, Stark's number $\epsilon_{m}(\fc)$ is then
	\[\epsilon_{m}(\fc) = X_1(\fC)^{-mt}\,,\]
where $\fc$ is the image of $\fC$ in $\Cl_{K}^{+}(\ff)/\ang{\fC_1}$, and $t=1$ if $\fC_1=1$ and $t=2$ otherwise. 

It remains to relate the invariant $X_1(\fC)$ to Faddeev's modular quantum dilogarithm. We denote 
	\[\gamma=\pmat {b_0}{-1}{1}{0}\pmat {b_1}{-1}{1}{0}\cdots \pmat {b_{\ell-1}}{-1}{1}{0}\in \SL_2(\ZZ)\,.\]
With the standard action of $\SL_2(\ZZ)$ on $\mathbb{P}^1(\CC)$, we have $\gamma\omega_0=\omega_0$ and also $\eps_{\ff}=c\omega_0+d$, where $\gamma^r=(\smat abcd)$.

\begin{proposition} \label{prop:starkfaddeev}
With the above notation, we have
	\[X_1(\fC) = |\Phi_{\gamma^{r},m,0}(x_0\omega_0+y_0;\omega_0)|\,,\]
where $m=(d-1)x_0-cy_0\in\ZZ$.
\end{proposition}
\begin{proof}
The double sine function can be expressed in terms of Faddeev's function via
    \[\lS(z;\tau) = \Phi_{S,0,0}(z-1;\tau)\exp
	  \Big[\frac{\pi i}{2\tau}\Big(z^2-(\tau+1)z+\frac{1+3\tau+\tau^2}{6}\Big)\Big]\,.\]
Therefore,
    \[X_1(\fC) = \prod_{j=1}^{N}|\Phi_{S,0,0}(x_j\omega_j+y_j-1;\omega_j)|\,,\]
where we set $N=r\ell$. Setting $\tau_j = \omega_j$, $z_N=x_0\omega_0+y_0=x_N\omega_N+y_N$, and $z_{j-1}=z_j/\omega_j$, $j=1,\dots,N$, we have
    \[\tau_{j-1}=b_{j-1}-\frac{1}{\tau_j}\,,\qquad z_{j-1}=\frac{z_j}{\tau_j}\,.\]
If we write $z_j=u_j\tau_j+v_j$, then $u_N=x_0$, $v_N=y_0$, and the recursion is $u_{j-1}=-v_j$, $v_{j-1}=u_{j}+v_jb_{j-1}$. Let us choose $m_j$ so that $z_j+m_j\tau_j-m_{j-1}=x_{j}\omega_{j}+y_j-1$, $j=1,\dots,N$. This can be achieved by setting $m_N=0$ and $m_{j}=x_j-u_j$, $j=0,\dots,N-1$; both sets of equations hold since $y_j+x_{j-1}=1$. Therefore, by part (ii) of Proposition~\ref{prop:prod.id.mod,fad}
    \[X_1(\fC) = |\Phi_{\gamma^r,0,m_0}(x_0\omega_0+y_0;\omega_0)|\,.\]
We compute $m_0=x_0-u_0=cy_0+(1-d)x_0=-m$.
Finally, since $\gamma\omega_0=\omega_0$ and $z_N=\frac{z_N}{c\omega_0+d}+m_0\omega_0$, by~\eqref{eq:faddeevperiod}
    \[\Phi_{\gamma^r,0,m_0}(z_N;\omega_0)= \Phi_{\gamma^r,-m_0,0}(z_N;\omega_0)\,,\]
proving the claim.
\end{proof}
This shows that for real quadratic fields all Stark units appear among special values of the modular quantum dilogarithm.

\section{Identities for special values of Faddeev's modular quantum dilogarithm}
\label{sec:mainidentities}

In this section, we will prove that certain specialisations of Faddeev's quantum dilogarithm at algebraic arguments lie on algebraic varieties defined over $\ZZ[\zeta_{2N},1/N]$, where $\zeta_{2N}$ is a primitive $2N$-th root of unity.
These varieties give analogues of the functional equations of the quantum dilogarithm.
This will be explained in detail in \S\ref{sec:finitedilog}. We will show in \S\ref{sec:algebraicity} that these varieties are in fact zero dimensional, implying algebraicity of the special values. Before proving the general result, we will treat the case of Faddeev's quantum dilogarithm $\Phi=\Phi_{S,0,0}$ that is technically easier.

\subsection{Special values of Faddeev's quantum dilogarithm}

The evaluation of Faddeev's quantum dilogarithm in the upper and lower half planes is explicitly given by Equation~\eqref{eq:prod.fadeev}.
It was shown in~\cite{GK17} how to compute the function~$\Phi_{S,m,n}(z;\tau)$ explicitly for $\tau\in\QQ$.
Our main interest is the value when~$\tau$ solves certain quadratic equations.
In particular, fix $N\in\ZZ_{>0}$ and consider $\tau=\eps$ where $\eps^2-(N+2)\eps+1=0$.
Thanks to~\cite{Lub99}, we can give explicit and convergent summation formulas for $\Phi_{S,m,n}(z,\eps)$ when $\Im(z)\neq 0$.
However, we also want $z\in\RR$ and in particular $z\in\frac{\eps-1}{N}\ZZ$.
The reason that these arguments are of particular interest is that one would naively expect their value to be unity.
Indeed, at these points we find that $q=\tq$ and $\ee(z)=\ee(z/\eps)$.
However, Faddeev's quantum dilogarithm is a cocycle and not a coboundary (which could be proved by calculating the value at these points to some numerical accuracy).
While these values are not unity, they are algebraic numbers (that are expected to be units), see Theorem~\ref{thm:algebraicity}. For now, we show that they give rise to points on a variety that is a finite analogue of the Equations~\eqref{eq:PhiS.reflection},~\eqref{eq:PhiS.fourier},~and~\eqref{eq:PhiS.5term}.

For $u\in\ZZ/N\ZZ$, define $F_N^{\pm}(u)=\Phi_{S,0,0}(\frac{u}{N}(1-\eps);\eps)\ee(\frac{N-1}{24})$ with $\pm=\sgn(u/(\eps-1))$ where $\eps^2-(N+2)\eps+1=0$.
This is well defined, since for $z=\frac{u}{N}(1-\eps)$ and $z\neq\eps-1$, we have
\[
    \Phi_{S,0,0}(z+(1-\eps);\eps)=\frac{1-q^{-1}\ee(z)}{1-\tq^{-1}\ee(z/\eps)}\Phi_{S,0,0}(z)=\Phi_{S,0,0}(z)\,.
\]
From this computation, we also see that $F_N^+(u)=\eps^{\pm \delta_{u,0}}F_{N}^-(u)$ and $F_{N}^+(0)>F_{N}^-(0)$.
(We note that $\Phi_{S,0,0}(0;\eps)=\eps$.)
We have the following:
\begin{theorem}\label{thm:fn.equs}
The numbers $F_N^{\pm}(u)$ satisfy the algebraic equations:
\begin{align}
    F_N^\pm(0)&=\sqrt{\frac{N+2\pm\sqrt{N^2+4N}}{2}}\,,\label{eq:Fpm.zero}\\
    F_N^{+}(u)F_N^{-}(-u)
    &=
    (-1)^u\ee\Big(-\frac{u^2}{2N}\Big)\,,\label{eq:Fpm.reflection}\\
    \frac{1}{\sqrt{N}}\sum_{x\in\ZZ/N\ZZ}\ee\Big(\frac{xu}{N}\Big)F_{N}^+(x)
    &=
    \ee\Big(\frac{N-1}{24}\Big)\frac{1}{F_{N}^{-}(u)}\,,\label{eq:Fpm.fourier}\\
    \frac{1}{\sqrt{N}}\sum_{x\in\ZZ/N\ZZ}\ee\Big(\frac{xu}{N}\Big)\frac{F_{N}^+(x)}{F_{N}^-(x+v)}
    &=
    \frac{F_N^+(u+v)}{F_{N}^-(u)F_{N}^-(v)}\,,\qquad (u,v)\ne (0,0)\,.\label{eq:Fpm.5term}
\end{align}
\end{theorem}
\begin{proof}
Equation~\eqref{eq:Fpm.zero} follows from the modularity of the $\eta$-function.
Equation~\eqref{eq:Fpm.reflection} follows from the specialisation of Equation~\eqref{eq:PhiS.reflection} (essentially the modularity of the $\theta$-function).
To prove Equation~\eqref{eq:Fpm.fourier}, we use the residue theorem and Equation~\eqref{eq:PhiS.fourier}.
Note that $q=\tq$ and suppose that $q\ee(w/\eps)=\ee(w)$ and for simplicity that $0<\eps<1$ (a similar argument works when $\eps>1$).
Then $w-1=u(1-\eps)/N$ for some $u\in\ZZ$ and we have the following:
\[
\begin{aligned}
    &\frac{\eps}{1-\eps}\sum_{x=-N}^{-1}\ee\Big(\frac{xu}{N}\Big)\Phi_{S,0,0}\Big(\frac{x}{N}(1-\eps);\eps\Big)\\
    &=
    2\pi i\sum_{x=-N}^{-1}\Res{z=\frac{x}{N}(1-\eps)}\ee\Big(\frac{z(w-1)}{\eps}\Big)\frac{\Phi_{S,0,0}(z;\eps)}{1-\ee(z(\eps-1)/\eps)}\;dz\\
    &=
    \bigg(\int_{i\RR-0}-\int_{\eps-1+i\RR-0}\bigg)\ee\Big(\frac{zw}{\eps}\Big)\frac{\Phi_{S,0,0}(z;\eps)}{\ee(z/\eps)-\ee(z)}\;dz\\
    &=
    \int_{i\RR-0}\ee\Big(\frac{zw}{\eps}\Big)\frac{\Phi_{S,0,0}(z;\eps)}{\ee(z/\eps)-\ee(z)}\;dz
    -\int_{\eps-1+i\RR-0}\ee\Big(\frac{zw}{\eps}+\frac{w(1-\eps)}{\eps}\Big)\frac{\Phi_{S,0,0}(z;\eps)}{\tq^{-1}\ee(z/\eps)-q^{-1}\ee(z)}\;dz\\
    &=
    \int_{i\RR-0}\ee\Big(\frac{zw}{\eps}\Big)\frac{\Phi_{S,0,0}(z;\eps)-\Phi_{S,0,0}(z+\eps-1;\eps)}{\ee(z/\eps)-\ee(z)}\;dz\\
    &=
    \int_{i\RR-0}\Phi_{S,1,0}(z;\eps)\;\ee\Big(\frac{zw}{\eps}\Big)\frac{(1-\ee(z))-(1-\ee(z/\eps))}{\ee(z/\eps)-\ee(z)}\;dz\\
    &=\int_{i\RR-0}\Phi_{S,1,0}(z;\eps)\;\ee\Big(\frac{zw}{\eps}\Big)\;dz\\
    &=\ee(1/8)\sqrt{\eps}\ee\Big(\frac{-N-2}{24}\Big)\Phi_{S,0,0}\Big(\frac{u}{N}(1-\eps);\eps\Big)^{-1}\,.
\end{aligned}
\]
(To compute the residue we need to assume not just that $0<w$ but that $0<w-\eps$, which implies that $w-1>\eps-1$.)
Noting that for $0<\eps<1$ and $\ell\in\ZZ_{\geq0}$, we have $\Phi_{S,0,0}(\ell(1-\eps);\eps)\ee(\frac{N-1}{24})=\sqrt{\eps}=F_N^-(0)$ and $\Phi_{S,0,0}(\eps-1;\eps)\ee(\frac{N-1}{24})=\sqrt{1/\eps}=F_N^+(0)$.
This therefore implies Equation~\eqref{eq:Fpm.fourier}.

A similar trick also allows for a proof of the 5-term identity using Equation~\eqref{eq:PhiS.5term}. (Again we will assume $0<\eps<1$ with the other case following from a similar argument.)
We suppose that $v\neq0$, since the identity for $v=0$ follows from elementary considerations.
Let $v,u$ be represented in $(-N,0]$ with $y=v(1-\eps)/N$ and $w-1=u(1-\eps)/N$.
This implies that $\eps<w$, $y+w-1<0$, and $-y-1<-\eps$.
Consider the contour~$\calC$ in Figure~\ref{fig:contour.5term}.
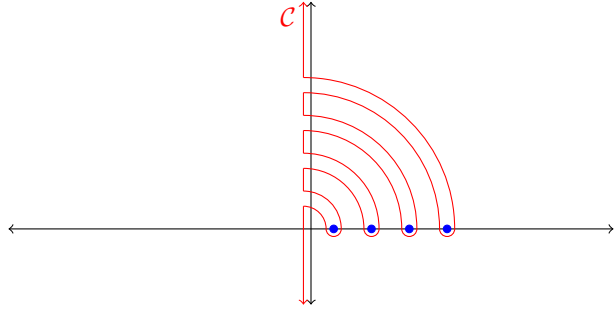
\begin{figure}
\centering
\begin{tikzpicture}
\draw[<->](-4,0)--(4,0);
\draw[<->](0,-1)--(0,3);
\draw[red,<-,xshift=-0.1cm](0,3)--(0,2);
\draw[red](-0.3,2.8) node {$\calC$};

\foreach \x in {0,1,2,3}
{\draw[red,xshift=-0.1cm](0,2-\x*0.5) arc (90:0:2-\x*0.5);
\draw[red,xshift=-0.1cm] (2-\x*0.5,0) arc (0:-180:0.1);
\draw[red,xshift=-0.1cm](0,2-\x*0.5-0.2) arc (90:0:2-\x*0.5-0.2);
\filldraw[blue,xshift=-0.1cm](2-\x*0.5-0.1,0) circle (0.05cm);
}

\foreach \x in {2,1.5,1}
\draw[red,xshift=-0.1cm](0,\x-0.2)--(0,\x-0.5);

\draw[->,red,xshift=-0.1cm](0,0.3)--(0,-1);
\end{tikzpicture}
\caption{The contour \textcolor{red}{$\calC$} used to prove Equation~\eqref{eq:Fpm.5term}.
The \textcolor{blue}{blue} bullets are located at $-y-k\tau$ for $k\in\ZZ_{>0}$.
This is one particular choice, the main point is that asymptotically the contour is $i\RR-0$ (even this can be relaxed) and partitions the complex plan into two domains; the one on the right containing only poles of $\Phi_{S,1,0}(z;\eps)$ and one on the left containing only poles of $\Phi_{S,0,0}(z+y;\eps)^{-1}$.
We remark that the contour cannot be deformed to avoid poles when $-y+k\tau+\ell$ for $k-1,\ell\in\ZZ_{\geq0}$ lands at a point $i\tau+j$ with $i,j\in\ZZ$, which are exactly the poles of the RHS of Equation~\eqref{eq:PhiS.5term}.
}
\label{fig:contour.5term}
\end{figure}
Using in many places that $q=\tq$, $\ee(y)=\ee(y/\eps)$, and $\ee(w)=\tq\ee(w/\tau)$, we find that
\[
\begin{aligned}
&\frac{\eps}{1-\eps}\sum_{x\in\ZZ/N\ZZ}\ee\Big(\frac{xu}{N}\Big)\frac{F_{N}^+(x)}{F_{N}^-(x+v)}\\
&=
\bigg(\int_\calC-\int_{\calC+\eps-1}\bigg)
\frac{\Phi_{S,0,0}(z;\eps)}{\Phi_{S,0,0}(z+y;\eps)}\frac{\ee(wz/\eps)}{\ee(z/\eps)-\ee(z)}dz\\
&=
\int_\calC
\frac{\Phi_{S,0,0}(z;\eps)}{\Phi_{S,0,0}(z+y;\eps)}\frac{\ee(wz/\eps)}{\ee(z/\eps)-\ee(z)}\Big(1-\frac{(1-\ee(z/\eps))(1-\ee(y+z))}{(1-\ee(z))(1-\ee((y+z)/\eps))}\Big)dz\\
&=
\int_\calC
\frac{\Phi_{S,1,0}(z;\eps)}{\Phi_{S,0,0}(z+y-1;\eps)}\ee(wz/\eps)(1-\ee(y))dz\\
&=
\sqrt{\eps}
\frac{F_N^+(u+v)}{F_{N}^-(u)F_{N}^-(v)}\,.
\end{aligned}
\]
where the convergence and applicability of the residue theorem are guaranteed by the choice of contour---which avoids other unwanted singularities---and the inequalities above (since the integrand grows like $\ee(z(w+\eps)/\eps)$ when $z\sim i\infty$ and like $\ee(z(y+w-1)/\eps)$ when $z\sim-\infty$).
The final equality uses the analytic continuation of Equation~\eqref{eq:PhiS.5term}, where we see the contour $\calC$ exactly comes from the deformation that deforms around the zeros of $\Phi_{S,0,0}(z+y-1;\eps)$. 
\end{proof}

We will close with some brief and immediate observations.
Firstly, if $F_N^{\pm}(u)$ solves the equations of Theorem~\ref{thm:fn.equs}, then $F_N^{\pm}(-u)$ is another solution.
The two choices of solution for $\eps^2-(N+2)\eps+1$ give rise to two solutions related by this symmetry.
This can be proved using the reflection formulas for Faddeev's quantum dilogarithm.
A similar, but non-algebraic, equation can also be derived for the values $F_{N}^{\pm}(u)$ from the reflection formulas of the quantum dilogarithm.
In particular,
\[
    \overline{F_{N}^{\pm}(u)}
    =
    \frac{1}{F_{N}^{\mp}(-u)}\,.
\]
This immediately implies that
\[
    \overline{F_{N}^{\pm}(u)} = F_N^{\pm}(u)(-1)^u\ee\Big(\frac{u^2}{2N}\Big)\,.
\]
\begin{remark}
In defining $F_N^\pm$ for $\ZZ/N\ZZ$, we have implicitly used the fact that $\Phi_{T,m,n}(z;\tau)=1$. Indeed, it would be more natural to define
    \[F_N^{\pm}(u)=\Phi_{T^{N+2}S,0,0}\Big(\frac{u}{N}(1-\eps);\eps\Big)\ee\Big(\frac{N-1}{24}\Big)\,,\]
with $\pm=\sgn(u/(\eps-1))$ where $\eps^2-(N+2)\eps+1=0$.
\end{remark}

\subsection{Proof of Theorem \texorpdfstring{\ref{thm:fg.equs}}{2}}

We will now describe how the last section generalises to Faddeev's modular quantum dilogarithm.
This plays the important role of constructing solutions to new systems of equations similar to those in Theorem~\ref{thm:fn.equs} associated to a product of two cyclic groups.
It however---and equally importantly---constructs new solutions to the equations in Theorem~\ref{thm:fn.equs}.

Fix a hyperbolic $\gamma\in\SL_2(\ZZ)$ with $c>0$ and trace $N+2$ for $N>0$, and let $\tau$ be a fixed point. Recall $G$, $\Lambda$, $F_{\gamma}^{\pm}(u)$, $\ang{u}_{\gamma}$, and $\ang{u;v}_{\gamma}$ from the introduction, and in particular Equation~\eqref{eq:fgam.def}.
We note that
\[
    \ang{u;v}_\gamma
    =
    \frac{\ang{u+v}_\gamma}{\ang{u}_\gamma\ang{v}_\gamma}\,.
\]

\begin{lemma}\label{lem:lam.inv}
The functions $F_{\gamma}^{\pm}(u)$ are well-defined on~$G$.
\end{lemma}
\begin{proof}
Firstly, for all $m\in\ZZ$, note that
    \[
    \Phi_{\gamma,u_1,0}\Big(\frac{u_1\tau+u_2}{\eps^{-1}-1}\Big)
    =
    \Phi_{\gamma,u_1+m,m}\Big(\frac{u_1\tau+u_2}{\eps^{-1}-1}\Big)\,.
    \]
Recalling that $N=a+d-2$, notice that
\[
\begin{aligned}
\Phi_{\gamma,u_1+N,0}\Big(N\frac{\tau}{\eps^{-1}-1}+\frac{u_1\tau+u_2}{\eps^{-1}-1}\Big)
&=
\Phi_{\gamma,u_1+N,0}\Big(\tau-(a\tau+b)+\frac{u_1\tau+u_2}{\eps^{-1}-1}\Big)\\
&=
\Phi_{\gamma,u_1+N+1-a,d-1}\Big(\frac{u_1\tau+u_2}{\eps^{-1}-1}\Big)\\
&=
\Phi_{\gamma,u_1,0}\Big(\frac{u_1\tau+u_2}{\eps^{-1}-1}\Big)\,.
\end{aligned}
\]
Similarly,
\[
\begin{aligned}
\Phi_{\gamma,u_1,0}\Big(N\frac{1}{\eps^{-1}-1}+\frac{u_1\tau+u_2}{\eps^{-1}-1}\Big)
&=
\Phi_{\gamma,u_1,0}\Big(1-(c\tau+d)+\frac{u_1\tau+u_2}{\eps^{-1}-1}\Big)\\
&=
\Phi_{\gamma,u_1-c,-c}\Big(\frac{u_1\tau+u_2}{\eps^{-1}-1}\Big)\\
&=
\Phi_{\gamma,u_1,0}\Big(\frac{u_1\tau+u_2}{\eps^{-1}-1}\Big)\,.
\end{aligned}
\]
Finally, if $v(I-\gamma)=Nw$ then $w(\gamma^{-1}-I)=v$, and we find that $\frac{v_1\tau+v_2}{\eps^{-1}-1}=(v_1\tau+v_2)(1-\eps)/N=w_1\tau+w_2$.
Hence,
\[
\begin{aligned}
\Phi_{\gamma,u_1+v_1,0}\Big(\frac{(u_1+v_1)\tau+(u_2+v_2)}{\eps^{-1}-1}\Big)
&=
\Phi_{\gamma,u_1+v_1,0}\Big(w_1\tau+w_2+\frac{u_1\tau+u_2}{\eps^{-1}-1}\Big)\\
&=
\Phi_{\gamma,u_1+v_1+(1-d)w_1+cw_2,0}\Big(w_1\tau+w_2+\frac{u_1\tau+u_2}{\eps^{-1}-1}\Big)\\
&=
\Phi_{\gamma,u_1,0}\Big(\frac{u_1\tau+u_2}{\eps^{-1}-1}\Big)\,.
\end{aligned}
\]
Therefore, this verifies that $F_\gamma^{\pm}$ is $\Lambda$-invariant, and concludes the proof.
\end{proof}

We note that $G$ has fundamental representatives $(u_1,u_2)\in\ZZ^2$ with $0\le u_1< N/\gcd(a-1,b)$, $0\le u_2<\gcd(a-1,b)$ or similarly $0\le u_1< \gcd(c,d-1)$, $0\le u_2<N/\gcd(c,d-1)$.
The case $\gamma=T^{N+2}S=(\smat {N+2}{-1}{1}0)$ recovers the setup of the previous subsection, where it is easy to see that $F_{T^{N+2}S}^\pm=F_{N}^\pm$ and $\ang{u}_{T^{N+2}S}=(-1)^u\ee_{2N}(u^2)$.

\begin{proof}[Proof of Theorem~\ref{thm:fg.equs}.]

Equation~\eqref{eq:Fgpm.zero} follows from the very definition. (These correspond to values of $\Phi_{\gamma,m,n}$ on the lattice and their agreement again follows from the modularity of the $\eta$-function.)
Equation~\eqref{eq:Fgpm.reflection} follows from Proposition~\ref{prop:reflection}.
We can prove the other identities by proving for example that
\[
\frac{1}{c\sqrt{N}}\sum_{x_1=0}^{c-1}\sum_{x_2=0}^{N-1}\ang{x;u}_{\gamma}F_{\gamma}^{+}(x) = \mu_{\gamma}\frac{1}{F_{\gamma}^{-}(u)}.
\]
The sum on the left hand side over-counts by a factor of $c$, which is compensated by a factor of $1/c$.
Indeed, we find that we can---completely analogously to the proof of Theorem~\ref{thm:fn.equs}---compute the sums via the residue theorem along with the functional equations of Theorem~\ref{thm:5term.mod.fad}.

Denote $z_{m,k}=\frac{m\tau+k}{\eps^{-1}-1}$, and notice that
\[
    \ee\Big(\frac{cz_{m,w}z_{\ell,x}}{\eps}+mz_{\ell,x}+\ell(z_{m,w}+m\tau)\Big)
    =
    \ee_{N}((a-1)xm-b\ell m+cxw+(1-d)\ell w)
    =
    \ang{(m,w);(\ell,x)}_\gamma\,.
\]
Then, using the notation of Theorem~\ref{thm:5term.mod.fad}, we note that $\calC_{m,0}+\frac{m}{c}\eps=i\RR-0=\calC_{m',0}+\frac{m'}{c}\eps$.
Moreover, $a\tau+b=\frac{a\eps-1}{c}$ and if we take $m'\equiv m+1-a\pmod{c}$ with $0\leq m,m'<c$ we find that
\[
    a\tau+b+\Big\lfloor\frac{m'+1-a}{c}\Big\rfloor(c\tau+d)
    =
    \frac{\eps-1}{c}
    +\frac{(m'-m)\eps}{c}\,.
\]
Suppose that $q^\ell\ee(w)=\tq\ee(w/\eps)$ with $w=\frac{u_1\tau+u_2}{\eps^{-1}-1}$ and again that $0<\eps<1$ (the case $\eps>1$ following from a similar argument).
We find that
\[
\begin{aligned}
&\frac{\eps}{1-\eps}\sum_{x_1=0}^{c-1}\sum_{x_2=0}^{N-1}\ang{x;u}_{\gamma}F_{\gamma}^{+}(x)\\
&=\sum_{m=0}^{c-1}
\bigg(\int_{\calC_{m,0}}-\int_{\calC_{m,0}+\eps-1}\bigg)\frac{\ee\big(\frac{(cz+(c\tau+d)m)w}{c\tau+d}+\ell(z+m\tau)\big)}{\ee(z/\eps)-q^m\ee(z)}\Phi_{\gamma,m,0}(z;\tau)dz\\
&=\sum_{m,j=0}^{c-1}
\bigg(\int_{\calC_{m,0}+j\frac{\eps-1}{c}}-\int_{a\tau+b+\lfloor\frac{m'+1-a}{c}\rfloor(c\tau+d)+\calC_{m',0}+j\frac{\eps-1}{c}}\bigg)\\&\qquad\qquad\qquad\qquad\qquad\qquad\qquad\qquad\qquad\qquad\quad\frac{\ee\big(\frac{(cz+(c\tau+d)m)w}{c\tau+d}+\ell(z+m\tau)\big)}{\ee(z/\eps)-q^m\ee(z)}\Phi_{\gamma,m,0}(z;\tau)dz\\
&=\sum_{m,j=0}^{c-1}
\int_{\calC_{m,0}+j\frac{\eps-1}{c}}\frac{\ee\big(\frac{(cz+(c\tau+d)m)w}{c\tau+d}+\ell(z+m\tau)\big)}{\ee(z/\eps)-q^m\ee(z)}\Phi_{\gamma,m,0}(z;\tau)dz\\
&\qquad\qquad-\sum_{m,j=0}^{c-1}
\int_{a\tau+b+\calC_{m,0}+j\frac{\eps-1}{c}}\frac{\ee\big(\frac{(cz+(c\tau+d)(m+1-a))w}{c\tau+d}+\ell (z+(m+1-a)\tau)\big)}{\ee(z/\eps)-q^{m+1-a}\ee(z)}\Phi_{\gamma,m+1-a,0}(z;\tau)dz\\
&=\sum_{m,j=0}^{c-1}
\int_{\calC_{m,0}+j\frac{\eps-1}{c}}\frac{\ee\big(\frac{(cz+(c\tau+d)m)w}{c\tau+d}+\ell(z+m\tau)\big)}{\ee(z/\eps)-q^m\ee(z)}\big(\Phi_{\gamma,m,0}(z;\tau)-\Phi_{\gamma,m+1,1}(z;\tau)\big)\,dz\\
&=\sum_{m,j=0}^{c-1}
\int_{\calC_{m,0}+j\frac{\eps-1}{c}}\ee\big(\frac{(cz-n+(c\tau+d)m)w}{c\tau+d}+\ell(z+m\tau)\big)\Phi_{\gamma,m+1,0}(z;\tau)\,dz\\
&=
\frac{c\mu_\gamma\sqrt{\eps}}{F_{\gamma}^{-}(u)}\,.
\end{aligned}
\]

Finally, a similar computation produces the $5$-term relation.
We use contours $\calD_{m}(y)$, which are deformations of the contours $\calC_{m,0}$, which (just as described in Figure~\ref{fig:contour.5term}) asymptotically behave like $\calC_{m,0}$ and separates the complex plane into two domains with one containing the poles of $\Phi_{\gamma,m+1,0}(z;\tau)$ and the other containing the poles of $\Phi_{\gamma,m+p,0}(z+y;\tau)^{-1}$.
(Again, this can only not be achieved at the poles of the LHS of Equation~\eqref{eq:5term.int}.)
Then, supposing that $v\neq0$, $q^p\ee(y)=\ee(y/\eps)$, and $q^\ell\ee(w)=\tq\ee(w/\eps)$
\[
\begin{aligned}
&\frac{\eps}{1-\eps}\sum_{x_1=0}^{c-1}\sum_{x_2=0}^{N-1}
\ang{x;u}_{\gamma}\frac{F_{\gamma}^{+}(x)}{F_{\gamma}^{-}(x+v)}\\
&=\sum_{m=0}^{c-1}
\bigg(\int_{\calD_{m}(y)}-\int_{\calD_{m}(y)+\eps-1}\bigg)\frac{\ee\big(\frac{(cz+\eps m)w}{\eps}+\ell(z+m\tau)\big)}{\ee(z/\eps)-q^m\ee(z)}\frac{\Phi_{\gamma,m,0}(z;\tau)}{\Phi_{\gamma,m+p,0}(z+y;\tau)}dz\\
&=\sum_{m,j=0}^{c-1}
\int_{\calD_m(y)+j\frac{\eps-1}{c}}\frac{\ee\big(\frac{(cz+\eps m)w}{\eps}+\ell(z+m\tau)\big)}{\ee(z/\eps)-q^m\ee(z)}\frac{\Phi_{\gamma,m,0}(z;\tau)}{\Phi_{\gamma,m+p,0}(z+y;\tau)}dz\\
&\qquad\qquad-\sum_{m,j=0}^{c-1}
\int_{a\tau+b+\calD_m(y)+j\frac{\eps-1}{c}}\frac{\ee\big(\frac{(cz+\eps(m+1-a))w}{\eps}+\ell (z+(m+1-a)\tau)\big)}{\ee(z/\eps)-q^{m+1-a}\ee(z)}\frac{\Phi_{\gamma,m+1-a,0}(z;\tau)}{\Phi_{\gamma,m+p+1-a,0}(z+y;\tau)}dz\\
&=\sum_{m,j=0}^{c-1}
\int_{\calD_m(y)+j\frac{\eps-1}{c}}\frac{\ee\big(\frac{(cz+\eps m)w}{\eps}+\ell(z+m\tau)\big)}{\ee(z/\eps)-q^m\ee(z)}\frac{\Phi_{\gamma,m,0}(z;\tau)}{\Phi_{\gamma,m+p,0}(z;\tau)}\Big(1-\frac{(1-\ee(z/\eps))(1-q^{m+p}\ee(z+y))}{(1-q^m\ee(z))(1-\ee((z+y)/\eps))}\Big)\,dz\\
&=\sum_{m,j=0}^{c-1}
\int_{\calD_m(y)+j\frac{\eps-1}{c}}\ee\big(\frac{(cz-n+\eps m)w}{\eps}+\ell(z+m\tau)\big)\frac{\Phi_{\gamma,m+1,0}(z;\tau)}{\Phi_{\gamma,m+p-a,0}(z+y+a\tau+b;\tau)}(1-q^p\ee(y))\,dz\\
&=
c\sqrt{\eps}\mu_\gamma^{-1}\frac{\Phi_{\gamma,p+\ell,1}(w+y;\tau)}{\Phi_{\gamma,p+1,1}(y;\tau)\Phi_{\gamma,\ell,1}(w;\tau)}\\
&=
c\sqrt{\eps}\frac{F_{\gamma}^{+}(u+v)}{F_{\gamma}^{-}(u)F_{\gamma}^{-}(v)}\,.
\end{aligned}
\]
The case $v=0$ can be proved by elementary means.
\end{proof}

We close by summarising some simpler properties of the numbers $F_\gamma^\pm(u)$.
\begin{proposition}\label{prop:basic.sym.fpm}
If $F_\gamma^{\pm}(u)$, $u\in G$, satisfies the equations of Theorem~\ref{thm:fg.equs}, then $F_\gamma^{\pm}(-u)$ is another solution.
\end{proposition}
\begin{proposition}
The numbers $F_\gamma^{\pm}(u)$ satisfy
\[
    \overline{F_{\gamma}^{\pm}(u)}
    =
    \frac{1}{F_{\gamma}^{\mp}(-u)}\,.
\]
\end{proposition}
\noindent Together with~\eqref{eq:Fgpm.reflection} this gives the following:
\begin{corollary}\label{cor:argument}
We have
    \begin{equation} \label{eq:Fgconjugation}
    \overline{F_{\gamma}^{\pm}(u)} = \ang{u}_\gamma F_\gamma^{\pm}(u)\,.
    \end{equation}
\end{corollary}

\subsection{Relations between special values}\label{sec:rels.bet.g}
The previous identities have all been between values $F_\gamma^\pm(u)$ for a fixed~$\gamma\in\SL_2(\ZZ)$. However, there are some simple relations that the values satisfy for varying~$\gamma$. We summarise some of these relations in this section.
\begin{proposition}\label{prop:powers}
For any $r\ne 0$, the numbers $F_\gamma^\pm(u)$ satisfy the equations:
\[
\begin{aligned}
    F_{\gamma^r}^{\pm}(u^{(r)})
    &=
    F_{\gamma}^{\pm}(u)^r\,,\\
    F_{\gamma^r}^{\pm}(u\gamma)
    &=
    F_{\gamma^r}^{\pm}(u)\,,
\end{aligned}
\]
where $\gamma\tau=\tau$ and $\ee(\frac{u^{(r)}_1\tau+u^{(r)}_2}{\eps^{-r}-1}+u^{(r)}_1\tau)=\ee(\frac{u_1\tau+u_2}{\eps^{-1}-1}+u_1\tau)$.
\end{proposition}
\begin{proof}
This follows from the cocycle property of $\Phi_{\gamma,m,n}$ and in particular Equation~\eqref{eq:gamr.cocyc}.
For the second equation, we observe that multiplication by $\gamma$ permutes the factors in the product, which of course leaves the product fixed. (Note that the factors in this product are not necessarily algebraic in general.)
\end{proof}
\noindent Therefore, all values of $F_{\gamma}^{\pm}$ are determined by the values of $F_{\gamma^r}^{\pm}$ for any $r\neq0$ (using for example Corollary~\ref{cor:argument} to remove the ambiguity of $r$-th roots of unity).

The next identity relates different matrices with the same trace. Denote $z_{m,k}=\frac{m\tau+k}{\eps^{-1}-1}$. Suppose that $g=(m,k)\in G$ has order $\mathrm{ord}(g)$.
Moreover, suppose that $h$ is the smallest number such that $\ee(z_{hm,hk}+hm\tau)\in\ee(x/y+\tau\ZZ)$, for some $x,y\in\ZZ$.
Then, we observe that $h(m,k)(I-\gamma)^{-1}\in(\ZZ,x/y)$ and hence $h(m,k)\in(\ZZ,x/y)(I-\gamma)=(\ZZ-cx/y,\ZZ+(1-d)x/y)$. Therefore, we see that $c/y,(1-d)/y\in\ZZ$ and so
\[
    U_{y}^{-1}\gamma U_{y}=\begin{pmatrix}a&yb\\c/y&d\end{pmatrix}\in\SL_2(\ZZ)\,,\qquad\text{where}\qquad
    U_m=\begin{pmatrix}1&0\\0&m\end{pmatrix}.
\]
Therefore, suppose that for some $w\in\ZZ$ that $\ee(wyz_{m,k}+wym\tau)\in\ee(f/h)q^{y/h+y\ZZ}$, then we find that $(1,f)(I-U_{y}^{-1}\gamma U_{y})\in h\ZZ^2$, and hence 
\[
V_{f,h}U_{y}^{-1}\gamma U_{y}V_{f,h}^{-1}=\begin{pmatrix}a+cf/y&(-fya + y^2b - f^2c + fyd)/yh\\hc/y&d-cf/y\end{pmatrix}\in\SL_2(\ZZ)\,,
\qquad\text{where}\qquad
V_{f,h}=\begin{pmatrix}1&f\\0&h\end{pmatrix}.
\]
\begin{proposition}\label{prop:dist}
With the above notation, we find that
\begin{equation}\label{eq:dist.prop}
    \mu_\gamma^{-\mathrm{ord}(g)}\prod_{n=1}^{\mathrm{ord}(g)}F_\gamma^{(\pm)}(\ell+nm,j+nk)
    =
    \mu_{V_{f,h}U_{y}^{-1}\gamma U_{y}V_{f,h}^{-1}}^{-1}F_{V_{f,h}U_{y}^{-1}\gamma U_{y}V_{f,h}^{-1}}^{(\pm)}(h\ell,yj-f\ell)\,.
\end{equation}
\end{proposition}
\begin{proof}
This follows from the identities
\[
    \prod_{j\in\ZZ/M\ZZ}(\ee_M(j)x;q)_\infty
    =
    (x^M;q^M)_\infty\,,
    \qquad\text{and}\qquad
    \prod_{j=0}^{M-1}(q^jx;q^M)_\infty
    =
    (x;q)_\infty\,,
\]
which give analogous identities for Faddeev's modular quantum dilogarithms,
and the following:
\[
\begin{aligned}
    \mu_\gamma^{-\mathrm{ord}(g)}
    \prod_{n=1}^{\mathrm{ord}(g)}F_\gamma^{(\pm)}(\ell+nm,j+nk)
    &=
    \mu_\gamma^{-\mathrm{ord}(g)}
    \prod_{p=0}^{y-1}
    \prod_{n=1}^{h}
    F_\gamma^{(\pm)}(\ell+(ph+n)m,j+(ph+n)k)\\
    &=
    \mu_{U_{y}^{-1}\gamma U_{y}}^{-\mathrm{ord}(h)}
    \prod_{n=1}^{h}
    F_{U_{y}^{-1}\gamma U_{y}}^{(\pm)}(\ell+nm,y(j+nk))\\
    &=
    \mu_{V_{f,h}U_{y}^{-1}\gamma U_{y}V_{f,h}^{-1}}^{-1}F_{V_{f,h}U_{y}^{-1}\gamma U_{y}V_{f,h}^{-1}}^{(\pm)}(h\ell,yj-f\ell)\,. \qedhere
\end{aligned}
\]
\end{proof}

\bigskip
\section{Finite quantum dilogarithms}
\label{sec:finitedilog}
In this section, we will reinterpret the identities from \S\ref{sec:mainidentities} in a more abstract way by introducing the notion of a finite quantum dilogarithm and deriving its basic properties. A more conceptual point of view will be given in~\S\ref{sec:izumi}.

\subsection{Andersen--Kashaev's framework for quantum dilogarithms}
\label{sec:andersenkashaev}

To motivate our definition we begin by recalling Andersen--Kashaev's abstract framework for quantum dilogarithms~\cite[\S6]{AK14b}. The difference is that we will be dealing exclusively with finite abelian groups~$G$, so the analytic aspects are not going to matter. Let $N=|G|$.
\begin{definition}
A metric group is a finite abelian group~$G$ together with a choice of a function $\ang{\cdot}\colon G\to \CC^{\times}$ such that $\ang{x}=\ang{-x}$ for all $x$ and 
    \[\ang{x;y} := \frac{\ang{x+y}}{\ang{x}\ang{y}}\]
is a non-degenerate bicharacter on~$G$.
\end{definition}
Let $G$ be a metric group. We normalise the counting measure on~$G$ by $1/\sqrt{N}$ and denote
    \[\int_Gf(x)dx := \frac{1}{\sqrt{N}}\sum_{x\in G} f(x)\,.\]
The Fourier transform
    \[\widehat{f}(y) = (\bF^{-1}f)(y) := \int_{G}f(x)\ang{x;-y}{}d\mu(x)\]
then satisfies $\bF^4=1$, $\bF^2f(x)=f(-x)$. (Here~$\bF$ is the operator with the kernel $\ang{x;y}$, so in our notation it is the inverse Fourier transform.)

\begin{definition} A u-operator is a group homomorphism $\mathbf{f}$ from $G$ to the group of unitary operators on $L^2(G)$. We denote it by $x\mapsto \mathbf{f}^x$.
\end{definition}

\begin{definition} Two u-operators $\bp$ and $\bq$ form a Heisenberg pair of level $\ell=\ell(\bp,\bq)$, if for all $x,y\in G$ they satisfy	the relation
	\[\bp^x\bq^y = \ang{x;y}{}^{\ell}\bq^y\bp^x\,.\]
\end{definition}
Any $u$-operator forms a Heisenberg pair of level~$0$ with itself. A standard example of a Heisenberg pair of level~$1$
is the pair defined by
    \[(\bp^x{}f)(y)=f(x+y)\,,\qquad\qquad (\bq^xf)(y)=\ang{y;x}{}f(y)\,.\]
For any $\bp$, $\bq$ forming a Heisenberg pair, we define their sum $\bp+\bq$ by
	\[(\bp+\bq)^x = \ang{x}{}^{\ell(\bp,\bq)}\bq^x\bp^x\,.\]
From now on, we reserve notation $\bp$, $\bq$ for the above mentioned canonical Heisenberg pair of level 1. Suppose that $g\colon G\to\CC$ and a u-operator $\mathbf{f}$ is given. Then we define
	\[g(\mathbf{f})=\int_{G}\widehat{g}(x)\,\mathbf{f}^xdx\,.\]
For the standard Heisenberg pair, we have
\begin{align*}
	g(\bq)f(x) &= g(x)f(x)\,,\\
	g(\bp) &= \bF g(\bq)\bF^{-1}\,,\\
	g(\bp+\bq) &= \ang{\bq}{}^{-1}g(\bp)\ang{\bq}{}\,.
\end{align*}
Indeed, for the first identity it follows from $\int_{G}\widehat{g}(x)\ang{y;x} = g(y)$, for the second it suffices to note that $\bp^x\bF = \bF \bq^x$, and the last one follows from 
	\[(\ang{\bq}(\bp+\bq)^xf)(y) = \ang{y}\ang{x}\ang{y;x}f(x+y) = \ang{x+y}f(x+y) = (\bp^x\ang{\bq}f)(y)\,.\]
In~\cite[\S6]{AK14b}, Andersen and Kashaev gave a general definition of a quantum dilogarithm $f\colon G\to \CC$ for a locally compact abelian group~$G$. The main property is the non-commutative pentagon relation
    \[f(\bp)f(\bq) =  f(\bq)f(\bp+\bq)f(\bp)\,,\]
but in addition $f$ is required to satisfy a reflection identity. In the case $G=\RR$, their definition agrees with the properties of Faddeev's quantum dilogarithm~$\Phi$. When specialised to finite groups however, their definition does not seem to give any interesting functions. We now give a definition that appears to work better for finite groups.

\subsection{Finite quantum dilogarithms}
 Our definition of a finite quantum dilogarithm differs from Andersen--Kashaev in that we allow for a specific defect in the pentagon relation.
\begin{definition}
    Let $G$ be a metric group and assume $|G|>1$. A \emph{finite quantum dilogarithm} for a metric group~$G$ is a function $E\colon G\to\CC$ satisfying
    \begin{equation} \label{eq:pentagonmatrix}
    \widehat{E}(\bp)\widehat{E}(\bq) - \widehat{E}(\bq)\widehat{E}(\bp+\bq)\widehat{E}(\bp) = C \delta(\bq)\delta(\bp)\,,
	\end{equation}
	for some non-zero constant $C\in\CC^{\times}$, where $\delta\colon G\to\CC$ is the delta function at~$0$.
\end{definition}
Equivalently, $E$ is a finite quantum dilogarithm if and only if it satisfies
\begin{equation} \label{eq:pentagonintegral}
	\ang{x;y}E(x)E(y) = C + \int_{G}E(y-z)E(z)\ang{z}E(x-z)dz
\end{equation}
for all $x,y\in G$, where $C\in \CC^{\times}$ is a non-zero constant (we recall that $\int_{G}f(x)dx=\frac{1}{\sqrt{N}}\sum_{x\in G}f(x)$). Indeed, the equivalence can be seen by writing
	\[\widehat{E}(\bp) = \int_{G}E(x)\bp^{-x}dx\,,\qquad \widehat{E}(\bq) = \int_{G}E(x)\bq^{-x}dx\,,
	\qquad \widehat{E}(\bp+\bq) = \int_{G}\ang{x}E(x)\bq^{-x}\bp^{-x}dx\,,\]
and comparing the coefficients of $\bq^{-y}\bp^{-x}$ on both sides of~\eqref{eq:pentagonmatrix}.

The following theorem gives the basic properties of finite quantum dilogarithms.

\begin{theorem} \label{thm:fqdilogbasicproperties}
	Let $E\colon G\to \CC$ be a finite quantum dilogarithm for a metric group $G$ with $|G|=N>4$. Then
	\begin{thmenum}
		\item $E(0)^2=\sqrt{N}E(0)+1$ and $C=-\widehat{E}(0)$;
		\item $E(u)E(-u) = \ang{u}^{-1}$ for all $u\ne 0$;
		\item $\widehat{E}(u) = \lambda\ang{u}E(u)$ for all $u\in G$, where $\lambda=\widehat{E}(0)/E(0)$ satisfies $\lambda^3=\int_{G}\ang{x}^{-1}dx$;
		\item the function $\widetilde{E}(u)=\ang{u}E(u)$ satisfies the dual pentagon equation
		\begin{equation} \label{eq:pentagonintegraldual}
		\ang{x;y}^{-1}\widetilde E(x)\widetilde E(y) = \widetilde{C} + \int_{G}\widetilde E(y-z)\widetilde E(z)\ang{z}^{-1}\widetilde E(x-z)dz
		\end{equation}
		for $\widetilde C=\lambda^{-2}C$.
	\end{thmenum}
\end{theorem}
\begin{proof}
	Let $p_x(z)=E(x-z)E(z)\ang{z}$. Then~\eqref{eq:pentagonintegral} can be rewritten using convolution as
	\[E(x)E(y)\ang{x;y} = C + (p_x\star E)(y)\,,\]
	which after Fourier inversion becomes
	\begin{equation} \label{eq:supportnew}
		E(x)\widehat{E}(w-x) = C\sqrt{N}\delta(w) + \widehat{E}(w)\widehat{p_x}(w)\,.
	\end{equation}
	This immediately implies that $0\in\mathrm{supp}(\widehat{E})$ and that $\mathrm{supp}(\widehat{E})$ is $H$-invariant, where $H$ is the span of $S=\mathrm{supp}(E)$. From~\eqref{eq:pentagonintegral} we also immediately get that $H=G$, since otherwise taking $x,y\not\in H$ would contradict $C\ne0$. Then we also get $\mathrm{supp}(\widehat{E})=G$. Taking $x=0$ in~\eqref{eq:supportnew} gives that $\widehat{p_0}(w)=E(0)$ for $w\ne 0$ and
		\[\widehat{p_0}(0) = E(0)-\frac{C\sqrt{N}}{\widehat{E}(0)}\,.\]
	This implies that 
	\begin{equation} \label{eq:reflection1}
    E(u)E(-u)\ang{u}=\sqrt{N}E(0)\delta(u)-\frac{C}{\widehat{E}(0)}\,.
    \end{equation}
	In particular, this implies that $E(u)\ne 0$ for all $u\in G$. Next, note that
		\[\widehat{p_x}(w) 
		= \int_{G}E(z)\ang{z}E(x-z)\ang{-w;z}dz 
		= \ang{x}\ang{-w;x}\int_{G}E(x-z)\ang{z}E(z)\ang{w-x;z}dz = \frac{\ang{x-w}}{\ang{w}}\widehat{p_{x}}(x-w) \,,\]	
	so
	\[\frac{E(x)\widehat{E}(w-x)-C\sqrt{N}\delta(w)}{\widehat{E}(w)\ang{w}^{-1}} = \frac{E(x)\widehat{E}(-w)-C\sqrt{N}\delta(x-w)}{\widehat{E}(x-w)\ang{x-w}^{-1}}\,.\]
	When $w\ne0$ and $x\ne w$, we get
	\[\widehat{E}(w-x)\widehat{E}(x-w)\ang{x-w}^{-1} = \widehat{E}(w)\ang{w}^{-1}\widehat{E}(-w)\,,\]
	so that $\widehat{E}(u)\widehat{E}(-u)\ang{u}^{-1} = d$ for some constant~$d$, for all $u\ne 0$. Then, setting $w=0$ and $x\ne 0$ gives
	\[E(x)\widehat{E}(-x)\widehat{E}(x)\ang{x}^{-1}-C\sqrt{N}\widehat{E}(x)\ang{x}^{-1} = E(x)\widehat{E}(0)^2\,,\]
	which can be rewritten as
	\begin{equation} \label{eq:weileigen}
	E(x)(d - \widehat{E}(0)^2) = C\sqrt{N}\widehat{E}(x)\ang{x}^{-1}\,,
	\end{equation}
	hence $\widehat{E}(x) = \lambda \ang{x}E(x)$ for all $x\ne 0$, for some constant $\lambda\in\CC^{\times}$.
    
    Thus to prove (i) and (ii) it remains to show that $C=-\widehat{E}(0)$ and $\lambda = \widehat{E}(0)/E(0)$. Taking $x=w\ne 0$ in~\eqref{eq:supportnew} and using the symmetry $\widehat{p_x}(x)=\widehat{p_x}(0)\ang{x}^{-1}$ gives
	\[E(x)\widehat{E}(0) = \widehat{E}(x)\ang{x}^{-1}\widehat{p_x}(0)\,,\]
	so $\widehat{E}(0)=\lambda\widehat{p_x}(0)$ for $x\ne 0$. Since $\widehat{p_x}(0)=(\widetilde{E}\star E)(x)$, we see that $\widehat{\widetilde{E}}(y)\widehat{E}(y)$ is constant on $y\ne 0$, so $\widehat{\widetilde{E}}(y)$ is proportional to $E(-y)$ on $G\sm \{0\}$. Therefore, for some $c\in \CC$, $\widetilde{E}(x)-c\widehat{E}(x)$ is constant, and since $\ang{x}E(x)$ is not constant on $G\sm\{0\}$\footnote{Since $|G|>4$, if $x\mapsto\ang{x}E(x)$ were constant on $G\sm\{0\}$, then by~\eqref{eq:reflection1} so would be~$\ang{\cdot}$, which is impossible for $N>4$.}, we see that $\widehat{E}(x)=\lambda\ang{x}E(x)$ for all $x$, so $\lambda=\widehat{E}(0)/E(0)$. This then also shows that $d=-\lambda^2C/\widehat{E}(0) = -C\widehat{E}(0)/E(0)^2$. Plugging this back into~\eqref{eq:weileigen} then gives
		\[-C\widehat{E}(0)/E(0)^2 - \widehat{E}(0)^2 = C\sqrt{N}\lambda\,.\]
	We can rewrite this as
        \[E(0)^2 = -\frac{C}{\widehat{E}(0)}(E(0)\sqrt{N}+1)\,,\]
    and using~\eqref{eq:reflection1} gives
        \[E(0)^2 = (E(0)^2-E(0)\sqrt{N})(E(0)\sqrt{N}+1)\,,\]
    which, since $E(0)\ne0$, is equivalent to $E(0)^2=\sqrt{N}E(0)+1$. Comparing this with~\eqref{eq:reflection1} again also gives us the identity $C=-\widehat{E}(0)$.

    To prove (iii), it remains to show that $\lambda^3$ is equal to the Gauss sum $\int_G\ang{x}^{-1}dx$. This is a general property of the Fourier--Weil operator $Wf(x)=\ang{x}^{-1}\widehat{f}(x)$: 
    \begin{align*}
    W^3f(x) &= \int_{G^3}\ang{x}^{-1}\ang{x;r}^{-1}\ang{r}^{-1}\ang{r;s}^{-1}\ang{s}^{-1}\ang{s;t}^{-1}f(t)drdsdt\\
            &= \int_{G^3}\ang{x+r}^{-1}\ang{r+t}\ang{r+s+t}^{-1}f(t)drdsdt\\
            &= \int_{G}\ang{s}^{-1}ds\int_{G^2}\ang{x+r}^{-1}\ang{r+t}f(t)drdt\\
            &= \int_{G}\ang{s}^{-1}ds\int_{G^2}\ang{t-x}\ang{x+r;t-x}f(t)drdt = f(x)\int_{G}\ang{s}^{-1}ds\,.
    \end{align*}
    Specialising this to $f=E\ne0$ gives $\lambda^3 = \int_{G}\ang{x}^{-1}dx$, as claimed.
    
	Finally, we show (iv). Multiplying~\eqref{eq:pentagonintegraldual} by $\ang{y;w}$ and integrating in $y$ gives an equivalent relation
	\[\lambda^{-1}\ang{x}E(x)E(w-x) = \sqrt{N}\widetilde{C}\delta_0(w) 
	+ \lambda^{-1}E(w)\int_{G}E(x-z)E(z)\ang{z;w}\ang{x-z}dz\,,\]
	where we used $\widehat{E}(x)=\lambda \ang{x}E(x)=\lambda \widetilde{E}(x)$.
    Multiplying both sides by $\lambda^2\ang{x}^{-1}\ang{w-x}$ gives
		\[ E(x)\widehat{E}(w-x) = \lambda^2\sqrt{N}\widetilde{C}\delta(w) 
    	+ \ang{w;-x}\widehat{E}(w)\int_{G}E(x-z)E(z)\ang{z;w}\ang{x-z}dz\,,\]
	or equivalently
		\[ E(x)\widehat{E}(w-x) = \lambda^2\sqrt{N}\widetilde{C}\delta(w) 
    	+ \widehat{E}(w)\int_{G}E(x-z)E(z)\ang{z-x;w}\ang{x-z}dz\,.\]
    Changing $z\mapsto x-z$ in the integral gives an equivalent identity
		\[E(x)\widehat{E}(w-x) = \sqrt{N}\lambda^2\widetilde{C}\delta(w)  + \widehat{E}(w)\widehat{p_x}(w)\,.\]
	Since this is just~\eqref{eq:supportnew}, (iv) is proved with $\widetilde{C}=\lambda^{-2}C$.
\end{proof}
\begin{remark}
It is convenient to extend/change the definition of finite quantum dilogarithm for $|G|\le 4$ by imposing the relation $E(0)^2=\sqrt{|G|}E(0)+1$; then Theorem~\ref{thm:fqdilogbasicproperties} holds for all~$G$. Without imposing the condition on $E(0)$, there are additional exceptional solutions for $G=\ZZ/2\ZZ$, $G=\ZZ/3\ZZ$, and $G=(\ZZ/2\ZZ)^2$.
\end{remark}
\begin{remark} \label{rem:dualpentagon}
It is not hard to see from the proof that~~\eqref{eq:pentagonintegraldual} also implies~\eqref{eq:pentagonintegral}, so the two pentagon relations are in fact equivalent.
\end{remark}
\begin{remark}
One can define finite quantum dilogarithms with values in any field $\bk$, provided that it contains enough roots of unity and its characteristic is coprime to $2|G|$. The proof of Theorem~\ref{thm:fqdilogbasicproperties} does not use any special properties of~$\CC$.
\end{remark}

It is sometimes more convenient to use the following form of the pentagon relation (it can be seen from~\eqref{eq:supportnew} using parts (ii) and (iii) of Theorem~\ref{thm:fqdilogbasicproperties}):
    \begin{equation} \label{eq:pentagonproduct3}
    \frac{1}{\sqrt{N}}\sum_{t\in G}E(t)\ang{t}E(x-t)\ang{t;y} =
	\ang{x+y}E(-x-y)E(x)E(y) - NE(0)\delta(x)\delta(y)\,.
    \end{equation}

We can apply Plancherel's theorem to~\eqref{eq:supportnew} to get nontrivial identities involving only $|E(x)|^2$.
\begin{corollary}\label{cor:abs2identity}
    If $E\colon G\to\CC$ is a finite quantum dilogarithm, then for all $x\in G\sm \{0\}$ we have
    \[\sum_{u+v=x}|E(u)|^2|E(v)|^2 = |E(x)|^2\sum_{u+v=-x}|E(u)|^2|E(v)|^2\,.\]
\end{corollary}
\begin{proof}
    With $p_x(z) = E(x-z)\ang{z}E(z)$, we can rewrite~\eqref{eq:supportnew} as
    \[\widehat{p_x}(w)=\begin{cases}
        E(x)\widehat{E}(w-x)\widehat{E}(w)^{-1}\,, &w\ne 0\,,\\
        E(0)\,, &w = 0\,.
    \end{cases}\]
    Therefore, for $x\ne 0$ and all $w$ we have $|\widehat{p_x}(w)|=|E(x)E(w-x)E(-w)|$. Then
    \[\sum_{z}|E(x-z)E(z)|^2 = \sum_{z}|p_x(z)|^2 = \sum_{w}|\widehat{p_x}(w)|^2 = \sum_{w}|E(x)E(w-x)E(-w)|^2\,,\]
    for all $x\ne0$, which is equivalent to the claim.
\end{proof}

There is another curious property of finite quantum dilogarithms, analogous to the integral operator version of the pentagon relation given by Goncharov in~\cite{Gon08}. Given a finite quantum dilogarithm $E\colon G\to \CC$, let us define the linear operator $\bK=\bK_E$ on $\CC^{G}$ by
    \[\bK f(x) := E(x)\widehat{f}(-x)\,.\]
It turns out that this operator has a hidden $5$-fold symmetry.
\begin{theorem}
    For all $f\colon G\to \CC$ we have that $(\bK^5-\lambda^2)f$ is a $\CC$-linear combination of $1$, $\delta$, and $E$. More precisely, we have
    \[
    (\lambda^{-2}\bK^5-1)f(x) = \sqrt{N}E(0)\big(f(0)+E(0)\widehat{f}(0)\big)\delta(x) + E(0)\widehat{f}(0) + \Big(E(0)\int{f(y)E(-y)\ang{y}dy}\Big)E(x)\,.
    \]
\end{theorem}
\begin{proof}
Let us compute the kernel of $\bK^5$ directly: its $(x_0,x_5)$-entry is
    \begin{align*}
    (\bK^5)_{x_0,x_5}&=\int_{G^4}E(x_0)\ang{x_0;x_1}E(x_1)\ang{x_1;x_2}E(x_2)\ang{x_2;x_3}E(x_3)\ang{x_3;x_4}E(x_4)\ang{x_4;x_5}dx_1dx_2dx_3dx_4 \\
    &=E(x_0)\int_{G^2}\widehat{E}(-x_0-x_2)E(x_2)\ang{x_2;x_3}E(x_3)\widehat{E}(-x_3-x_5)dx_2dx_3\\
    &=CE(0)^2E(x_0)+E(x_0)\int_{G^3}\widehat{E}(-x_0-x_2)E(x_2-t)E(t)\ang{t}E(x_3-t)\widehat{E}(-x_3-x_5)dx_2dx_3dt\\
    &=CE(0)^2E(x_0)+\lambda^2E(x_0)\int_{G}\widehat{p_{-x_0-t}}(0)E(t)\ang{t}\widehat{p_{-x_5-t}}(0)dt\,,
    \end{align*}
    where, as in the proof of Theorem~\ref{thm:fqdilogbasicproperties}, we set $p_x(z) = E(x-z)E(z)\ang{z}$. By~\eqref{eq:supportnew}, we have
    \[\widehat{p_x}(0) = \sqrt{N} + \frac{E(x)\widehat{E}(-x)}{\widehat{E}(0)} = \sqrt{N}+\frac{(E(0)^2-1)\delta(x)+1}{E(0)} = E(0)+\sqrt{N}\delta(x)\,,\]
    so we may rewrite
    \[(\bK^5)_{x_0,x_5} = E(x_0)\Big(CE(0)^2 + \lambda E(0)^3 + \lambda^2E(0)(E(-x_0)\ang{x_0}+E(-x_5)\ang{x_5})+\lambda^2\sqrt{N}\delta(x_0-x_5)E(-x_0)\ang{x_0}\Big).\]
    Since $C=-\widehat{E}(0)$ and $\lambda = \widehat{E}(0)/E(0)$, the first two terms inside the brackets cancel. Therefore,
    \[\bK^5f(x) = \lambda^2 E(x)E(-x)\ang{x}\Big(f(x) + E(0)\int_{G}f(y)dy\Big)+\lambda^2E(0)E(x)\int_{G}f(y)E(-y)\ang{y}dy\,,\]
    concluding the proof, since $E(x)E(-x)\ang{x}=\sqrt{N}E(0)\delta(x)+1$.
\end{proof}

From the above calculation, we see that the two vector spaces
    \[U = \CC1+\CC\delta+\CC E\,\qquad\mbox{ and }\qquad V = \Big\{f\in \CC^A\colon f(0)=\widehat{f}(0)=\int_{G}f(y)E(-y)\ang{y}dy=0\Big\}\]
are $\bK$-invariant and give a direct sum decomposition of $\CC^A$. 
It is also easy to check that the restriction of~$\bK$ to~$U$ has characteristic polynomial $X^3-\widehat{E}(0)X-\widehat{E}(0)$.\footnote{For small $N$, when $\dim(U)<3$, it annihilates $\bK|_{U}$.}

\subsection{Finite quantum dilogarithms from real quadratic fields}
Let us now verify that the equations of Theorem~\ref{thm:fg.equs} fit our abstract definition of a finite quantum dilogarithm. As before, for $\gamma\in\SL_2(\ZZ)$ with trace $N+2$ and $N\ge1$, we set $G=\ZZ^2/\Lambda_{\gamma}$, where $\Lambda_{\gamma}=N\ZZ^2+\ker(\gamma-1)=\mathrm{im}(\gamma-1)$, and take
	\[\ang{u} := (-1)^{u_1+u_2+u_1u_2}\ee_{2N}(\omega(u\gamma,u))\,,\qquad\text{and}\qquad \ang{u;v} = \ee_N(\omega(u\gamma-u,v))\,.\]
With this structure~$G$ becomes a metric group of order~$N$. Let us define
    \[E(u) = F_{\gamma}^{+}(u)\,.\]
Note that $E(0)^2=\sqrt{N}E(0)+1$ is true by~\eqref{eq:Fgpm.zero}. From~\eqref{eq:Fgpm.fourier} and~\eqref{eq:Fgpm.reflection} we get that $\widehat{E}(x)=\lambda \ang{x}E(x)$ with $\lambda=\mu_{\gamma}$. 
By~\eqref{eq:Fgpm.reflection} and~\eqref{eq:Fgpm.5term}, we get
    \[\frac{1}{\sqrt{N}}\sum_{x\in G}\ang{x;u}E(x)E(-x-v)\ang{x+v} =
    E(u+v)E(-u)E(-v)\ang{u}\ang{v}\,,\qquad (u,v)\ne (0,0)\,.\]
In view of $E(0)^2=\sqrt{N}E(0)+1$, we can rewrite this as
    \[\frac{1}{\sqrt{N}}\sum_{x\in G}\ang{-x;u}E(-x-v)E(x)\ang{x} =
    \ang{u+v}E(u+v)E(-u)E(-v) - NE(0)\delta(u)\delta(v)\,.\]
Multiply both sides by $\ang{u;z}\ang{v;w}$ and integrate over~$u$ and~$v$. Then the identity becomes
    \[\widehat{E}(w)E(z)\ang{z}\ang{z;w}^{-1} = -E(0)+\int_{G^2}\ang{u+v}E(u+v)E(-u)E(-v)\ang{u;z}\ang{v;w}dudv\]
or, using $\ang{u+v}E(u+v)=\lambda^{-1}\widehat{E}(u+v)$, it can be rewritten as
    \begin{align*}
    \lambda E(w)\ang{w}E(z)\ang{z}\ang{z;w}^{-1} &= -E(0)+\lambda^{-1}\int_{G^3}E(t)\ang{-u-v;t}E(-u)E(-v)\ang{u;z}\ang{v;w}dudvdt \\
     &= -E(0)+\lambda\int_{G^3}E(t-u)\ang{t-u}E(t)E(t-v)\ang{t-v}dt\,.
    \end{align*}
This is precisely the dual pentagon relation~\eqref{eq:pentagonintegraldual}, so in view of Remark~\ref{rem:dualpentagon} we see that~$E$ fits the definition of a finite quantum dilogarithm. We can summarise this as.
\begin{proposition} \label{prop:rqffinitedilog}
For any $\gamma\in\SL_2(\ZZ)$ with trace $N+2\ge3$, setting $G$ and $\ang{\cdot}$ as above, the function $E\colon G\to\CC$ given by $E(u)=F_{\gamma}^{+}(u)$ is a finite quantum dilogarithm on~$G$.
\end{proposition}

Taking $\gamma=(\smat{ab+1}{a}{b}{1})$ for $a|b$ gives $G\cong \ZZ/a\ZZ\times \ZZ/b\ZZ$. Since any product of two cyclic groups is isomorphic to a group of this form, we get the following immediate corollary.
\begin{corollary} \label{cor:existence}
Finite quantum dilogarithms exist for all groups $G$ of the form $G=\ZZ/n\ZZ\times \ZZ/m\ZZ$.
\end{corollary}
\begin{remark}
Since the only construction we have is intimately tied to real quadratic fields, it is tempting to conjecture that finite quantum dilogarithms only exist for groups $G=\ZZ/n\ZZ\times \ZZ/m\ZZ$. It is easy to check by hand that none exist for $G=(\ZZ/2\ZZ)^k$, $k\ge3$, and using computer algebra we have also verified that finite quantum dilogarithms do not exist for $G=(\ZZ/3\ZZ)^3$.
\end{remark}

\textbf{An example.}
Recall that conjugacy classes of matrices of trace $\ge3$ in $\SL_2(\ZZ)$ are in bijection with cycles\footnote{Here by a cycle we mean an equivalence class of a tuple modulo cyclic permutations.} $(b_0,\dots,b_{\ell-1})$ of integers $b_i\ge2$, not all equal to~$2$, see Neumann~\cite[Prop. 6.3]{Neu81}. The map from cycles to conjugacy classes is given by
    \[(b_0,\dots,b_{\ell-1})\quad\mapsto\quad \pmat{b_0}{-1}{1}0 \pmat{b_1}{-1}{1}0\cdots\pmat{b_{\ell-1}}{-1}{1}0 = T^{b_0}S\,T^{b_1}S\,\cdots \,T^{b_{\ell-1}}S\,.\]
This gives a convenient way to catalogue finite quantum dilogarithms coming from different conjugacy classes in $\SL_2(\ZZ)$. Here we give just one example.

For $N=5$, Proposition~\ref{prop:rqffinitedilog} gives two non-trivial examples of finite quantum dilogarithms on $G=\ZZ/5\ZZ$ (up to the symmetry $E(x)\mapsto E(-x)$): one coming from the conjugacy class with cycle~$(3,3)$ and another from the cycle~$(7)$. Let us denote them by~$E_a$ and~$E_b$. The values at~$0$ are the same: $E_a(0)=E_b(0)=\frac{3+\sqrt{5}}{2}$. Exceptionally, $E_a$ is an even function and its values on~$G\sm\{0\}$ are roots of unity:
    \[E_a(1) = E_a(-1) = \zeta_{10}\,,\qquad E_a(2)=E_a(-2)=\zeta_{10}^{-1}\,,\]
with the Gaussian $\ang{x}_{a} = \zeta_{10}^{3x(x+5)}$. For $E_b$ we have $\ang{x}_b=\zeta_{10}^{x(x+5)}$ and the non-trivial values of~$E_b$ are
    \[E_b(1)=r_3\zeta_{5},\quad E_b(2)=r_2\zeta_{5}^4,\quad 
      E_b(3)=r_1\zeta_{5}^4,\quad E_b(4)=r_4\zeta_{5},\]
where $r_1<r_2<r_3<r_4$ are the four roots of $x^4-x^3-\frac{3+3\sqrt{5}}{2}x^2-x+1$. In this case one can verify directly that all finite quantum dilogarithms on~$\ZZ/5\ZZ$ are Galois conjugate to either~$E_a$ or~$E_b$.

\medskip
\textbf{Ideal-theoretic point of view.} So far we were labelling special values of the modular quantum dilogarithm by a pair of $\gamma\in\SL_2(\ZZ)$ and a point in $\ZZ^2$. A more conceptual point of view is to attach these special values to torsion elements of quadratic ideals fixed by the unit. Specifically, by the Latimer-MacDuffee theorem~\cite{LM33}\footnote{More precisely, by its easy to derive $\SL_2(\ZZ)$-variant.}, conjugacy classes of $\SL_2(\ZZ)$ matrices of trace~$t>2$ are in bijection with oriented ideal classes for the order $\ZZ[\eps]$, where $\eps^2-t\eps+1=0$. For $\gamma=(\smat abcd)\in\SL_2(\ZZ)$ with $\tr\gamma=N+2$, we set $\tau = \frac{a-d+\sqrt{N(N+4)}}{2c}$ to be its attractive fixed point; then $\eps=c\tau+d>1$ is a unit of trace $N+2$. Taking $I_{\tau}=\ZZ\tau+\ZZ$ with orientation given by $(\tau,1)$ gives an ideal over $\ZZ[\eps]$, since multiplication by $\eps$ maps $(\tau,1)$ to $(a\tau+b,c\tau+d)$. We can then view $G$ as the subgroup of $(\QQ\otimes I_{\tau})/I_{\tau}$ fixed by multiplication by~$\eps$, or otherwise, as the group
    \[G=(\tfrac{1}{\eps-1}I_{\tau})/I_{\tau}.\]
For any $z\in G$, we consider the number
\begin{equation}\label{eq:fgam.defalt}
    F_{\gamma}(z) := 
	\Phi_{\gamma,u,0}(z;\tau)\,\mu_{\gamma}\,,
\end{equation}
where $z(\eps^{-1}-1)=u\tau+v$, $u,v\in\ZZ$ (for $z\in I_{\tau}$ we set $F_{\gamma}(z)=\eps^{1/2}$ as before). By Lemma~\ref{lem:lam.inv}, this is well-defined.

\medskip
\subsection{Application: quadratic identities related to Zauner's conjecture}
The result of Theorem~\ref{thm:fg.equs} appears to be the first non-trivial proved family of identities for Stark units, but there have been some explicit conjectural identities discovered recently in connection with Zauner's conjecture (see~\cite{Zau11}) concerning collections of $M^2$ complex equiangular lines in $\CC^M$.\footnote{They are called SIC-POVMs in the quantum information theory literature.} For details, we refer to the papers of Appleby--Flammia--McConnell--Yard~\cite{AFMY17}, Kopp~\cite{Kop21}, Appleby--Flammia--Kopp~\cite{AFK25}, and the references therein. Here we give an explicit twisted convolution identity related to Zauner's conjecture that can be easily derived from the pentagon relation. This identity is equivalent to the principal ideal case of the identity conjectured by Appleby--Flammia--Kopp in~\cite[Conj.~1.35]{AFK25}.

Let $\alpha=(\smat{M-1}{-1}{1}{0})$, $M\ge4$, $\gamma=\alpha^3$ and let $\eps>1$ satisfy $\eps^{2}-(M-1)\eps+1=0$ (note that this is the unit corresponding to~$\alpha$, not~$\gamma$). It will be convenient to work in the above ideal-theoretic formulation, with the special values $F_{\gamma}(z)$ defined by~\eqref{eq:fgam.defalt}.
The relevant group is $G=\frac{1}{\eps^3-1}\ZZ[\eps]/\ZZ[\eps]$, of order $M^2(M-3)$.
Since $\eps^3\equiv 1\Mod{M}$, the $M$-torsion subgroup of~$G$ is just $(\frac{1}{M}\ZZ[\eps])/\ZZ[\eps]$, isomorphic to $(\ZZ/M\ZZ)^2$; we denote it by~$H$.
For the Gaussian $\ang{\cdot}\colon G\to\CC^{\times}$ we have the following useful formulas covering both~$H$ and~$G$:
	\begin{align*}
	\Big\langle\frac{k\eps+l}{M}\Big\rangle &= \ee_M(k^2-kl+l^2)\,,\\
	\Big\langle\frac{a\eps+b}{\eps^{-3}-1}\Big\rangle &= 
	(-1)^{a+b+ab}\ee_{2M(M-3)}\big((M-2)(a^2+b^2+(M-1)ab)\big)\,.
	 \end{align*}
From this we calculate the following formulas for the bicharacter:
	\begin{align*}
	\Big\langle\frac{a_1\eps+a_2}{\eps^{-3}-1};\frac{b_1\eps+b_2}{\eps^{-3}-1}\Big\rangle &= 
	(-1)^{a_1b_2+a_2b_1}\ee_{2M(M-3)}((M-2)(2a_1b_1+2a_2b_2+(M-1)(a_1b_2+a_2b_1)))\\
	&= 
	\ee_{M(M-3)}((M-2)(a_1b_1+a_2b_2)+(a_1b_2+a_2b_1))\,,\\
	\Big\langle\frac{a_1\eps+a_2}{\eps^{-3}-1};\frac{v_1\eps+v_2}{M}\Big\rangle &= 
	\ee_{M(M-3)}(((M-2)a_1+a_2)((2-M)v_2-v_1)+((M-2)a_2+a_1)(v_2+(M-2)v_1))\\
	&=\ee_{M}(a_1v_2-a_2v_1)\,.
    \end{align*}
Note that by Proposition~\ref{prop:powers}, $F_{\gamma}(z)$ has a 3-fold symmetry, since it is invariant under multiplication by $\eps$:
    \begin{equation} \label{eq:fgamma3sym}
    F_{\gamma}(z)=F_{\gamma}(\eps z)\,.
    \end{equation}

Recall from~\S\ref{sec:andersenkashaev} the standard Heisenberg pair $\bp,\bq$ of operators acting on $L^2(\ZZ/M\ZZ)$.
    
\begin{theorem} \label{thm:ghostsic}
Let $\gamma=(\smat{M-1}{-1}{1}{0})^3$, $\eps^2-(M-1)\eps+1=0$, $\eps>1$, and consider the operator
	\[P = \frac{1}{M}\sum_{i,j=0}^{M-1}F_{\gamma}\Big(\frac{i\eps+j}{M}\Big)(-1)^{i+j}\ee_{2M}(i^2+j^2)\bq^i\bp^j\,.\]
	Then
	\begin{equation} \label{eq:operatorquadratic}
    (P-\eps^{1/2})(P+\eps^{-1/2})=0\,.
    \end{equation}
\end{theorem}
\begin{proof}
	Expanding~\eqref{eq:operatorquadratic} in the basis $\bq^r\bp^s$ gives an equivalent identity
	\[\frac{1}{M}\sum_{i,j=0}^{M-1}F_{\gamma}\Big(\frac{i\eps+j}{M}\Big)F_{\gamma}\Big(\frac{(r-i)\eps+(s-j)}{M}\Big)\Big\langle\frac{i\eps+j}{M}\Big\rangle\ee_M(i(-r) + j(r-s)) = \sqrt{M-3}F_{\gamma}\Big(\frac{r\eps+s}{M}\Big)+M\delta(r)\delta(s)\,,\]
	and using the above formula for pairings, we see that it is equivalent to
	\[\frac{1}{M}\sum_{i,j=0}^{M-1}F_{\gamma}\Big(\frac{i\eps+j}{M}\Big)F_{\gamma}\Big(\frac{(r-i)\eps+(s-j)}{M}\Big)\Big\langle\frac{i\eps+j}{M}\Big\rangle\ang{w;\tfrac{i\eps+j}{M}} = \sqrt{M-3}F_{\gamma}\Big(\frac{r\eps+s}{M}\Big)+M\delta(r)\delta(s)\,,\]
    with $w=\tfrac{(r-s)\eps+r}{\eps^{-3}-1}$. Note that we have an exact sequence
		\[0\longrightarrow H\longrightarrow G \xlongrightarrow{\cdot M} A \longrightarrow 0\,,\]
	where $A\cong \ZZ/(M-3)\ZZ$ is generated by $\frac{\eps-1}{M-3}$. Denoting $v=\frac{r\eps+s}{M}$, the above identity can be rewritten as
	\[\frac{1}{M}\sum_{u\in H}F_{\gamma}(u)\ang{u}F_{\gamma}(v-u)\ang{w;u} = \sqrt{M-3}F_{\gamma}(v)+M\delta(v)\]
	For $v=0$, the identity is trivial. For $v\ne 0$, we can rewrite the left hand side as
	\begin{equation} \label{eq:sicmaintrick}
    \begin{aligned}
    \frac{1}{M}&\sum_{u\in H}F_{\gamma}(u)\ang{u}F_{\gamma}(v-u)\ang{w;u} \\
	&= \frac{1}{M(M-3)}\sum_{j=0}^{M-4}\sum_{u\in H}F_{\gamma}(u)\ang{u}F_{\gamma}(v-u)\ang{w+j\tfrac{\eps-1}{M-3};u} \\
	&= \frac{1}{M(M-3)}\sum_{j=0}^{M-4}\sum_{u\in G}F_{\gamma}(u)\ang{u}F_{\gamma}(v-u)\ang{w+j\tfrac{\eps-1}{M-3};u}\\
	&= \frac{1}{\sqrt{M-3}}\sum_{j=0}^{M-4}F_{\gamma}(v)F_{\gamma}(w+j\tfrac{\eps-1}{M-3})\widetilde F_{\gamma}(-v-w-j\tfrac{\eps-1}{M-3})\,,
	\end{aligned}
    \end{equation}
	where in the last step we used~\eqref{eq:pentagonproduct3} with the shorthand $\widetilde{F}_{\gamma}(x)=\ang{x}F_{\gamma}(x)$.
	We claim that
	\[F_{\gamma}(w+j\tfrac{\eps-1}{M-3})= F_{\gamma}(v+w+j\tfrac{\eps-1}{M-3})\,.\]
	Indeed, $(\eps-1)^2=(M-3)\epsilon$, and note that
		\[(\eps-1)w = (\eps-1)\frac{(r-s)\eps+r}{\eps^{-3}-1} = \frac{r\eps+s}{M}+((1-M)r+(M-2)s)\eps+(r-s) \equiv v \Mod{\ZZ[\eps]}\,,\]
	so by~\eqref{eq:fgamma3sym}
	\[F_{\gamma}(w+j\tfrac{\eps-1}{M-3})=F_{\gamma}((w+j\tfrac{\eps-1}{M-3})\eps)= F_{\gamma}(v+w+j\tfrac{\eps-1}{M-3})\,.\]
	Thus, by the reflection property of~$F_{\gamma}$ (Theorem~\ref{thm:fqdilogbasicproperties} (ii)), in the last sum in~\eqref{eq:sicmaintrick} all terms are equal to $F_{\gamma}(v)$, and we get the claim.
\end{proof}

Since $\tr(P)=\eps^{3/2}=(M-1)\eps^{1/2}+(-\eps^{-1/2})$, Theorem~\ref{thm:ghostsic} implies that the $(-\eps^{-1/2})$-eigenspace of $P$ is 1-dimensional. Using the fact that
    \[\ol{F_{\gamma}\Big(\frac{i\eps+j}{M}\Big)} = F_{\gamma}\Big(\frac{i\eps+j}{M}\Big)\Big\langle\frac{i\eps+j}{M}\Big\rangle\,,\]
it is easily checked that the operator~$P$ satisfies $P^*=RPR$, where $Rf(x)=f(-x)$. This gives a factorisation of $P-\eps^{1/2}$ as $(g_M(i)\ol{g_M(-j)})_{i,j}$ using only the eigenvector $g_M\colon \ZZ/M\ZZ\to\CC$, and together with the reflection relation for $F_{\gamma}$, it is then easy to show that $g_M$, suitably normalised, has the following property:
    \[\sum_{j\Mod{M}}g_M(j)g_M(j+k+l)\overline{g_M(-j-k)g_M(-j-l)} = \frac{\delta(k)+\delta(l)}{M+1}\,.\]
This is exactly analogous to the property characterising the so-called fiducial vectors $f_M\colon \ZZ/M\ZZ\to \CC$, whose existence is the subject of Zauner's conjecture:
    \[\sum_{j\Mod{M}}f_M(j)f_M(j+k+l)\overline{f_M(j+k)f_M(j+l)} = \frac{\delta(k)+\delta(l)}{M+1}\,.\]

\begin{remark}
It is straightforward to adapt the proof to get an analogous convolution identity for $\gamma=\alpha^3$ for any~$\alpha$ with trace $\ge3$. This proves~\cite[Conj.~1.35]{AFK25} in full generality; the only reason we do not state this explicitly is the somewhat nontrivial task of matching our notation to~\cite{AFK25}. The above special case already covers all cyclic groups of order $\ge 4$, and the only remaining obstacle to proving Zauner's conjecture in full is to show that there exists~$\sigma\in\Gal(\ol{\QQ}/\QQ)$ for which the values $F_{\gamma}(u)^{\sigma}$, $u\ne 0$, all land on the unit circle.
\end{remark}

\bigskip
\section{Algebraicity of values of finite quantum dilogarithms}
\label{sec:algebraicity}
In this section, we will prove the following result. 
\begin{theorem} \label{thm:algebraicity}
	Let $E\colon G\to \CC$ be a finite quantum dilogarithm on a finite abelian group~$G$. Then $E(u)\in \overline{\QQ}$ for all $u\in G$.
\end{theorem}
We will first give a simple argument that works when $|G|$ is a prime.
For odd $|G|$, we give a $2$-adic proof that extends to even $|G|$ if one assumes additional relations across multiple groups (that in particular hold for the for values of $F_\gamma^\pm$).
For general abelian groups, the proof is more involved, but it essentially follows from known results on fusion categories.
Indeed, thanks to the works of Izumi~\cite{Izu93,Izu01} and Evans--Gannon~\cite{EG17}, we can give an interpretation of a finite quantum dilogarithm as the input data for constructing a certain fusion category. Algebraicity then follows from Ocneanu's rigidity theorem~\cite{ENO05}. As a byproduct of the construction, we get that non-unitary near-group fusion categories of Izumi type exist for all groups~$G$ of the form $\ZZ/n\ZZ\times \ZZ/m\ZZ$.



\subsection{Algebraicity for \texorpdfstring{$N$}{N} prime}

When $|G|$ is prime, there is a particularly simple proof.

\begin{theorem}
	Let $E\colon \ZZ/p\ZZ\to \CC$ be a finite quantum dilogarithm on a cyclic group of prime order. Then $E(u)\in \overline{\QQ}$ for all $u\in \ZZ/p\ZZ$.
\end{theorem}
\begin{proof}
	Let $X_u$ be a variable corresponding to $E(u)$ and consider the variety~$\mathcal{X}$ defined by equations (i)-(iii) in Theorem~\ref{thm:fqdilogbasicproperties}, namely:
	\begin{align*}
		X_0^2&=\sqrt{N}X_0+1\,,\\
		X_uX_{-u}&=\ang{u}^{-1}\,,\qquad u\ne 0\,,\\
		\frac{1}{\sqrt{N}}\sum_{v}X_v\ang{-u;v} &= \lambda \ang{u}X_u\,,\qquad u\in \ZZ/p\ZZ\,.
	\end{align*}
	Consider the projectivisation of this variety, with $X_u=Y_u/Z$, where $(Y_0\colon \cdots\colon Y_{p-1}\colon Z)\in\mathbb{P}^p(\CC)$, and look at its intersection with the hyperplane at infinity, given by $Z=0$. For any point of that intersection we have $Y_0=0$, $Y_uY_{-u}=0$, $u\in\ZZ/p\ZZ$, and $\widehat{Y}_u=\lambda \ang{u}Y_u$, where $\widehat{Y}_u=N^{-1/2}\sum_{v}Y_v\ang{-u;v}$. In particular, the supports of $u\mapsto Y_u$ and $u\mapsto \widehat{Y}_u$ are equal and have size at most $(p-1)/2$. By Tao's uncertainty principle~\cite{Tao05} for Fourier transform on cyclic groups of prime order, this implies that $Y_u=0$ for all $u$, and the intersection of the closure of $\mathcal{X}$ with the hyperplane at infinity is empty. Therefore, the variety~$\mathcal{X}$ is zero-dimensional, and since it is defined over a cyclotomic field, this shows that $E(u)\in\overline{\QQ}$ for all~$u$.
\end{proof}
Note that this already implies algebraicity of Stark units in infinitely many non-trivial cases (namely, whenever $\eps_{\ff}=\eps$ has trace $p+2$, see~\S\ref{sec:starkshintani}).

\subsection{Algebraicity by $2$-adic bounds}
\label{sec:2adic}

For all $G$ of odd order, there is a $2$-adic proof of algebraicity for any finite quantum dilogarithm.
There is also a similar proof that $F_\gamma$ is algebraic when the associated $G$ has even order, which makes use of the additional distribution relations given in Proposition~\ref{prop:dist}.
Similar arguments were also used in a recent work~\cite{Gannon26}.

\begin{theorem}\label{thm:2adic1}
Let $E\colon G\to \ol{\QQ}$ be a finite quantum dilogarithm on $G$ of odd order $N=|G|$. Then $E(u)$ are $2$-adic units for all $u\in G$.
\end{theorem}
\begin{proof}
A specialisation of the $5$-term equation~\eqref{eq:pentagonproduct3} with
$x=2u$, $y=-u$ gives the following:
    \[
    \sqrt{N}E(2u) =
    E(u)^4+
    E(u)^2\sum_{t\in G\sm 0}\ang{t}E(u+t)E(u-t)\,,\qquad u\ne 0\,.
    \]
Let $v_2$ be any valuation above~$2$. If $E(u)$ takes the smallest valuation (i.e., $v_2(E(u))\leq v_2(E(x))$ for all $x\in G$), then $v_2(E(u))\leq v_2(E(2u))=4v_2(E(u))$, since every term in the sum appears twice.
Therefore, we conclude that $v_2(E(u))\geq0$ for all $u$.
Then reflection formula implies that $v_2(E(u)E(-u))=0$ and hence $v_2(E(u))=0$.
\end{proof}
\begin{corollary}\label{cor:2adic}
Theorem~\ref{thm:algebraicity} holds for odd $N$.
\end{corollary}
\begin{proof}
We can project the variety defining $E$ onto any of the coordinate axes $E(u)$ to obtain a constructible set in $\ol{\QQ}$ from Chevalley's theorem.
We have just shown that this constructible set must have zero $2$-adic valuation and hence must consist of a finite set of points (since any infinite constructible subset of $\ol{\QQ}$ is cofinite, and will take arbitrary valuations). Therefore the set of solutions over $\ol{\QQ}$ is finite, and so the variety itself is $0$-dimensional. Since it is defined over a cyclotomic field, this shows that for any solution~$E$ the values~$E(u)$ are algebraic.
\end{proof}
\begin{corollary}
For $E,G$ as in Theorem~\ref{thm:2adic1}, we find that
\[
    E(2u)\equiv \frac{1}{\sqrt{N}}E(u)^4\pmod{2}\,.
\]
\end{corollary}

The above proof does not immediately apply to $G$ of even order, but a similar proof works if we assume additional structure, which is available for the values of~$F_\gamma$. When one considers varieties associated to all groups given by a product of two cyclic groups with a fixed order along with the additional equations coming from Proposition~\ref{prop:dist}, we obtain a similar statement.

\begin{theorem}
Let $\mathcal{X}_N$ be the variety defined by the equations of Theorem~\ref{thm:fg.equs} and Proposition~\ref{prop:dist} for $G_\gamma$ as~$\gamma$ runs over all conjugacy classes with fixed $N=\tr(\gamma)-2=|G_\gamma|$. Then all coordinates of $\mathcal{X}_N(\ol{\QQ})$ are $2$-adic units and $\mathcal{X}_N$ is zero-dimensional.
\end{theorem}
\begin{proof}
For odd $N$, this follows from Theorem~\ref{thm:2adic1}.
We start with the same identity used in the proof of Theorem~\ref{thm:2adic1}.
For even $N$, it takes the more general form
\[
    \sqrt{N}E_\gamma(2u)
    =
    E_\gamma(u)^4
    +\sum_{0\neq x\in G,\;2x=0}\ang{x}E_\gamma(u)^2E_\gamma(u+x)^2
    +
    E_\gamma(u)^2\sum_{x\in G,\;2x\neq0}\ang{x}E_\gamma(u+x)E(u-x)\,.
\]
Notice that the equations from Proposition~\ref{prop:dist} imply that for $2x=0$ we have
\[
E_\gamma(u)E_\gamma(u+x)=E_{\gamma_x}(u_x)
\]
for some $\gamma_x$ with $\tr(\gamma_x)=\tr(\gamma)$ and $u_x\in G_{\gamma_x}$.
Therefore, we see that
\[
    \sqrt{N}E_\gamma(2u)
    =
    E_\gamma(u)^4
    +\sum_{0\neq x\in G,\;2x=0}\ang{x}E_{\gamma_x}(u_x)^2
    +
    E_\gamma(u)^2\sum_{x\in G,\;2x\neq0}\ang{x}E_\gamma(u+x)E(u-x)\,.
\]
Supposing that $v_p(E_\gamma(u))\leq v_p(E_{\gamma'}(u'))$ for all $\tr(\gamma')=\tr(\gamma)$ and $u'\in G_{\gamma'}$, we deduce that
\[
v_2(E_\gamma(u))\leq v_2(E_\gamma(2u))\leq 4v_2(E_\gamma(u))\,,
\]
since $E_{\gamma_x}(u_x)^2$ is only quadratic and in the other sum all terms appear twice.
Therefore, we find that $0\leq v_2(E_\gamma(u))$.
We can conclude by the same arguments as in the end of Theorem~\ref{thm:2adic1} and Corollary~\ref{cor:2adic}.
\end{proof}
Note, in particular, that this proves Theorem~\ref{thm:faddeevqbar}.

\subsection{Fusion categories and their construction via endomorphisms of the Leavitt algebra}
We will work over~$\CC$. Recall (see~\cite{ENO05}) that a fusion category $\lC$ is a $\CC$-linear semisimple rigid tensor category with simple unit object $\mathbf{1}$, finitely many isomorphism classes of simple objects and finite-dimensional spaces of morphisms. Here rigid means that for every object $X$ there is a (left) dual object $X^{*}$ and a choice of evaluation $ev_{X}\colon X^{*}\otimes X\to \mathbf{1}$ and co-evaluation $db_{X}\colon \mathbf{1}\to X\otimes X^{*}$ morphisms satisfying the snake relations $(id_X\otimes ev_X)\circ (db_X\otimes id_X) = id_X$ and $(ev_{X}\otimes id_{X^*})\circ (id_{X^*}\otimes db_{X}) = id_{X^*}$, and similarly there exist right dual objects ${}^*X$ with the corresponding co-evaluation and evaluation morphisms.

The decomposition of tensor products of simple objects of~$\lC$ is encoded by $N_{ij}^{k}$, the multiplicity of $X_k$ in the direct sum decomposition of $X_i\otimes X_j$, where $X_i$, $i\in I$ denotes the set of (equivalence class representatives of) simple objects. The numbers $N_{ij}^{k}$ give the structure constants of the Grothendieck ring $K_0(\lC)$: 
	\begin{equation} \label{eq:fusionrules}
	[X_i][X_j] = \sum_{k}N_{ij}^{k}[X_k]\,.
	\end{equation}
The relations~\eqref{eq:fusionrules} are sometimes called fusion rules.
A ring $R$ with a $\ZZ$-basis $r_i$, indexed by a finite set~$I$ containing $0$, $r_0=1$, and associative multiplication $r_ir_j=\sum_{k}N_{ij}^{k}r_k$ with $N_{ij}^{k}\ge0$, together with an involution $*$ on $I$ and anti-involution $*$ on $R$, such that $r_i^*=r_{i^*}$ and $N_{ij}^{0}=\delta_{i,j^{*}}$ is called a fusion ring. If there is a fusion category $\lC$ whose multiplicity numbers are $N_{ij}^{k}$, it is called a categorification of the fusion ring~$R$. For more details on fusion categories, we refer to~\cite{EGNO17}.

Next, recall that the Leavitt algebra $L_N=L(1,N)$ of type $(1,N)$, $N\ge2$, is the quotient of the free associative algebra $\CC\ang{x_1,\dots,x_N,x_1',\dots,x_N'}$ by the two-sided ideal generated by the relations 
	\begin{equation} \label{eq:leavitt}
	x_i'x_j = \delta_{i,j}\,,\qquad x_1x_1'+\dots+x_Nx_N'=1\,.
	\end{equation}
For a sequence $I=(i_1,\dots,i_k)$, $i_j\in\{1,\dots,N\}$, we set $x_I=x_{i_1}\cdots x_{i_k}$ and $x_I'=x_{i_k}'\cdots x_{i_1}'$. We also set $x_{\varnothing}=x_{\varnothing}'=1$. 

The Leavitt algebra has an anti-linear involutive anti-automorphism $'$, given on generators by $(x_i)' = x_i'$, $(x_j')' = x_j$, and acting on scalars by complex conjugation. We will not assume that endomorphisms respect this anti-automorphism, but if $\beta$ is an endomorphism of $L_N$, then so is $\tilde{\beta}$ defined by $\tilde{\beta}(x)=\beta(x')'$.

We will use the following properties of the Leavitt algebra.
	\begin{lemma} \label{lem:leavitt}
	Let $N\ge2$ and let $L_N$ be the Leavitt algebra defined by~\eqref{eq:leavitt}. Then
	\begin{thmenum}
	\item The set of reduced words $x_Ix_J'$ where $I=(i_1,\dots,i_k)$, $J=(j_1,\dots,j_l)$, and $(i_k,j_l)\ne (1,1)$ if $k,l\ge1$ forms a basis of~$L_N$ as a vector space;
	\item Centre of $L_N$ consists of scalars: $\mathcal{Z}(L_N)=\CC\,1$;
	\item If $y\in L_N$ satisfies $x_iy=yx_j$ for some $i\ne j$, then $y=0$;
    \item The centraliser of an element $u=x_1+P(x_2,\dots,x_N)$, where $P$ is a non-commutative polynomial without a constant term, is $\CC[u]$. The same conclusion holds for $u=x_1'+P(x_2',\dots,x_N')$.
	\end{thmenum}
	\end{lemma}
	\begin{proof}
	Parts (i) and (ii) are well known, see Leavitt's work such as~\cite{Lea65} and its citations. For (iii), note that for all $k\ge0$ we have, by induction,
		\[y=(x_i')^k y x_j^k\,.\]
	Let $y=\sum_{I,J}c_{I,J} x_Ix_J'$ be the decomposition of $y$ into the reduced words basis, and assume that $|I|,|J|<k$ for all pairs $(I,J)$ that occur. Then for all $I,J$ with $c_{I,J}\ne 0$, we have
	\[(x_i')^kx_Ix_J'x_j^k = (x_i')^{k-|I|}x_j^{k-|J|} = 0\,,\]
	by~\eqref{eq:leavitt}, so $y=(x_i')^k y x_j^k = 0$. For part (iv), when $P=0$, this is explained in~\cite[Proof of Proposition~1]{EG17}. For non-zero $P$, simply note that $x_1\mapsto x_1+P(x_2,\dots,x_N)$, $x_1'\mapsto x_1'$, $x_i\mapsto x_i$, $x_i'\mapsto x_i'-x_i'P(x_2,\dots,x_N)x_1'$, $i\ge2$, is an automorphism of $L_N$, so the statement reduces to the case $P=0$. For the last statement, applying the anti-involution reduces it to the statement we have just proved.
	\end{proof}
Let us also note that two unital endomorphisms $\sigma,\tau$ of $L_N$ are equal if and only if $\sigma(x_i)=\tau(x_i)$ for $i=1,\dots,N$. Necessity is evident, and for sufficiency, simply note that if $\sigma(x_i)=\tau(x_i)$ for all~$i$, then
	\begin{equation}\label{eq:completenesstrick}
	\sigma(x_i') = \sum_j\sigma(x_i')\tau(x_jx_j') = \sum_j\sigma(x_i')\tau(x_j)\tau(x_j') = \sum_j\sigma(x_i'x_j)\tau(x_j') = \tau(x_i')\,.
	\end{equation}

Following Evans--Gannon~\cite{EG17}, we construct tensor categories from systems of endomorphisms of~$L_N$. First, consider the endomorphism category of $L_N$, whose objects are unital endomorphisms $\beta\colon L_N\to L_N$ and whose morphisms are elements of the intertwiner spaces:
	\[\Hom(\sigma,\tau) = \{x\in L_N\colon x\sigma(y) = \tau(y)x\mbox{ for all }y\in L_N\}\,.\]
The monoidal structure is given by composition of endomorphisms on objects, and on morphisms, if $r\in \Hom(\beta,\gamma)$ and $s\in\Hom(\sigma,\tau)$, we set
	\[r\otimes s = r\beta(s) = \gamma(s)r \in \Hom(\beta\circ\sigma,\gamma\circ\tau)\,.\]

To construct a tensor category, Evans and Gannon consider a set $\mathcal{E}$ of unital endomorphisms of $L_N$ that contains the identity and is closed under composition, apply idempotent completion to it, and then take formal direct sums. Here is their construction in more detail: First one considers the full monoidal subcategory $\lC(\mathcal{E})$ of the endomorphism category of $L_N$. It is a strict pre-additive tensor category. 
Next, one considers the category $\ol{\mathcal{C}(\mathcal{E})}$:
\begin{itemize}
	\item Objects of $\ol{\mathcal{C}(\mathcal{E})}$: pairs $(p,\beta)$, where $\beta$ is an element of $\mathcal{E}$ and $p^2=p\in \Hom(\beta,\beta)$ is an idempotent.
	\item Morphisms: $\Hom((p,\beta),(q,\gamma)) = q\Hom(\beta,\gamma)p$.
	\item Tensor product: $(p,\beta)\otimes (q,\gamma) = (p\beta(q), \beta\circ \gamma)$.
\end{itemize}
Finally, one adds formal direct sums to form $\ol{\lC(\mathcal{E})}^{ds}$:
\begin{itemize}
	\item Objects of $\ol{\lC(\mathcal{E})}^{ds}$: formal finite direct sums $\bigoplus_{i}(p_i,\beta_i)$.
	\item Morphisms: matrices of intertwiners $\Hom(\bigoplus_{j=1}^{m}(p_j,\beta_j),\bigoplus_{i=1}^{n}(q_i,\gamma_i)) = (q_i\Hom(\beta_j,\gamma_i)p_j)_{i=1,j=1}^{n,m}$. Composition is given by matrix product.
	\item Tensor product: tensor product of $\lC(\mathcal{E})$ extended as Kronecker product to matrices.
\end{itemize}
The following result is \cite[Lemma 2]{EG17}.
\begin{proposition} \label{prop:evansgannon}
Let $\mathcal{E}$ be a set of endomorphisms of $L_N$, $N\ge2$. Assume that $id_{L_N}\in\mathcal{E}$, $\mathcal{E}$ is closed under composition, $\dim_{\CC}\Hom(\beta,\gamma)<\infty$ and $\Hom(\beta,\gamma) = \Hom(\tilde\beta,\tilde\gamma)$ for all $\beta,\gamma\in\mathcal{E}$. Then $\ol{\lC(\mathcal{E})}^{ds}$ is a semisimple strict $\CC$-linear tensor category with finite dimensional spaces of morphisms. Moreover, if $\lC(\mathcal{E})$ is rigid, then so is $\ol{\lC(\mathcal{E})}^{ds}$.
\end{proposition}
With this result, if one can in addition show that there are finitely many isomorphism classes of simple objects and define compatible duality and (co-)evaluation maps, one gets a fusion category. The rigidity data for~$\lC(\mathcal{E})$ is supplied by a choice of $e_{\beta}\in\Hom(\beta^{*}\beta,1)$ and $b_{\beta}\in\Hom(1,\beta\beta^{*})$ satisfying
	\[\beta(e_{\beta})b_{\beta}=1\,,\qquad \qquad e_{\beta}\beta^{*}(b_{\beta})=1\,.\]

\subsection{Non-unitary Izumi fusion categories}
\label{sec:izumi}
Let $G$ be a finite abelian group of order $|G|=N$. Following Izumi~\cite{Izu93}, consider a fusion ring~$R_G$ with the index set
	\[I=\{g\colon g\in G\}\cup \{\rho\}\]
and fusion rules
	\[[g][h] = [g+h]\,,\qquad [g][\rho] = [\rho][g]=[\rho]\,,
	\qquad [\rho]^2 = N[\rho] + \sum_{g\in G}[g]\,.\]
(The involution is given by $[g]^{*}=[-g]$ and $[\rho]^{*}=[\rho]$.) We will call $R_G$ the Izumi fusion ring of~$G$. What follows largely repeats~\cite[\S5]{Izu01}, but since we do not impose the unitary structure, we need to realise the fusion category using endomorphisms of the Leavitt algebra $L_{2N}$.

\textbf{Construction.}
The input data for Izumi's construction\footnote{In~\cite{Izu01} the notation is $a,b,c,d$, with $a(x)=\ang{x}$, $b(x)=E(x)/\sqrt{N}$, $c=\lambda$, $d=d$, and the pairing is $\ang{x;y}'=\ang{x;y}^{-1}$.} is a metric group $(G,\ang{\cdot})$ of order $N=|G|$, a finite quantum dilogarithm $E\colon G\to \CC^{\times}$ with $\widehat{E}(x)=\lambda\ang{x}E(x)$ as in Theorem~\ref{thm:fqdilogbasicproperties}~(iii). We also denote $d=-\sqrt{N}/E(0)$, so that $d^2=N(d+1)$, and we fix a choice of $\sqrt{d}$.

Let $L_{2N}$ be the Leavitt algebra of rank $2N$ (over $\CC$), whose generators are $\{S_g\}_{g\in G}$, $\{S_g'\}_{g\in G}$, $\{T_g\}_{g\in G}$, $\{T_g'\}_{g\in G}$, subject to the relations
\begin{equation} \label{eq:leavitt2Na}
	S_g'S_h = \delta_{g,h}1\,,\qquad\qquad
	T_g'T_h = \delta_{g,h}1\,,\qquad\qquad
	S_g'T_h = T_h'S_g=0\,,
\end{equation}
for all $g,h\in G$, together with
	\begin{equation} \label{eq:leavitt2Nb}
	\sum_{g\in G}S_gS_g' + \sum_{g\in G}T_gT_g' = 1\,.
	\end{equation}

The construction is done in several steps. First, for all $g\in G$ we define an automorphism $\alpha_{g}\colon L_{2N}\to L_{2N}$:
\begin{equation}\label{eq:alphag}
	\begin{aligned}
	\alpha_g(S_h) &= S_{g+h}\,, &&\alpha_g(T_h) = \ang{g;h}^{-1}T_{h}\,,\\ 
	\alpha_g(S_h') &= S_{g+h}'\,, &&\alpha_g(T_h') = \ang{g;h}T_{h}'\,.
	\end{aligned}
\end{equation}
These are evidently well-defined automorphisms of $L_{2N}$. 

Next, we define a representation $U\colon G\to L_{2N}^{\times}$ by
	\[U_g = \sum_{h}\ang{g;h}^{-1}S_hS_h' + \sum_{h}T_{h-g}T_h'\,.\]
It is evident by~\eqref{eq:leavitt2Nb} that $U_0=1$, so we only need to verify that $U_{g_1}U_{g_2}=U_{g_1+g_2}$:
	\begin{align*}
	U_{g_1}U_{g_2} &= \Big(\sum_{h}\ang{g_1;h}^{-1}S_hS_h' + \sum_{h}T_{h-g_1}T_h'\Big)
	 								\Big(\sum_{h}\ang{g_2;h}^{-1}S_hS_h' + \sum_{h}T_{h}T_{h+g_2}'\Big)\\
	   &= \sum_{h}\ang{g_1+g_2;h}^{-1}S_hS_h' + \sum_{h}T_{h-g_1}T_{h+g_2}' = U_{g_1+g_2}\,,
	\end{align*}
where we used the Leavitt relations~\eqref{eq:leavitt2Na}. We also note that
	\begin{equation} \label{eq:alphagUg}
	\alpha_{g_1}(U_{g_2}) = \sum_{h}\ang{g_2;h}^{-1}S_{g_1+h}S_{g_1+h}' + \sum_{h}\ang{g_1;h-g_2}^{-1}\ang{g_1;h}T_{h-g_2}T_h' = \ang{g_1;g_2}U_{g_2}\,.
	\end{equation}

The main part of the construction is the endomorphism associated to $\rho$. The definition of the action of~$\rho$ on the generators~$S_g$ and~$S_g'$ is
\begin{align*}
	\rho(S_g) &= U_g\Big(\frac{1}{d}\sum_{h}S_h + \frac{1}{\sqrt{d}}\sum_{h}\ang{h}T_hT_{-h}\Big)U_{-g} = \Big(\frac{1}{d}\sum_{h}\ang{g;h}^{-1}S_h + \frac{1}{\sqrt{d}}\sum_{h}\ang{h}T_{h-g}T_{-h}\Big)U_{-g}\,,\\ 
	\rho(S_g') &= U_g\Big(\frac{1}{d}\sum_{h}S_h' + \frac{1}{\sqrt{d}}\sum_{h}\ang{h}^{-1}T_{-h}'T_{h}'\Big)U_{-g} = U_g\Big(\frac{1}{d}\sum_{h}\ang{g;h}S_h' + \frac{1}{\sqrt{d}}\sum_{h}\ang{h}^{-1}T_{-h}'T_{h-g}'\Big)\,.
\end{align*}
The definition of the action of $\rho$ on the generators $T_h$ uses the values of~$E$:
\begin{align*}
	\rho(T_g) =\; &\frac{\lambda}{\sqrt{d}\sqrt{N}}\sum_{h,k\in G}\ang{h-g;k}S_hT_k'
	+\frac{\lambda^{-1}}{\ang{g}\sqrt{N}}\sum_{h,k\in G}\ang{g+h;k}^{-1}T_kS_hS_h'\\
	+&\frac{1}{\sqrt{N}}\sum_{h,k\in G}\ang{h}E(h+g)\ang{g;k}^{-1}T_{h+k}T_{-h}T_k'
\end{align*}
and
\begin{align*}
	\rho(T_g') =\; &\frac{\lambda^{-1}}{\sqrt{d}\sqrt{N}}\sum_{h,k\in G}\ang{h-g;k}^{-1}T_kS_h'
	+\frac{\lambda \ang{g}}{\sqrt{N}}\sum_{h,k\in G}\ang{g+h;k}S_hS_h'T_k'\\
	+&\frac{1}{\sqrt{N}}\sum_{h,k\in G}\frac{\ang{g+h}}{\ang{h}}E(-h-g)\ang{g;k}T_kT_{-h}'T_{h+k}'\,.
\end{align*}

In Appendix~\ref{sec:appendixBwelldefined}, we check that $\rho$ is indeed a unital endomorphism of~$L_{2N}$.

\textbf{Compositions with $\alpha_g$.} Next, we would like to show that $\alpha_g\circ \rho=\rho$ and $\rho(\alpha_g(x)) = U_g\rho(x)U_{-g}$. The property $\alpha_g\circ \rho=\rho$ is straightforward from definitions using~\eqref{eq:alphagUg}. The property $\rho(\alpha_g(x))=U_g\rho(x)U_{-g}$ is obvious for $x=S_r$ and $x=S_r'$. For $x=T_r$ it easily follows from
	\begin{align*}
	U_{g}S_hT_{k}'U_{-g} &=  \ang{g;h}^{-1}S_hT_{k-g}'\,,\\
	U_{g}T_{k}S_hS_h'U_{-g} &=  \ang{g;h}T_{k-g}S_hS_h'\,,\\
	U_{g}T_{h+k}T_{-h}T_{k}'U_{-g} &=  T_{h+k-g}T_{-h}T_{k-g}'\,,
	\end{align*} 
and similarly for $x=T_r'$.

\textbf{The fusion rule for $\rho^2$.} The most important part of the construction is to verify the fusion rule\footnote{This exactly tells us that in $\ol{\lC(\mathcal{E})}^{ds}$ one has $\rho^2\cong \bigoplus_{g}\alpha_g\oplus \rho^{\oplus N}$.}
	\begin{equation} \label{eq:rho2fusion1}
	\rho^2(x) = \sum_{g\in G}S_g\alpha_g(x)S_g' + \sum_{h\in G}T_h\rho(x)T_h'\,.
	\end{equation}
This is also the part that crucially uses the pentagon relation~\eqref{eq:pentagonintegral}. The proof is essentially that of Izumi~\cite{Izu01}, but because we work in the purely algebraic setting, several $C^{*}$-algebraic shortcuts are unavailable, making the arguments longer. We provide the details in Appendix~\ref{sec:appendixBfusion}.

\textbf{Duals and (co-)evaluation morphisms.}
We define the dual (both left and right) of $\alpha_g$ to be $\alpha_{-g}$ and set both evaluation and co-evaluation morphisms to be $e_{\alpha_g}=b_{\alpha_g}=1$ (since the relevant hom-space is $\Hom(1,1)=\CC 1$). The object $\rho$ is self-dual, and since $S_0\in\Hom(1,\rho^2)$ and $S_0'\in \Hom(\rho^2,1)$, we set $e_{\rho}=dS_0'$, $b_{\rho}=S_0$, the snake relations are then simply
	\[\rho(dS_0')S_0 = dS_0'\rho(S_0) = 1\,.\]

\textbf{Intertwiner spaces.} The only remaining non-trivial check is to compute the full intertwiner spaces between the basic endomorphisms. First, $\Hom(\alpha_g,\alpha_g)=\CC 1$, since $\mathcal{Z}(L_{2N})=\CC 1$. For $h\ne g$ we have $\Hom(\alpha_g,\alpha_h)=0$, since if $y\in \Hom(\alpha_g,\alpha_h)$, then $yS_{g}=S_{h}y$, and by Lemma~\ref{lem:leavitt} (iii) we get $y=0$. Next, if $y\in \Hom(\rho,\rho)$, then $y$ is a centraliser of both $\rho(S_0)$ and $\rho(S_0')$, and by Lemma~\ref{lem:leavitt}~(iv), this means that $y\in \CC[\rho(S_0)]\cap \CC[\rho(S_0')]=\CC 1$, so $\Hom(\rho,\rho)=\CC 1$.

Next, we have $\Hom(\rho,\rho)\cong \Hom(1,\rho^2)$, with the two maps given by $x\mapsto \rho(x)S_0$ and $y\mapsto dS_0'\rho(y)$. Therefore, $\Hom(1,\rho^2)=\CC S_0$.
If $y\in \Hom(1,\rho)$, then $T_0y\in \Hom(1,\rho^2)=\CC S_0$, and multiplying by $T_0'$ gives $y=0$. Similarly, we get $\Hom(1,\rho)=\Hom(\rho,1)=0$, and applying $\alpha_g$ gives $\Hom(\alpha_g,\rho)=\Hom(\rho,\alpha_g)=0$. Therefore $\Hom(\alpha_g,\alpha_g)=\CC1$, $\Hom(\rho,\rho)=\CC1$, and all other intertwiner spaces between these endomorphisms are trivial.

\medskip

Note that, because of the exact identity $\alpha_g\circ \rho=\rho$, all endomorphisms generated by $\rho$ and $\alpha_g$ are of the form $\rho^m\alpha_g$. For $m=0$ they are just $\alpha_g$, and for $m=1$, they are all isomorphic to $\rho$ because of $\rho\alpha_g=\mathrm{Ad}(U_g)\rho\cong \rho$. For $m=2$, Equation~\eqref{eq:rho2fusion} implies that
	\[\rho^2\alpha_g(x) = \sum_{h}S_h\alpha_{g+h}(x)S_h' + \sum_h(T_hU_g)\rho(x)(U_{-g}T_h')\,,\]
giving a direct sum decomposition for $\rho^2\alpha_g$. For $m>2$, Equation~\eqref{eq:rho2fusion} implies that
	\[\rho^m\alpha_g(x) = \sum_{k}S_k\rho^{m-2}\alpha_g(x)S_k' + \sum_hT_h\rho^{m-1}\alpha_g(x)T_h'\,,\]
so---by induction---all endomorphisms in $\mathcal{E}=\{\rho^m\alpha_g\colon m\ge0, g\in G\}$ are split as direct sums of $\alpha_g$ and $\rho$. This proves that there are finitely many simple objects, the finite-dimensionality of morphism spaces, and, by the computation of hom spaces between $\alpha_g$ and $\rho$, semisimplicity (after adjoining finite direct sums). Rigidity also extends from simple objects (via Proposition~\ref{prop:evansgannon}). We have thus proved the following:
\begin{theorem}
	The system of endomorphisms $\mathcal{E}=\{\rho^m\alpha_g\colon m\ge0, g\in G\}$ gives rise, via Evans--Gannon construction, to a categorification of the Izumi fusion ring $R_G$.
\end{theorem}

From Corollary~\ref{cor:existence}, we immediately get.
\begin{corollary}
	Non-unitary Izumi fusion categories exist for $G=\ZZ/n\ZZ\times \ZZ/m\ZZ$ for all $n,m\ge1$.
\end{corollary}

\medskip
\subsection{Reconstruction and the proof of algebraicity}
Let us show that the function $E\colon G\to\CC$ can be recovered from the (strict) fusion category itself, up to an automorphism of~$G$. For this, we will first see that the basis $T_g\in \Hom(\rho,\rho^2)$ can be uniquely reconstructed up to a simultaneous rescaling. Pick any isomorphisms $a_g\colon \alpha_g\otimes \rho \to \rho$ and let $T_0\ne0$ satisfy
    \[(a_g\otimes1_{\rho})(1_{\alpha_g}\otimes T_0) = T_0a_g\]
for all $g\in G$. In the above construction, this corresponds to $\alpha_g(T_0)=T_0$. Since $\alpha_g(T_h)=\ang{g;h}^{-1}T_h$, $T_0\in\Hom(\rho,\rho^2)$ is well-defined up to rescaling.

Next, choose an isomorphism $u_g\colon \rho \to \rho\otimes \alpha_g$ in such a way that $U_g\in\Hom(\rho\otimes\rho,\rho\otimes\rho)$, defined by
    \[U_g = (1_{\rho}\otimes a_g)(u_g\otimes 1_{\rho})\,,\]
acts identically on $\Hom(1,\rho^2)$. This corresponds to $U_gS_0=S_0$ and~$U_g$ is canonically defined (it is well-defined up to rescaling, and the action on $\Hom(1,\rho^2)$ pins down the scaling). Then we simply define $T_g = U_{-g}T_0$. Note that the whole family~$T_g$ is well-defined up to rescaling. We may then define the dual basis $T_h'\in \Hom(\rho^2,\rho)$ to satisfy the relation $T_h'T_g=\delta_{g,h}1_{\rho}$. Then it is easy to compute that
	\[\frac{E(g)}{\sqrt{N}}1_{\rho} = T_0'(T_0'\otimes 1_{\rho})(1_{\rho}\otimes T_g)T_0\,.\]
Since the right hand side does not depend on the choice of scalar in the definition of $T_g$, $T_g'$, we see that the values of~$E$ are canonically defined from the category, and the only freedom is in relabelling the group elements using an automorphism.

By Ocneanu's rigidity theorem~\cite[Theorem~2.28]{ENO05}, there exist at most finitely many inequivalent fusion categories with a given fusion ring. Since we have just showed that the values of~$E$ are invariants of the category itself, there are only finitely many solutions to the system of equations defining~$E$, thus all values of~$E$ are algebraic. This concludes the proof of Theorem~\ref{thm:algebraicity}.

\section{Concluding remarks}
\label{sec:conclusions}
We close the paper with a brief overview of interesting facts and questions, some of which we intend to investigate in a sequel to this paper.

\medskip

\noindent $\bullet$ {\bf Realness and unitarity.}
All finite quantum dilogarithms that we have constructed satisfy $E(0)>0$ and $\overline{E(x)} = \ang{x}E(x)$. Taking their Galois conjugates satisfying~$E(0)<0$, as far as we can check, always gives~$E$ for which $\overline{E(x)} = \ang{-x}E(-x)$ (in particular these conjugates satisfy~$|E(x)|=1$ for $x\ne0$). We refer to the two types of behaviour with respect to complex conjugation as real, if $\overline{E(x)} = \ang{x}E(x)$, and unitary, if $\overline{E(x)} = \ang{-x}E(-x)$. That unitary solutions should exist (for $G=\ZZ/n\ZZ\times \ZZ/m\ZZ$) would follow from one of the predictions of Stark's conjectures, but it would be interesting to see if all finite quantum dilogarithms are either real or unitary. This question is closely related to the distinction between unitary and pseudo-unitary fusion categories.

\medskip

\noindent $\bullet$ {\bf Invariants of 3-manifolds.} The basic structure of a finite quantum dilogarithm can be used to define $3$-manifold invariants, mapping class group representations etc.. These TQFTs interpolate the usual $\mathrm{SU}(2)$ Witten-Reshetikhin-Turaev invariants to real quadratic levels. This is of course illustrated in their links to fusion categories. A further study of these TQFT may be of independent interest.

\medskip

\noindent $\bullet$ {\bf Artin map and Galois action.}
Part of Stark's conjectures that we have not touched upon here predicts a precise rule for how the Artin map should act on the values of $F_\gamma^{\pm}$. One can hope for a situation analogous to the usual theory of complex multiplication. One of the difficulties in using our results to prove this lies in the general lack of understanding of properties of solutions of~\eqref{eq:pentagonintegral} with values in $\CC$ and in non-Archimedean local fields.

\newpage
\appendix
\section{Some standard identities and a proof}
\subsection{\texorpdfstring{$q$}{q}-Pochhammer symbol}
\label{sec:qpochhammer}
We define
	\begin{align*}
	(x;q)_\infty&:=\prod_{j=0}^{\infty}(1-xq^j)\,,\qquad &&|q|<1\,,\\
	(x;q)_\infty&:=\prod_{j=1}^{\infty}\frac{1}{1-xq^{-j}}
	=\frac{1}{(xq^{-1};q^{-1})_\infty}\,,\qquad &&|q|>1\,.
	\end{align*}
For \(n\in\ZZ\), we define
	\begin{equation}
	(x;q)_n:=\frac{(x;q)_\infty}{(xq^n;q)_\infty}.
	\label{eq:qpoch-finite}
	\end{equation}
Equivalently,
	\begin{align}
	(x;q)_n
	&=\prod_{j=0}^{n-1}(1-xq^j),
	&&n\geq0,
	\label{eq:qpoch-positive}\\
	(x;q)_{-k}
	&=\prod_{j=1}^{k}\frac{1}{1-xq^{-j}}=\frac{1}{(xq^{-k};q)_k}
	=\frac{(-1)^kx^{-k}q^{\binom{k+1}{2}}}
	{(x^{-1}q;q)_k},
	&&k\geq1.
	\label{eq:qpoch-negative}
	\end{align}
In particular, the finite product in~\eqref{eq:qpoch-positive} is only meant for non-negative indices. 
For $m,n\in\ZZ$, the extended symbol satisfies
\begin{align}
	(x;q)_{m+n}
	&=(x;q)_m(xq^m;q)_n,
	\label{eq:qpoch-splitting}\\
	(x;q^{-1})_n
	&=(-1)^nx^nq^{-\binom n2}(x^{-1};q)_n,
	\label{eq:qpoch-base-inversion}\\
	(x;q)_{-n}
	&=\frac{1}{(xq^{-n};q)_n}.
	\label{eq:qpoch-reflection}
\end{align}

The following $q$-expansions are very useful:
\begin{align}
	(x;q)_\infty
	&=\sum_{n=0}^{\infty}
	\frac{(-1)^nq^{\binom n2}x^n}{(q;q)_n},
	&& |x|<\max(1,|q|),
	\label{eq:euler-first}\\
	\frac{1}{(x;q)_\infty}
	&=\sum_{n=0}^{\infty}\frac{x^n}{(q;q)_n},
	&&|x|<1,|q|<1 \mbox{ or }|q|>1, \label{eq:euler-second}\\
	\log(x;q)_\infty &=-\sum_{n=1}^{\infty}\frac{x^n}{n(1-q^n)} =-\Li_2(x;q).
	&&|x|<\max(|q|,1).
	\label{eq:qpoch-log}
\end{align}
The $q$-binomial theorem states that
	\begin{equation}
	\sum_{n=0}^{\infty}
	\frac{(a;q)_n}{(q;q)_n}x^n
	=\frac{(ax;q)_\infty}{(x;q)_\infty}.
	\label{eq:q-binomial-series}
	\end{equation}
Finally, for $|q|<1$, Jacobi's triple product identity states that
\begin{equation}\label{eq:jac.trip}
    \sum_{n\in\ZZ}(-1)^nq^{n(n+1)/2}x^n
=(qx;q)_\infty(x^{-1};q)_\infty(q;q)_\infty\,.
\end{equation}

\subsection{Proof of Theorem \texorpdfstring{\ref{thm:5term.mod.fad}}{3}}\label{app:mod.fad}

In this appendix, we detail the proof of Theorem~\ref{thm:5term.mod.fad}.
It uses the residue and the $q$-binomial theorems (see Equation~\eqref{eq:q-binomial-series}).

\begin{proof}[Proof of Theorem~\ref{thm:5term.mod.fad}]
Note that the contours are to the left of all the poles of $\Phi_{\gamma,m+1,n}$.
We will prove the second equation, which implies the first after taking the analytic continuation and the limit $y\to i\infty$.
We will prove this result by analytic continuation and the residue theorem.
Suppose that $c>0$ (a similar proof works when $c<0$).
Take $\Im(\tau)>0$, then take a slightly rotated contour $\calC$.
The integral will be convergent and we can use Equation~\eqref{eq:phigam.def} to compute $\Phi_{\gamma,m,n}(z;\tau)$.
Moreover, the asymptotics for $\arg(c\tau+d)<\arg(z)<\pi$ and $\arg(-c\tau-d)<\arg(z)<0$ are the same as in Lemma~\ref{lem:asymp}.
In the region $0<\arg(z)<\arg(c\tau+d)$, the asymptotics are governed by a quotient of $\theta$-functions and are again convergent.
Therefore, we can apply Jordan's lemma to the cones to compute this integral via residues.
We find that Equation~\eqref{eq:phigam.def} implies that
\[
\begin{aligned}
&\sum_{m\in\ZZ/c\ZZ}\int_{\calD(y)}\frac{\Phi_{\gamma,m+1,n}(z;\tau)}{\Phi_{\gamma,m+p,n}(z+y;\tau)}\ee\Big(\frac{(cz-n+(c\tau+d)m)w}{c\tau+d}+\ell(z+m\tau)+h\Big(\frac{z}{c\tau+d}+n\frac{a\tau+b}{c\tau+d}\Big)\Big)dz\\
&=
\sum_{m\in\ZZ/c\ZZ}\int_{\calD(y)}\frac{(q^{m+1}\ee(z);q)_\infty(\tq^{n}\ee(\frac{z+y}{c\tau+d});\tq)_\infty}{(\tq^{n}\ee(\frac{z}{c\tau+d});\tq)_\infty(q^{m+p}\ee(z+y);q)_\infty}\\
&\qquad\qquad\qquad\qquad\times\ee\Big(\frac{(cz-n+(c\tau+d)m)w}{c\tau+d}+\ell(z+m\tau)+h(\frac{z}{c\tau+d}+n\frac{a\tau+b}{c\tau+d})\Big)dz
\end{aligned}
\]
We can compute this integral by Jordan's lemma and the residue theorem to find that it equals
\[
\begin{aligned}
&-2\pi i\sum_{m\in\ZZ/c\ZZ}\sum_{i,j,k}
\underset{z=i(c\tau+d)+j\tau+k}{\mathrm{Res}}\frac{(q^{m+1}\ee(z);q)_\infty(\tq^{n}\ee(\frac{z+y}{c\tau+d});\tq)_\infty}{(\tq^{n}\ee(\frac{z}{c\tau+d});\tq)_\infty(q^{m+p}\ee(z+y);q)_\infty}\\
&\qquad\qquad\qquad\qquad\times\ee\Big(\frac{(cz-n+(c\tau+d)m)w}{c\tau+d}+\ell(z+m\tau)+h(\frac{z}{c\tau+d}+n\frac{a\tau+b}{c\tau+d})\Big)dz\\
&=
-2\pi i\sum_{m\in\ZZ/c\ZZ}\sum_{i,j,k}
\underset{z=0}{\mathrm{Res}}\frac{(q^{m+1+ci+j}\ee(z);q)_\infty(\tq^{n+dj-ck}\ee(\frac{z+y}{c\tau+d});\tq)_\infty}{(\tq^{n+dj-ck}\ee(\frac{z}{c\tau+d});\tq)_\infty(q^{m+ci+j+p}\ee(z+y);q)_\infty}q^{\ell(ci+j+m)}\tq^{h(n+dj-ck)}\\
&\qquad\qquad\qquad\qquad\times\ee\Big(\frac{(cz+(ck-n-dj)+(c\tau+d)(m+ci+j))w}{c\tau+d}+\ell z+h\frac{z}{c\tau+d}\Big)dz\\
&=
-2\pi i
\sum_{m\in\ZZ/c\ZZ}\sum_{i,j,k}
\underset{z=0}{\mathrm{Res}}\frac{(q\ee(z);q)_\infty(\ee(\frac{z+y}{c\tau+d});\tq)_\infty}{(\tq\ee(\frac{z}{c\tau+d});\tq)_\infty(q^p\ee(z+y);q)_\infty}
\frac{q^{\ell(ci+j+m)}\tq^{h(n+dj-ck)}}{1-\ee(\frac{z}{c\tau+d})}\\
&\qquad\qquad\qquad\qquad\times\frac{(q^p\ee(z+y);q)_{ci+j+m}(\tq^{-1}\ee(\frac{z+y}{c\tau+d});\tq^{-1})_{ck-dj-n}}{(q\ee(z);q)_{ci+j+m}(\tq^{-1}\ee(\frac{z}{c\tau+d});\tq^{-1})_{ck-dj-n}}\\
&\qquad\qquad\qquad\qquad\times\ee\Big(\frac{(cz+(ck-n-dj)+(c\tau+d)(m+ci+j))w}{c\tau+d}\Big)dz\\
&=
\sum_{m\in\ZZ/c\ZZ}(c\tau+d)
\frac{(q;q)_\infty(\ee(\frac{y}{c\tau+d});\tq)_\infty}{(\tq;\tq)_\infty(q^p\ee(y);q)_\infty}
\sum_{i,j,k}
\frac{(q^p\ee(y);q)_{ci+j+m}(\tq^{-1}\ee(\frac{y}{c\tau+d});\tq^{-1})_{ck-dj-n}}{(q;q)_{ci+j+m}(\tq^{-1};\tq^{-1})_{ck-dj-n}}\\
&\qquad\qquad\qquad\qquad\times q^{\ell(ci+j+m)}\tq^{h(n+dj-ck)}\ee\Big(\frac{(ck-n-dj)w}{c\tau+d}+(m+ci+j)w\Big)
\end{aligned}
\]
We can then use the fact that $\sum_k\ee(jk/c)=c\,\delta_{j,0}$ to rewrite the sum and apply the $q$-binomial theorem.
We obtain that the original integral is equal to
\[
\begin{aligned}
&\sum_{m,s,r\in\ZZ/c\ZZ}(c\tau+d)
\frac{(q;q)_\infty(\ee(\frac{y}{c\tau+d});\tq)_\infty}{(\tq;\tq)_\infty(q^p\ee(y);q)_\infty}
\sum_{i,j,k}
\frac{q^{\ell i}\tq^{-hk}(q^p\ee(y);q)_{i}(\tq^{-1}\ee(\frac{y}{c\tau+d});\tq^{-1})_{k}}{(q;q)_{i}(\tq^{-1};\tq^{-1})_{k}}\\
&\qquad\qquad\qquad\qquad\times \ee\Big(iw+k\frac{w}{c\tau+d}+\frac{r(i-j-m)+s(k+dj+n)}{c}\Big)\\
&=
\sum_{m,s\in\ZZ/c\ZZ}\frac{c\tau+d}{c}
\frac{(q;q)_\infty(\ee(\frac{y}{c\tau+d});\tq)_\infty}{(\tq;\tq)_\infty(q^p\ee(y);q)_\infty}
\frac{(\tq^{1-h}\ee(\frac{w}{c\tau+d}+\frac{s}{c});\tq)_{\infty}(q^{\ell+p}\ee(w+y+\frac{ds}{c});q)_{\infty}}{(q^\ell\ee(w+\frac{ds}{c});q)_{\infty}(\tq^{-h}\ee(\frac{w+y}{c\tau+d}+\frac{s}{c});\tq)_{\infty}}\ee\Big(\frac{sn-smd}{c}\Big)\\
&=
(c\tau+d)
\frac{(q;q)_\infty(\ee(\frac{y}{c\tau+d});\tq)_\infty(\tq^{1-h}\ee(\frac{w}{c\tau+d});\tq)_{\infty}(q^{\ell+p}\ee(w+y);q)_{\infty}}{(\tq;\tq)_\infty(q^p\ee(y);q)_\infty(q^\ell\ee(w);q)_{\infty}(\tq^{-h}\ee(\frac{w+y}{c\tau+d});\tq)_{\infty}}\\
&=
(c\tau+d)\frac{(q;q)_\infty}{(\tq;\tq)_\infty}
\frac{\Phi_{\gamma,p+\ell,-h}(w+y;\tau)}{\Phi_{\gamma,p,0}(y;\tau)\Phi_{\gamma,\ell,1-h}(w;\tau)}\,.
\end{aligned}
\]

\noindent This identity then analytically continues to $c\tau+d\in\RR_{>0}$.
\end{proof}

\section{Properties of the endomorphism $\rho$}
\label{sec:appendixB}

In this appendix, we will give the omitted checks from~\S\ref{sec:izumi} that~$\rho$ is a well-defined endomorphism of the Leavitt algebra and that~$\rho^2$ satisfies the stated fusion rule. Recall that $E\colon G\to\CC$ is a finite quantum dilogarithm, $\lambda$ is defined by $\widehat{E}(u)=\lambda \ang{u}E(u)$, and $d=-\sqrt{N}/E(0)$ satisfies $d^2=Nd+N$. In addition to the properties of~$E$ stated in Theorem~\ref{thm:fqdilogbasicproperties}, we will make use of the relation
    \begin{equation} \label{eq:qdilogconvolution}
	\sum_{h}\ang{h+g_1}E(-h-g_1)E(h+g_2) = N\delta_{g_1,g_2}-\frac{N}{d}
    \end{equation}
that directly follows from~\eqref{eq:reflection1} together with the Fourier transform formula for~$E$. We keep the rest of notation as in~\S\ref{sec:izumi}.

\subsection{Well-definedness}
\label{sec:appendixBwelldefined}

For $\rho$ to be a well-defined endomorphism, we need to check the Leavitt relations~\eqref{eq:leavitt2Na} and~\eqref{eq:leavitt2Nb}. We compute them directly
	\[\rho(S_{g_1}')\rho(S_{g_2}) = U_{g_1}\Big(d^{-2}\sum_{h}\ang{g_1-g_2;h} + d^{-1}N\delta_{g_1,g_2}\Big)U_{-g_2} = \delta_{g_1,g_2}\,,\]
	\begin{align*}
	\rho(S_{g_1}')\rho(T_{g_2}) &= 
	U_{g_1}\Big(\frac{1}{d}\sum_{h}\ang{g_1;h}S_h' + \frac{1}{\sqrt{d}}\sum_{h}\ang{h}^{-1}T_{-h}'T_{h-g_1}'\Big)\rho(T_{g_2})\\
	&= U_{g_1}\Big(\frac{\lambda}{\sqrt{N}d\sqrt{d}}\sum_{h,k}\ang{g_1;h}\ang{h-g_2;k}T_k' + \frac{1}{\sqrt{d}\sqrt{N}}\sum_{h}E(h+g_2)\ang{g_2;-g_1}^{-1}T_{-g_1}'\Big)\\
	&= U_{g_1}\Big(\frac{\lambda\sqrt{N}}{d\sqrt{d}}\ang{g_2;g_1}T_{-g_1}' + \frac{\lambda E(0)}{\sqrt{d}}\ang{g_2;-g_1}^{-1}T_{-g_1}'\Big) = 0\,,
	\end{align*}
	\begin{align*}
	\rho(T_{g_1}')\rho(S_{g_2}) &= 
	\rho(T_{g_1}')\Big(\frac{1}{d}\sum_{h}\ang{g_2;h}^{-1}S_h + \frac{1}{\sqrt{d}}\sum_{h}\ang{h}T_{h-g_2}T_{-h}\Big)U_{-g_2} \\
	&= \Big(\frac{\lambda^{-1}}{d\sqrt{d}\sqrt{N}}\sum_{h,k}\ang{h-g_1;k}^{-1}\ang{g_2;h}^{-1}T_k + \frac{1}{\sqrt{d}\sqrt{N}}\sum_{h}\ang{g_1+h}E(-h-g_1)\ang{g_1;-g_2}T_{-g_2}\Big)U_{-g_2}\\
	&= \Big(\frac{\lambda^{-1}\sqrt{N}}{d\sqrt{d}}\ang{g_2;g_1}^{-1}T_{-g_2} + \frac{\lambda^{-1} E(0)}{\sqrt{d}}\ang{g_2;g_1}^{-1}T_{-g_2}\Big)U_{-g_2} = 0\,,
	\end{align*}
and
	\begin{align*}
	\rho(T_{g_1}')\rho(T_{g_2})
	&= \frac{1}{d}\sum_{k}\ang{g_1-g_2;k}T_kT_k' + \frac{\ang{g_1}}{N\ang{g_2}}\sum_{h,k}\ang{g_1-g_2;k}S_hS_h' \\
	&\qquad+\frac{1}{N}\sum_{h,k}\ang{h+g_1}E(-h-g_1)E(h+g_2)\ang{g_1-g_2;k}T_kT_k'\\
	&=\delta_{g_1,g_2}\Big(\sum_{h}S_hS_h' + \sum_kT_kT_k'\Big) = \delta_{g_1,g_2}\,,
	\end{align*}
where we used~\eqref{eq:qdilogconvolution}. For the completeness relation~\eqref{eq:leavitt2Nb}, we first note that the coefficients of $S_hT_{-r}'T_{r+k}'$ and $T_{-r}T_{r+k}S_h'$ in $\sum_{g}\rho(S_g)\rho(S_g')+\sum_{h}\rho(T_h)\rho(T_h')$ cancel, and the remaining sum then simplifies to
	\begin{align*}
	\sum_g \rho(S_g)\rho(S_g') + \sum_h\rho(T_h)\rho(T_h') =\; 
	&\frac{N}{d^2}(N+d)\sum_hS_hS_h' + \sum_{h,k}T_kS_hS_h'T_k'\\
	&+\sum_{h,k,r}\frac{\ang{h}}{\ang{r}}\Big(\frac1d+\frac1N\sum_{g}E(h+g)E(-r-g)\ang{r+g}\Big)T_{h+k}T_{-h}T_{-r}'T_{r+k}'\\
	=\;& \sum_hS_hS_h' +\sum_{h,k}T_kS_hS_h'T_k'+\sum_{h,k}T_{h+k}T_{-h}T_{-h}'T_{h+k}'\\
	=\;& \sum_hS_hS_h' +\sum_{k}T_k\Big(\sum_hS_hS_h'+\sum_{h}T_hT_h'\Big)T_k' = 1\,,
	\end{align*}
where we again used~\eqref{eq:qdilogconvolution}.

\subsection{The fusion rule for $\rho^2$}
\label{sec:appendixBfusion}
Here we show that $\rho$ satisfies the fusion rule
	\begin{equation} \label{eq:rho2fusion}
	\rho^2(x) = \sum_{g\in G}S_g\alpha_g(x)S_g' + \sum_{h\in G}T_h\rho(x)T_h' =: \Psi(x)\,.
	\end{equation}
Since both sides of~\eqref{eq:rho2fusion} are unital endomorphisms, it is enough to check it for all $x=S_r$ and $x=T_r$ (the identity for $x=S_r'$ and $T_r'$ is then automatic, see~\eqref{eq:completenesstrick}). First, we need the following computation (this is~\cite[Lemma 5.1]{Izu01}).
\begin{lemma}
For all $g\in G$ we have
	\begin{align} 
	\label{eq:rhoUg}
	\rho(U_g)&=\sum_hS_{h-g}S_h'+\sum_h\frac{\ang{h+g}}{\ang{h}}T_{h+g}U_gT_h'\,,\\
	\label{eq:rho2Sg}
	\rho^2(S_g)&=\sum_hS_{h}S_{g+h}S_h' + \sum_hT_{h}\rho(S_g)T_h'\,,\\
	\label{eq:rho2Ug}
	\rho^2(U_g)&=\sum_hS_{h}\alpha_h(U_g)S_h'+\sum_hT_{h}\rho(U_g)T_h'\,,\\
	\label{eq:Sgintertwine}
	\rho^2(x)S_g&=S_g\alpha_g(x)\,,\qquad\qquad\mbox{for all }x\,.
	\end{align}
\end{lemma}
\begin{proof}
	First, we claim that
		\[S_{h_1}'\rho(U_g)T_{h_2} = T_{h_1}'\rho(U_g)S_{h_2} = 0\,.\]
	Indeed, 
		\[\sum_{r\in G}S_{h_1}'\ang{r;g}^{-1}\rho(S_r)\rho(S_r')T_{h_2} = 
		\sum_{r\in G}\frac{\ang{r;g}^{-1}}{d\sqrt{d}}\ang{r;h_1}^{-1}\ang{h_2+r}^{-1}T_{-h_2-r}'\]
	and 
		\begin{align*}
		\sum_r S_{h_1}'\rho(T_{r-g})\rho(T_r')T_{h_2} &= 
		\frac{\lambda}{N\sqrt{d}}
		\sum_r\Big(\sum_{k\in G}\ang{h_1-r+g;k}T_k'\Big)
		\Big(\sum_{h,k\in G}\frac{\ang{r+h}}{\ang{h}}E(-h-r)\ang{r;k}T_kT_{-h}'T_{h+k}'\Big)
		T_{h_2}\\
		&=\frac{\lambda}{N\sqrt{d}}\sum_{r,k\in G}\frac{\ang{r+h_2-k}}{\ang{h_2-k}}E(k-h_2-r)\ang{h_1+g;k}T_{k-h_2}' \\
		&= -\frac{1}{d\sqrt{d}}\sum_{k}\frac{\ang{h_1+g;k}}{\ang{h_2-k}}T_{k-h_2}'\,.
		\end{align*}
	Summing up these two identities proves the first equality $S_{h_1}'\rho(U_g)T_{h_2}=0$, and the second equality is analogous. Next, we calculate
	\[\sum_{r\in G}S_{h_1}'\ang{r;g}^{-1}\rho(S_r)\rho(S_r')S_{h_2} = \frac{1}{d^2}\sum_{r\in G}\ang{r;h_2-h_1-g} = \frac{N}{d^2}\delta_{h_1+g,h_2}\]
	and
	\[\sum_{r\in G}S_{h_1}'\rho(T_{r-g})\rho(T_r')S_{h_2} = \frac{1}{dN}\sum_{r,k\in G}\ang{h_1-(r-g);k}\ang{h_2-r;k}^{-1} = \frac{N}{d}\delta_{h_1+g,h_2}\,,\]
	so $S_{h_1}'\rho(U_g)S_{h_2} = \delta_{h_1+g,h_2}$. Finally, we have
	\begin{equation} \label{eq:intermediate}
	\sum_{r\in G}T_{h_1}'\ang{r;g}^{-1}\rho(S_r)\rho(S_r')T_{h_2} = \frac{1}{d}\sum_{r\in G}\frac{\ang{h_1+r}}{\ang{h_2+r}}\ang{g;-r}T_{-r-h_1}T_{-r-h_2}'
	\end{equation}
	and
	\begin{align*}
	\sum_{r\in G}&T_{h_1}'\rho(T_{r-g})\rho(T_r')T_{h_2} = \frac{1}{N}\sum_{r\in G}\frac{\ang{r}}{\ang{r-g}}\Big(\sum_{h\in G}\ang{g;h_1}\ang{r+h;h_2-h_1}S_hS_h'\Big)\\
	 &+\frac{1}{N}\sum_{r\in G}\Big(\sum_{k\in G}\ang{h_1-k}E(h_1-k+r-g)\ang{g;k}\frac{\ang{r+h_2-k}}{\ang{h_2-k}}E(k-h_2-r)T_{k-h_1}T_{k-h_2}'\Big)\\
	 &=\frac{1}{N}\sum_{r\in G}\frac{\ang{r;g}}{\ang{g}}\Big(\sum_{h\in G}\ang{g;h_1}\ang{r+h;h_2-h_1}S_hS_h'\Big)
	   +\sum_{k\in G}\frac{\ang{h_1-k}}{\ang{h_2-k}}(\delta_{h_1,h_2+g}-1/d)\ang{g;k}T_{k-h_1}T_{k-h_2}'\,.
	\end{align*}
	The $-1/d$ part in the second sum cancels with the one in~\eqref{eq:intermediate}, and we are left with
	\begin{align*}
	T_{h_1}'\rho(U_g)T_{h_2} &= \delta_{h_1,h_2+g}\Big(\sum_{h\in G}\frac{\ang{g;h_1}}{\ang{g}}\ang{h;g}^{-1}S_hS_h'+\sum_{k\in G}\frac{\ang{h_1-k}}{\ang{h_2-k}}\ang{g;k}T_{k-h_1}T_{k-h_2}'\Big)\\
	&=\delta_{h_1,h_2+g}\Big(\sum_{h\in G}\frac{\ang{h_1}}{\ang{h_2}}\ang{h;g}^{-1}S_hS_h'+\sum_{k\in G}\frac{\ang{h_1}}{\ang{h_2}}T_{k-h_1}T_{k-h_2}'\Big)\,.
	\end{align*}
	Together with~\eqref{eq:leavitt2Nb} this proves~\eqref{eq:rhoUg}.
	
	Next, we check that
	\begin{equation} \label{eq:s0}
		\rho^2(S_0) = \sum_{g\in G}S_g^2S_g' + \sum_{h\in G}T_h\rho(S_0)T_h'\,.
	\end{equation}
	Indeed, we have
	\begin{align*}
		S_0'\rho^2(S_0) &= \frac{1}{d}\sum_{h}S_0'\rho(S_h) + \frac{\lambda}{d\sqrt{N}}\sum_{h}\ang{h}\Big(\sum_{k\in G}\ang{-h;k}T_k'\Big)\rho(T_{-h})\\
		&=\frac{1}{d^2}\sum_{h}U_{-h} + \frac{N}{d}S_{0}S_{0}' + \frac{\lambda}{dN}\sum_{h,k}\ang{h}\ang{-h;k}\sum_{h'\in G}\ang{h'}E(h'-h)\ang{h;k-h'}T_{-h'}T_{k-h'}'\\
		&=\frac{1}{d^2}\sum_{h}U_{-h} + \frac{N}{d}S_{0}S_{0}' + \frac{\lambda}{dN}\sum_{h,k,h'}\ang{h'-h}E(h'-h)T_{-h'}T_{k-h'}'\\
		&=\frac{N}{d^2}S_0S_0'+\frac{1}{d^2}\sum_{k,h'}T_{-h'}T_{k-h'}' + \frac{N}{d}S_{0}S_{0}' -\frac{1}{d^2}\sum_{k,h'}T_{-h'}T_{k-h'}' = S_0S_0'\,.
	\end{align*}
	Applying $\alpha_g$ then gives $S_g'\rho^2(S_0) = S_gS_g'$. Similarly, 
	\begin{align*}
		T_0'\rho^2(S_0) &= \frac{1}{d\sqrt{d}}\sum_{h}\ang{h}T_hU_h + \frac{1}{\lambda \sqrt{d}\sqrt{N}}\sum_{h,r}S_rS_r'\rho(T_{-h})
		+ \frac{1}{\sqrt{d}\sqrt{N}}\sum_{h,k}\ang{h-k}E(h-k)T_kT_k'\rho(T_{-h})\\
		&=\frac{1}{d\sqrt{d}}\sum_{h}\ang{h}T_hU_h + \frac{1}{d}\sum_r S_rT_0'
		+ \frac{1}{\sqrt{d}\sqrt{N}}\sum_{h,k}\ang{h-k}E(h-k)T_kT_k'\rho(T_{-h})\,.
	\end{align*}
	From~\eqref{eq:qdilogconvolution} we can compute that
	\[\frac{1}{\sqrt{N}}\sum_{h}\ang{h-k}E(h-k)T_k'\rho(T_{-h}) = \ang{k}\Big(T_{-k}T_0'-\frac{U_k}{d}\Big)\,,\]
	so
	\[T_0'\rho^2(S_0) = \frac{1}{d}\sum_r S_rT_0' +\frac{1}{\sqrt{d}}\sum_{k}\ang{k}T_kT_{-k}T_0' = \rho(S_0)T_0'\,.\]
	Since $T_g'\rho^2(S_0) = T_g'U_{-g}\rho^2(S_0)U_g = T_0'\rho^2(S_0)U_g=\rho(S_0)T_0'U_g=\rho(S_0)T_g'$, we get~\eqref{eq:s0}.
	This now implies~\eqref{eq:rho2Sg} for all $g$, since
		\[\rho^2(S_g)=\rho(U_g)\rho^2(S_0)\rho(U_{-g})\,,\qquad \Psi(S_g)=\rho(U_g)\Psi(S_0)\rho(U_{-g})\,.\]
	
	Next, to get~\eqref{eq:rho2Ug}, note that by~\eqref{eq:rhoUg}, we have
		\[\rho^2(U_g)\rho(S_r)=\rho(S_{r-g})\,,\qquad \rho^2(U_g)\rho(T_r)=\frac{\ang{r+g}}{\ang{r}}\rho(T_{r+g})\rho(U_g)\,.\]
	Direct computation gives also
		\[\Psi(U_g)\rho(S_r)=\rho(S_{r-g})\,,\qquad \Psi(U_g)\rho(T_r)=\frac{\ang{r+g}}{\ang{r}}\rho(T_{r+g})\rho(U_g)\,,\]
	and since $1=\rho(\sum_rS_rS_r'+\sum_rT_rT_r')$, this proves~\eqref{eq:rho2Ug}.
	
	Finally, to prove~\eqref{eq:Sgintertwine}, because of $\alpha_g\circ\rho=\rho$ and~\eqref{eq:rho2Sg} (together with the completeness trick~\eqref{eq:completenesstrick}) it is enough to show that $\rho^2(T_g)S_0=S_0T_g$ for all~$g$, and by~\eqref{eq:rho2Ug}, it suffices to check it for $g=0$. Let us first check that
		\[\rho(S_0')\big(\rho^2(T_0)S_0-S_0T_0\big) = 0\,,\qquad 
		\rho(T_0')\big(\rho^2(T_0)S_0-S_0T_0\big) = 0\,.\]
	For the first identity we write
	\[\rho(S_0')\rho^2(T_0)S_0=\frac{\lambda}{\sqrt{d}\sqrt{N}}\sum_h\rho(T_h')S_0
		=\frac1{dN}\sum_{h,k}\ang{h;k}T_k=\frac{T_0}{d}=\rho(S_0')S_0T_0.\]
	For the second identity, first note that
	\[T_0'\rho(T_0)=\frac{\lambda^{-1}}{\sqrt{N}}\sum_h S_hS_h'+\frac{1}{\sqrt{N}}\sum_k \ang{k}E(-k)T_kT_k'\,.\]
	Applying \(\rho\) and multiplying by \(S_0\) gives
	\[\rho(T_0'\rho(T_0))S_0
		=\frac{1}{\lambda d\sqrt{N}}\sum_hU_h\rho(S_0)
		+\frac{1}{\lambda N\sqrt{d}}\sum_{k,h}\ang{k}E(-k)\ang{k;h}\rho(T_k)T_h\,.\]
	Since
	\[\rho(T_k)T_h=\frac{\lambda}{\sqrt{d}\sqrt{N}}\sum_g\ang{g-k;h}S_g
	+\frac{1}{\sqrt{N}}\sum_g \ang{g}E(g+k)\ang{k;h}^{-1}T_{g+h}T_{-g}\,,\]
	we can rewrite $\rho(T_0'\rho(T_0))S_0$ as
	\[\frac{\sqrt N}{\lambda d^2}S_0 +\frac1{\lambda d\sqrt{d}\sqrt{N}}\sum_{h,g}\ang{g}T_{g-h}T_{-g} +\frac{E(0)}{\lambda d}S_0 +\frac{1}{\lambda\sqrt{d}\sqrt{N}}\sum_hT_hT_0
	-\frac{1}{\lambda d\sqrt{d}\sqrt{N}}\sum_{g,h}\ang{g}T_{g+h}T_{-g}.\]
	(Here again we used Fourier transform and~\eqref{eq:qdilogconvolution}.)
	The two double sums cancel, and the coefficient of~$S_0$ vanishes because $E(0)=-\sqrt{N}/d$. Hence
	\[\rho(T_0')\rho^2(T_0)S_0=\frac1{\lambda\sqrt{d}\sqrt N}\sum_hT_hT_0=\rho(T_0')S_0T_0\,.\]
	From the above established identities we get
		\[U_g\big(\rho^2(T_0)S_0-S_0T_0\big) = \rho^2(T_0)S_0-S_0T_0\,,
		\qquad \rho(U_g)\big(\rho^2(T_0)S_0-S_0T_0\big) = \alpha_{-g}\big(\rho^2(T_0)S_0-S_0T_0\big)\,,\]
	and so we get also
	 \[\rho(S_g')\big(\rho^2(T_0)S_0-S_0T_0\big) = U_g\rho(S_0')U_{-g}\big(\rho^2(T_0)S_0-S_0T_0\big) = U_g\rho(S_0')\big(\rho^2(T_0)S_0-S_0T_0\big) = 0\,,\]
	and, similarly, $\rho(T_g')(\rho^2(T_0)S_0-S_0T_0)=0$.
	This proves $\rho^2(T_0)S_0=S_0T_0$, and, by preceding remarks, also the full identity~\eqref{eq:Sgintertwine}.
\end{proof}

It remains to prove~\eqref{eq:rho2fusion} for $x=T_g$. The key calculation is the following~(see \cite[Lemma 5.2, Theorem 5.3]{Izu01}).
\begin{lemma} \label{lem:izumikey}
Denote $\hat{T}_{g}=\frac{1}{\sqrt{N}}\sum_{h}\ang{g;h}T_h$. Then
	\begin{equation} \label{eq:izumikey}
	\rho(T_0'\rho(T_0))\hat{T}_0 = \rho(T_0')\hat{T}_0\rho(T_0)\,.
	\end{equation}
\end{lemma}
\begin{proof}
	Note that $U_g\hat{T}_0=\hat{T}_0$ and
	\begin{equation} \label{eq:rhotprimethat}
	\rho(T_h')\hat{T}_0 = \ang{h}\Big(\lambda S_{-h}S_{-h}' + \frac{1}{\sqrt{N}}\sum_{k}E(k-h)\ang{h;k}^{-1}\hat{T}_hT_k'\Big)\,.
	\end{equation}
	
	We expand both sides. For the left hand side we compute
	\begin{align*}
	\rho(T_0'\rho(T_0))\hat{T}_0 &= \frac{\lambda^{-1}}{\sqrt{N}}\sum_{h}\rho(S_hS_h')\hat{T}_0 + \frac{1}{\sqrt{N}}\sum_{h}\ang{h}E(-h)\rho(T_hT_h')\hat{T}_0\\
		  &= \frac{\lambda^{-1}}{\sqrt{N}}\sum_{h}U_h\rho(S_0S_0')U_{-h}\hat{T}_0 + \frac{1}{\sqrt{N}}\sum_{h}\ang{h}E(-h)\rho(T_hT_h')\hat{T}_0\\
		  &= \frac{\lambda^{-1}}{\sqrt{N}}\sum_{h}U_h\rho(S_0S_0')\hat{T}_0 + \frac{\lambda}{\sqrt{N}} \sum_{h}\ang{h}^2E(-h)\rho(T_h)S_{-h}S_{-h}' \\
		  &\qquad+ \frac{1}{N}\sum_{h,k}\ang{h}^2E(-h)E(k-h)\ang{h;k}^{-1}\rho(T_h)\hat{T}_hT_k'\,.
	\end{align*}
	Since 
	\begin{equation} \label{eq:rhoS0primeThat}
	\lambda \ang{h}\rho(T_h)S_{-h}S_{-h}' = \hat{T}_0S_{-h}S_{-h}'\,,\qquad
	\rho(S_0')\hat{T}_{0} = \frac{1}{\sqrt{N}\sqrt{d}}\sum_{r}\frac{T_r'}{\ang{r}}\,,
	\end{equation}
	and
	\[\sum_{h}U_h\rho(S_0) = \frac{N}{d}S_0+\frac{\sqrt{N}}{\sqrt{d}}\hat{T}_0\sum_{r}\ang{r}T_r\,,\]
	we can further rewrite the last expression as
	\begin{align*}
	\frac{\lambda^{-1}}{d\sqrt{d}}&\sum_k \ang{k}^{-1}S_0T_k' + \frac{\lambda^{-1}}{d\sqrt{N}}\sum_{r,k}\frac{\ang{r}}{\ang{k}}\hat{T}_0T_rT_k' + \frac{1}{\sqrt{N}} \sum_{h}\ang{h}E(-h)\hat{T}_0S_{-h}S_{-h}' \\
	&\quad+ \frac{1}{N}\sum_{h,k}\ang{h}^2E(-h)E(k-h)\ang{h;k}^{-1}\rho(T_h)\hat{T}_hT_k' \\
	=\frac{\lambda^{-1}}{d\sqrt{d}}&\sum_k \ang{k}^{-1}S_0T_k' + \frac{\lambda^{-1}}{d\sqrt{N}}\sum_{r,k}\frac{\ang{r}}{\ang{k}}\hat{T}_0T_rT_k' + \frac{1}{\sqrt{N}} \sum_{h}\ang{h}E(-h)\hat{T}_0S_{-h}S_{-h}' \\
	&\quad +\frac{\lambda}{N\sqrt{d}}\sum_{h,k}\frac{\ang{h}^2}{\ang{h;k}}E(-h)E(k-h)S_0T_k'+\frac{1}{N\sqrt{N}}\sum_{r,h,k}\frac{\ang{h}^2}{\ang{h;k}}E(-h)E(k-h)\ang{r}E(h-r)\hat{T}_0T_rT_k'\,,
	\end{align*}
	where we have used
	\[\rho(T_h)\hat{T}_h = \frac{\lambda}{\sqrt{d}}S_0 + \frac{1}{\sqrt{N}}\sum_r\ang{r}E(h-r)\hat{T}_0T_r\,.\]
	
	For the right hand side we calculate
	\begin{align*}
	\rho(T_0')\hat{T}_0\rho(T_0) = \Big(&\lambda S_{0}S_{0}' + \frac{1}{\sqrt{N}}\sum_{k}E(k)\hat{T}_0T_k'\Big)\rho(T_0)
	 = \frac{\lambda^2}{\sqrt{d}\sqrt{N}}\sum_{k}S_0T_k' + \frac{\lambda^{-1}}{N}\sum_{h,k}\frac{E(k)}{\ang{h;k}}\hat{T}_0S_hS_h'\\
	 &+\frac{1}{N}\sum_{r,k}E(k-r)\ang{-r}E(-r)\hat{T}_0T_{r}T_{k}'\,.
	\end{align*}
	Let us now compare the coefficients on both sides. For $\hat{T}_0S_hS_h'$ the identity we need is
	\[\lambda \ang{h}E(h) = \frac{1}{\sqrt{N}}\sum_{k}E(k)\ang{h;-k}\,,\]
	which is exactly the Fourier transform of~$E$. For $S_0T_k'$ we need
	\[\frac{\lambda}{N}\sum_h\ang{h}^2\ang{h;-k}E(-h)E(k-h) = \frac{\lambda^2}{\sqrt{N}} - \frac{\lambda^{-1}}{d\ang{k}}\,.\]
	This also holds as can be seen using the Fourier transform for~$b$ again
	\begin{align*}
	\frac{\lambda}{N}&\sum_h\ang{h}^2\ang{h;-k}E(-h)E(k-h) = \frac{\lambda}{N\ang{k}}\sum_h \ang{-h}E(-h)\ang{k-h}E(k-h)\\
	&=\frac{\lambda^{-1}}{N^2\ang{k}}\sum_{r,s,h} E(r)E(s)\ang{r;h}\ang{s;h-k}
	  = \frac{\lambda^{-1}}{N\ang{k}}\sum_{r} E(r)E(-r)\ang{r;k}\\
	 &= \frac{\lambda^{-1}}{N\ang{k}}\sum_{r} \big(\ang{r}^{-1}+(N/d^2-1)\delta(r)\big)\ang{r;k}
	  =\frac{\lambda^2}{\sqrt{N}}-\frac{\lambda^{-1}}{d\ang{k}}\,,
	\end{align*}
	where we used $N/d^2-1=-N/d$ and the Gaussian Fourier transform $\frac{1}{\sqrt{N}}\sum_{r}\ang{r}^{-1}\ang{r;k}=\lambda^3\ang{k}$. Finally, comparing the coefficients of $\hat{T}_0T_rT_k'$, we need
	\[-\lambda^{-1}E(0)\frac{\ang{r}}{\ang{k}}+\frac{\ang{r}}{\sqrt{N}}\sum_{h}\ang{h}^2E(-h)E(k-h)E(h-r)\ang{h;k}^{-1} = E(k-r)\ang{-r}E(-r)\]
	or equivalently
	\[-\lambda^{-1}E(0)+\frac{1}{\sqrt{N}}\sum_{h}\ang{y-h}E(y-h)E(h)\ang{x-h}E(x-h) = \ang{x-y}E(x)E(y)\,,\]
	which is precisely~\eqref{eq:pentagonintegraldual}.
\end{proof}
With~\eqref{eq:izumikey} we can now finish the proof of~\eqref{eq:rho2fusion}. For this we want to show that
	\begin{equation} \label{eq:hatT0intertwiner}
	\rho^2(T_0)\hat{T}_0 = \hat{T}_0\rho(T_0)
	\end{equation}
We already have 
	\[\rho(T_0')\big(\rho^2(T_0)\hat{T}_0 - \hat{T}_0\rho(T_0)\big)=0\,,\]
let us show that 
	\[\rho(S_0')\big(\rho^2(T_0)\hat{T}_0 - \hat{T}_0\rho(T_0)\big)=0\,.\]
Indeed, using~\eqref{eq:rhotprimethat} we have
	\begin{align*}
	\rho(S_0'\rho(T_0))\hat{T}_0 = \frac{\lambda}{\sqrt{d}\sqrt{N}}\sum_{k}\rho(T_k')\hat{T}_0 = 
	\frac{\lambda^2}{\sqrt{d}\sqrt{N}}\sum_{k}\ang{k}S_{-k}S_{-k}'+\frac{\lambda}{N\sqrt{d}}\sum_{h,k}\ang{k}E(h-k)\ang{k;h}^{-1}\hat{T}_{k}T_{h}'
	\end{align*}
and using~\eqref{eq:rhoS0primeThat}, we get that
	\[
	\rho(S_0')\hat{T}_0\rho(T_0) = \frac{1}{\sqrt{N}\sqrt{d}}\sum_{r}\frac{T_r'}{\ang{r}}\rho(T_0)
	 = \frac{1}{N\sqrt{d}}\sum_{r}\ang{r}^{-1}\Big(\lambda^{-1}\sum_{h\in G}\ang{h;r}^{-1}S_hS_h'
	 +\sum_{k\in G}\ang{r-k}E(r-k)T_{k-r}T_k'\Big)\,,
	\]
and we get equality using Gaussian Fourier transform and Fourier transform of~$E$ again. Note that
	\[U_g\big(\rho^2(T_0)\hat{T}_0 - \hat{T}_0\rho(T_0)\big) = \big(\rho^2(T_0)\hat{T}_0 - \hat{T}_0\rho(T_0)\big)\,,\]
and
	\[\rho(U_g)\big(\rho^2(T_0)\hat{T}_0 - \hat{T}_0\rho(T_0)\big) = \ang{g}^{-1}\alpha_{-g}\big(\rho^2(T_0)\hat{T}_0 - \hat{T}_0\rho(T_0)\big)U_g\,,\]
so using equivariance we get
	\[\rho(T_g')\big(\rho^2(T_0)\hat{T}_0 - \hat{T}_0\rho(T_0)\big)=
	\rho(S_g')\big(\rho^2(T_0)\hat{T}_0 - \hat{T}_0\rho(T_0)\big)=0\]
for all $g$. Then using completeness relation~\eqref{eq:leavitt2Nb} (after applying $\rho$ to it) gives~\eqref{eq:hatT0intertwiner}. Applying $\alpha_{-g}$ to both sides then gives
	\[\rho^2(T_0)\hat{T}_g = \hat{T}_g\rho(T_0)\,,\]
and since $\hat{T}_g$, $g\in G$ span the same $\CC$-vector space as $T_g$, $g\in G$, we get 
	\[\rho^2(T_0)T_g = T_g\rho(T_0)\,.\]
Finally, since (by~\eqref{eq:rho2Ug})
	\[T_h=U_{-h}T_0\,,\qquad\qquad \rho^2(U_{-h})T_g=T_g\rho(U_{-h})\,,\]
we get
	\[\rho^2(T_h)T_g = \rho^2(U_{-h})\rho^2(T_0)T_g = \rho^2(U_{-h})T_g\rho(T_0) = T_g\rho(U_{-h})\rho(T_0) = T_g\rho(T_h)\,.\]
Together with~\eqref{eq:rho2Sg} and~\eqref{eq:Sgintertwine}, this proves~\eqref{eq:rho2fusion}.

\bibliographystyle{abbrv}
\bibliography{biblio}

\begin{thebibliography}{10}

\bibitem{AK14b}
J.~E. Andersen and R.~Kashaev.
\newblock Complex quantum {C}hern--{S}imons, 2014.
\newblock Preprint, \texttt{arXiv:1409.1208}.

\bibitem{AK14a}
J.~E. Andersen and R.~Kashaev.
\newblock A {TQFT} from quantum {T}eichm{\"u}ller theory.
\newblock {\em Comm. Math. Phys.}, 330(3):887--934, 2014.
\newblock \texttt{arXiv:1109.6295}.

\bibitem{AFMY17}
M.~Appleby, S.~Flammia, G.~McConnell, and J.~Yard.
\newblock {SIC}s and algebraic number theory.
\newblock {\em Foundations of Physics}, 47(8):1042--1059, 2017.

\bibitem{AFK25}
M.~Appleby, S.~T. Flammia, and G.~S. Kopp.
\newblock A constructive approach to {Z}auner's conjecture via the {S}tark
  conjectures, 2025.
\newblock Preprint, \texttt{arXiv:2501.03970}.

\bibitem{Bar04}
E.~W. Barnes.
\newblock On the theory of the multiple gamma function.
\newblock {\em Trans. Cambridge Philos. Soc.}, 19:374--425, 1904.

\bibitem{DV21}
H.~Darmon and J.~Vonk.
\newblock Singular moduli for real quadratic fields: a rigid analytic approach.
\newblock {\em Duke Math. J.}, 170(1):23--93, 2021.

\bibitem{DK24}
S.~Dasgupta and M.~Kakde.
\newblock {B}rumer--{S}tark units and explicit class field theory.
\newblock {\em Duke Math. J.}, 173(8):1477--1555, 2024.

\bibitem{Dim15}
T.~Dimofte.
\newblock Complex {C}hern--{S}imons theory at level $k$ via the 3d--3d
  correspondence.
\newblock {\em Comm. Math. Phys.}, 339:619--662, 2015.
\newblock \texttt{arXiv:1409.0857}.

\bibitem{EGNO17}
P.~Etingof, S.~Gelaki, D.~Nikshych, and V.~Ostrik.
\newblock {\em Tensor Categories}.
\newblock Mathematical surveys and monographs. American Mathematical Society,
  2017.

\bibitem{ENO05}
P.~Etingof, D.~Nikshych, and V.~Ostrik.
\newblock On fusion categories.
\newblock {\em Annals of Mathematics}, pages 581--642, 2005.

\bibitem{EG17}
D.~E. Evans and T.~Gannon.
\newblock Non-unitary fusion categories and their doubles via endomorphisms.
\newblock {\em Advances in Mathematics}, 310:1--43, 2017.

\bibitem{Fad95}
L.~D. Faddeev.
\newblock Discrete {H}eisenberg--{W}eyl group and modular group.
\newblock {\em Lett. Math. Phys.}, 34(3):249--254, 1995.
\newblock \texttt{arXiv:hep-th/9504111}.

\bibitem{FK94}
L.~D. Faddeev and R.~M. Kashaev.
\newblock Quantum dilogarithm.
\newblock {\em Modern Phys. Lett. A}, 9(5):427--434, 1994.
\newblock \texttt{arXiv:hep-th/9310070}.

\bibitem{Gannon26}
T.~Gannon, A.~Schopieray, and H.~Yadav.
\newblock On {H}aagerup-{I}zumi fusion categories, 2026.
\newblock Preprint, \texttt{arXiv:2609.25185}.

\bibitem{GK17}
S.~Garoufalidis and R.~Kashaev.
\newblock From state integrals to $q$-series.
\newblock {\em Math. Res. Lett.}, 24(3):781--801, 2017.

\bibitem{GKZ}
S.~Garoufalidis, R.~Kashaev, and D.~Zagier.
\newblock An $\mathrm{SL}_2(\mathbb{Z})$-extension of {F}addeev's quantum
  dilogarithm, 2018.

\bibitem{GZ23}
S.~Garoufalidis and D.~Zagier.
\newblock Knots and their related $q$-series.
\newblock {\em SIGMA Symmetry Integrability Geom. Methods Appl.}, 19:Paper No.
  082, 39 pp., 2023.
\newblock \texttt{arXiv:2304.09377}.

\bibitem{GZ24}
S.~Garoufalidis and D.~Zagier.
\newblock Knots, perturbative series and quantum modularity.
\newblock {\em SIGMA Symmetry Integrability Geom. Methods Appl.}, 20:Paper No.
  055, 87 pp., 2024.
\newblock \texttt{arXiv:2111.06645}.

\bibitem{Gon08}
A.~B. Goncharov.
\newblock Pentagon relation for the quantum dilogarithm and quantized
  {$M_{0,5}^{cyc}$}.
\newblock In M.~Kapranov, Y.~I. Manin, P.~Moree, S.~Kolyada, and L.~Potyagailo,
  editors, {\em Geometry and Dynamics of Groups and Spaces: In {M}emory of
  {A}lexander {R}eznikov}, pages 415--428. Birkh{\"a}user Basel, Basel, 2008.

\bibitem{Hil00}
D.~Hilbert.
\newblock {M}athematische {P}robleme.
\newblock {\em Nachrichten von der K{\"o}nigl. Gesellschaft der Wissenschaften
  zu G{\"o}ttingen, Math.-Phys. Klasse}, pages 253--297, 1900.
\newblock English transl.: Bull. Amer. Math. Soc. \textbf{8} (1902), 437--479.

\bibitem{Huang26}
T.-C. Huang.
\newblock Cyclic {H}aagerup–{I}zumi fusion categories at every odd order,
  2026.
\newblock Preprint, \texttt{arXiv:2609.15986}.

\bibitem{Izu93}
M.~Izumi.
\newblock Subalgebras of infinite {C}*-algebras with finite {W}atatani indices
  {I}. {C}untz algebras.
\newblock {\em Communications in Mathematical Physics}, 155(1):157--182, 1993.

\bibitem{Izu00}
M.~Izumi.
\newblock The structure of sectors associated with {L}ongo--{R}ehren inclusions
  {I}. general theory.
\newblock {\em Communications in Mathematical Physics}, 213(1):127--179, 2000.

\bibitem{Izu01}
M.~Izumi.
\newblock The structure of sectors associated with {L}ongo--{R}ehren inclusions
  {II}: examples.
\newblock {\em Reviews in Mathematical Physics}, 13(05):603--674, 2001.

\bibitem{Kop21}
G.~S. Kopp.
\newblock {SIC}-{POVM}s and the {S}tark conjectures.
\newblock {\em International Mathematics Research Notices},
  2021(18):13812--13838, 09 2021.

\bibitem{Kop24}
G.~S. Kopp.
\newblock The {S}hintani--{F}addeev modular cocycle: {S}tark units from
  $q$-{P}ochhammer ratios, 2024.
\newblock Preprint, \texttt{arXiv:2411.06763}.

\bibitem{KK03}
N.~Kurokawa and S.-y. Koyama.
\newblock Multiple sine functions.
\newblock {\em Forum Math.}, 15:839--876, 2003.

\bibitem{LM33}
C.~G. Latimer and C.~C. MacDuffee.
\newblock A correspondence between classes of paired matrices and classes of
  ideals.
\newblock {\em Annals of Mathematics}, 34(2):313--316, 1933.

\bibitem{Lea65}
W.~G. Leavitt.
\newblock The module type of homomorphic images.
\newblock {\em Duke Math. J.}, 32(1):305--311, 1965.

\bibitem{Lub99}
D.~Lubinsky.
\newblock The size of $(q;q)_n$ for $q$ on the unit circle.
\newblock {\em Journal of Number Theory}, 76(2):217--247, 1999.

\bibitem{Neu81}
W.~D. Neumann.
\newblock A calculus for plumbing applied to the topology of complex surface
  singularities and degenerating complex curves.
\newblock {\em Transactions of the American Mathematical Society},
  268(2):299--344, 1981.

\bibitem{ST61}
G.~Shimura and Y.~Taniyama.
\newblock {\em Complex Multiplication of Abelian Varieties and Its Applications
  to Number Theory}, volume~6 of {\em Publications of the Mathematical Society
  of Japan}.
\newblock The Mathematical Society of Japan, Tokyo, 1961.

\bibitem{Shi77}
T.~Shintani.
\newblock On a {K}ronecker limit formula for real quadratic fields.
\newblock {\em J. Fac. Sci. Univ. Tokyo Sect. IA Math.}, 24(1):167--199, 1977.

\bibitem{Shi78}
T.~Shintani.
\newblock On certain ray class invariants of real quadratic fields.
\newblock {\em J. Math. Soc. Japan}, 30(1):139--167, 1978.

\bibitem{Sta76}
H.~M. Stark.
\newblock {$L$}-functions at $s=1$. {III}. {T}otally real fields and
  {H}ilbert's twelfth problem.
\newblock {\em Advances in Math.}, 22(1):64--84, 1976.

\bibitem{Sta80}
H.~M. Stark.
\newblock {$L$}-functions at $s=1$. {IV}. {F}irst derivatives at $s=0$.
\newblock {\em Advances in Math.}, 35(3):197--235, 1980.

\bibitem{Tan07}
B.~A. Tangedal.
\newblock Continued fractions, special values of the double sine function, and
  {S}tark units over real quadratic fields.
\newblock {\em J. Number Theory}, 124:291--313, 2007.

\bibitem{Tao05}
T.~Tao.
\newblock An uncertainty principle for cyclic groups of prime order.
\newblock {\em Mathematical Research Letters}, 12(1):121--127, 2005.

\bibitem{Tate84}
J.~Tate.
\newblock {\em Les {C}onjectures de {S}tark sur les {F}onctions {$L$} d'{A}rtin
  en $s=0$}, volume~47 of {\em Progr. Math.}
\newblock Birkh{\"a}user, Boston, MA, 1984.
\newblock Notes d'un cours {\`a} {O}rsay r{\'e}dig{\'e}es par D.~Bernardi et
  N.~Schappacher.

\bibitem{Yal26}
B.~Yalkinoglu.
\newblock A note on {S}hintani’s invariant.
\newblock {\em Functiones et Approximatio Commentarii Mathematici}, page
  1–15, Jan. 2026.

\bibitem{Yam08}
S.~Yamamoto.
\newblock On {K}ronecker limit formulas for real quadratic fields.
\newblock {\em J. Number Theory}, 128(2):426--450, 2008.

\bibitem{Yam10b}
S.~Yamamoto.
\newblock Factorization of {S}hintani's ray class invariant for totally real
  fields.
\newblock {\em RIMS K{\^o}ky{\^u}roku Bessatsu}, B19:249--254, 2010.

\bibitem{Yam10a}
S.~Yamamoto.
\newblock On {S}hintani's ray class invariant for totally real number fields.
\newblock {\em Math. Ann.}, 346:449--476, 2010.

\bibitem{Zag07}
D.~Zagier.
\newblock The dilogarithm function.
\newblock In P.~Cartier, P.~Moussa, B.~Julia, and P.~Vanhove, editors, {\em
  Frontiers in {N}umber {T}heory, {P}hysics, and {G}eometry {II}: On
  {C}onformal {F}ield {T}heories, {D}iscrete {G}roups and {R}enormalization},
  pages 3--65. Springer Berlin Heidelberg, Berlin, Heidelberg, 2007.

\bibitem{Zau11}
G.~Zauner.
\newblock Quantum designs: Foundations of a noncommutative design theory.
\newblock {\em International Journal of Quantum Information}, 09(01):445–507,
  Feb. 2011.

\end{thebibliography}

\end{document}